\documentclass[10pt,reqno]{amsart}
\usepackage{amsaddr}
\usepackage[a4paper, top=1.5cm, bottom=1.5cm, left=2cm, right=2cm, heightrounded, marginparwidth=3.5cm, marginparsep=.2cm, centering]{geometry}
\usepackage[utf8]{inputenc}
\usepackage[english]{babel}
\usepackage{amsmath}
\usepackage{mathrsfs}
\usepackage{amsfonts}
\usepackage{amsthm}
\usepackage{amssymb}
\usepackage[shortlabels]{enumitem}
\usepackage{tikz-cd}
\usepackage{fullpage}
\usepackage{pdfpages}
\usepackage[colorlinks=true,hyperindex,pagebackref,linktocpage=true]{hyperref}
\usepackage[nameinlink]{cleveref}
\hypersetup{
	colorlinks,
	linkcolor={red!40!black},
	citecolor={cyan!70!black},
	urlcolor={orange!60!black}
}
\newtheorem{thm}{Theorem}[section]
\newtheorem{thm2}{Theorem}

\newtheorem{conv}{Convention}

\newtheorem{pro}[thm]{Proposition}
\newtheorem{prop}[thm]{Proposition}
\newtheorem{ques}[thm]{Question}
\newtheorem{cor}[thm]{Corollary}
\newtheorem{lem}[thm]{Lemma}
\theoremstyle{definition}
\newtheorem{defi2}[thm2]{Definition}
\newtheorem{defi}[thm]{Definition}
\newtheorem{rem}[thm]{Remark}
\newtheorem{exa}[thm]{Example}

\newcommand{\bb}[1]{\mathbb{#1}}
\newcommand{\cali}[1]{{\mathcal{#1}}}
\newcommand{\mrm}[1]{{\mathrm{#1}}}

\newcommand{\Ki}{{\mathcal{K}}}

\newcommand{\C}{{\mrm{Spa}({C},{O_C})}}
\newcommand{\Cp}{{\mrm{Spa}({C'},{O_{C'}})}}
\newcommand{\Spac}[1]{\mrm{Spa}({#1},{#1}^\circ)}
\newcommand{\Spa}[1]{\mrm{Spa}({#1},#1^+)}
\newcommand{\Spf}[1]{\mrm{Spa}({#1})}
\newcommand{\Spd}[1]{{\mrm{Spd}({#1},#1^+)}}
\newcommand{\Spo}[1]{{\mrm{Spo}({#1},#1^+)}}
\newcommand{\Spor}[1]{{\mrm{Spo}({#1},{#1})}}

\newcommand{\Spof}[1]{{\mrm{Spo}({#1})}}
\newcommand{\Spdf}[1]{{\mrm{Spd}({#1})}}
\newcommand{\Huf}[1]{({#1},{#1})}
\newcommand{\Hub}[1]{({#1},#1^+)}
\newcommand{\Heuer}[1]{#1^{\diamond/\circ}}

\newcommand{\topPerf}{{\widetilde{\mrm{Perf}}}}
\newcommand{\topSch}{{\widetilde{\mrm{SchPerf}}}}
\newcommand{\red}{^{\mrm{red}}}

\newcommand{\Zp}{{\mathbb{Z}_p}}
\newcommand{\Qp}{{\mathbb{Q}_p}}

\newcommand{\Fp}{{\mathbb{F}_p}}
\newcommand{\Fpd}{{\mrm{Spd}(\mathbb{F}_p)}}
\newcommand{\Zpd}{{\mrm{Spd}(\mathbb{Z}_p)}}
\newcommand{\Wkd}{{\Spdf{W(k)}}}

\newcommand{\Qpd}{{\mrm{Spd}(\mathbb{Q}_p)}}

\newcommand{\Oe}{O_{{E}}}

\newcommand{\ar}{\arrow}
\newcommand{\ovr}[1]{{#1_{\mrm{red}}}}
\newcommand{\dia}{\diamondsuit}
\newcommand{\ov}[1]{{\overline{#1}}}

\newcommand{\und}[1]{{\underline{#1}}}
\newcommand{\pthr}{^{\frac{1}{p^\infty}}}

\newcommand{\F}{\cali{F}}

\newcommand{\G}{\cali{G}}
\newcommand{\g}{{G}}

\newcommand{\Di}{\mathscr{D}}
\newcommand{\Hi}{\cali{H}}
\newcommand{\Si}{\mathfrak{S}}

\newcommand{\Tubf}[2]{{\widehat{#1}_{/#2}}}
\newcommand{\Tf}[2]{{\widehat{#1}_{/#2}}}

\newcommand{\Tup}[2]{{{#1}^{\circledcirc}_{/#2}}}

\newcommand{\Grm}[1]{{\mrm{Gr}^{\g,\leq \mu}_{#1}}}

\newcommand{\Gru}[1]{{\mrm{Gr}^{G}_{#1}}}
\newcommand{\GruH}[1]{{\mrm{Gr}^{H}_{#1}}}
\newcommand{\GrumH}[1]{{\mrm{Gr}^{H,\leq\mu}_{#1}}}
\newcommand{\Grum}[1]{{\mrm{Gr}^{G,\leq\mu}_{#1}}}
\newcommand{\Gruw}[1]{{\mrm{Gr}^{G,\leq w}_{#1}}}
\newcommand{\GruwH}[1]{{\mrm{Gr}^{H,\leq w}_{#1}}}

\newcommand{\GrueH}[2]{{\mrm{Gr}^{H,\leq #2}_{#1}}}

\newcommand{\GrWme}[1]{{\mrm{Gr}^{\g,\leq \mu}_{\cali{W},#1}}}

\newcommand{\Wred}{{\cali{W}_{\mrm{red}}^+}}

\newcommand{\U}[2]{{U_{#1\leq #2\neq 0}}}

\newcommand{\N}[2]{{N_{#1\ll #2}}}
\newcommand{\I}[1]{{\cali{I}^{\preceq}(#1)}}
\newcommand{\Iv}[1]{{\cali{I}^{\preceq}_{\mrm{ver}}(#1)}}

\newcommand{\nc}{\newcommand}
\nc{\on}{\operatorname}
\nc{\pot}[1]{ [\hspace{-0,5mm}[ {#1} ]\hspace{-0,5mm}] }
\nc{\rpot}[1]{ (\hspace{-0,7mm}( {#1} )\hspace{-0,7mm}) }

\title{Specialization maps for Scholze's category of diamonds}
\author{Ian Gleason}
\email{igleason@uni-bonn.de}
\subjclass[2020]{14G45}
\address{Mathematisches Institut der Universit\"at Bonn, Endenicher Allee 60, 53115, Bonn, Germany}
\begin{document}
\begin{abstract}
We introduce the specialization map in Scholze's theory of diamonds. We consider v-sheaves that ``behave like formal schemes" and call them kimberlites. We attach to them: a reduced special fiber, an analytic locus, a specialization map, a Zariski site, and an \'etale site. When the kimberlite comes from a formal scheme, our sites recover the classical ones. We prove that unramified $p$-adic Beilinson--Drinfeld Grassmannians are kimberlites with finiteness and normality properties. 
\end{abstract}

\maketitle
\tableofcontents

\section*{Acknowledgements}
We thank the author’s PhD advisor, Sug Woo Shin, for his interest, his insightful questions and suggestions at every stage of the project, and for his generous constant encouragement and support during the PhD program. To David Hansen for many very helpful conversation that cleared misunderstandings of the author, for providing a reference and a proof to \Cref{pro:specialisniceforTate}, for kindly reading an early draft of this article. To Jo\~ao Louren\c{c}o for some correspondences that encouraged the author to pursue \Cref{thm2:non-analyticthm} in this generality, for his interest in our work and many conversations that have shaped our perspective on the subject. To Ben Heuer for explaining us his perspective on the specialization map for formal schemes.
To the anonymous referee for very helpful suggestions.
To Laurent Fargues, Peter Scholze and Jared Weinstein for answering questions the author had on $p$-adic Hodge theory and the theory of diamonds. To Johannes Ansch\"utz, Alexander Bertoloni, Patrick Daniels, Dong Gyu Lim, Zixin Jiang and Alex Youcis for helpful conversations and correspondences. \\ 

This work was supported by the Doctoral Fellowship from the “University of California Institute for Mexico and the United States” (UC MEXUS) and the “Consejo Nacional de Ciencia y Tecnolog\'ia” (CONACyT), and by ``Deutsche Forschungsgemeinschaft" (DFG, German Research Foundation) via Scholze's Leibniz-Preis. This research was conducted in during the authors PhD program at UCB (University of California Berkeley) and during the authors postdoctoral stay at Universit\"at Bonn, we are grateful to both institutions.

\section*{Statements and Declarations}
On behalf of all authors, the corresponding author states that there is no conflict of interest. My manuscript has no associated data.

\section*{Introduction}
As Fargues--Scholze \cite{FS21} and Scholze--Weinstein \cite{Ber} show, Scholze's theory of diamonds and v-sheaves \cite{Et} is a powerful geometric framework to study (among other things) the local Langlands correspondence and the theory of local Shimura varieties. One of the early milestones of Scholze's theory roughly says that diamonds capture correctly the \'etale site of an analytic adic space (\cite[\S 15]{Et}). Moreover, the category of diamonds contains many geometric objects of arithmetic interest that do not come from an analytic adic space, and for these spaces one still gets a well-behaved \'etale site.

Although Scholze's theory of v-sheaves is also useful to study adic spaces that are not analytic (like the ones coming from a formal scheme), some complications arise. For example, the v-sheaf associated to a non-analytic adic space has more open subsets than one would expect \cite[\S 18]{Ber}. In particular, a comparison of \'etale sites can't hold since even the site of open subsets do not coincide. Despite this complications, it is still profitable to understand the behavior of the v-sheaves associated to non-analytic adic spaces. The main motivation to work this out is because there are v-sheaves of arithmetic interest that do not come from an adic space, but ``resemble" the behavior of a formal scheme. The main examples to keep in mind are the integral models of moduli spaces of $p$-adic shtukas proposed in \cite[\S 25]{Ber} or the $p$-adic Beilinson--Drinfeld Grassmannians of \cite[\S 21]{Ber}. 

In rough terms this article does the following:  

\begin{enumerate}
	\item Study the topological space $|\Spd{A}|$ for a general Huber pair $\Hub{A}$ over $\Zp$. We overcome many of the technical difficulties of working with these spaces. 	
	\item Propose a rigorous definition of what it means for a v-sheaf to ``resemble" a formal scheme. We call these v-sheaves \textit{kimberlites}.
	\item Construct specialization maps attached to kimberlites. This recovers the classical specialization maps of formal schemes.  
	\item Attach ``Zariski" and ``\'etale" sites to a kimberlite. This allow us to recover the Zariski and \'etale sites of a formal scheme intrinsically from the v-sheaf attached to it.  
	\item Verify that the $p$-adic Beilinson--Drinfeld Grassmannians attached to reductive groups over $\Zp$ are interesting examples of kimberlites that do not come from formal schemes.  
\end{enumerate}

Although this work (admittedly of technical nature) is far from being a robust theory, we think it makes appreciable progress in our understanding of this kind of v-sheaves and sets a stepping stone for future investigations. For instance, the constructions and techniques discussed here have already found applications in the following works: 
\begin{enumerate}
	\item In our work on geometric connected components of local Shimura varieties \cite{gleason2021geometric}.
	\item In our collaborative work with Ansch\"utz, Louren\c{c}o and Richarz on the Scholze--Weinstein conjecture \cite{AGLR22}.
	\item In the representability results of Pappas and Rapoport \cite{pappas2021padic}.
\end{enumerate}
We find it reasonable to expect that our considerations will play a role in more general representability results of integral local Shimura varieties, and a role in the study of ``the nearby cycles functor" (\Cref{rem:naivenearbycycles}). \\    

Let us give a more detailed account of our results. In \cite{Et}, Scholze sets foundations for the theory of diamonds which can be defined as certain sheaves on the category of characteristic $p$ perfectoid spaces endowed with a Grothendieck topology called the v-topology. He associates to any pre-adic space $X$ over $\Zp$ (not necessarily analytic) a small v-sheaf $X^\diamondsuit$, and whenever $X$ is analytic he proves that $X^\diamondsuit$ is a locally spatial diamond. If $X=\Spa{B}$ for a Huber pair $\Hub{B}$, then $X^\dia$ is denoted $\Spd{B}$. Moreover, Scholze assigns to any small v-sheaf $\F$ an underlying topological space $|\F|$ and whenever $\F=X^\diamondsuit$ he constructs a functorial surjective and continuous map $|\F|\to|X|$. When $X$ is analytic it is proven in \cite{Ber} that this map is a homeomorphism, but this map fails to be injective almost always for pre-adic spaces that have non-analytic points. To tackle this difficulty, we associate to a Huber pair $\Hub{B}$ what we call below its \textit{olivine spectrum}, which we denote $\Spo{B}$. The definition of this topological space is concrete enough to allow computations to take place. Moreover, in most cases of interest $\Spo{B}$ recovers $|\Spd{B}|$. 

\begin{thm2}
	\label{thm2:mostHuberpairsareolivine}
	Let $\Hub{B}$ be a complete Huber pair over $\Zp$, there is a functorial bijective and continuous map $|\Spd{B}|\to \Spo{B}$. Moreover, it is a homeomorphism if $\Hub{B}$ is topologically of finite type over $(B_0,B_0)$ for $B_0\subseteq B^+$ a ring of definition.	
\end{thm2}
In \cite[\S 18]{Ber} Scholze and Weinstein attach to a perfect scheme $X$ in characteristic $p$ a v-sheaf denoted $X^\diamond$. Moreover, they prove that the functor $X\mapsto X^\diamond$ is fully-faithful. It will be clear to the reader that many of our techniques used to prove \Cref{thm2:mostHuberpairsareolivine} are borrowed from Scholze and Weinstein's approach to their full-faithfulness result, but our perspective allow us to go farther. 
\begin{thm2}
	\label{thm2:non-analyticthm}
	Let $Y$ be a perfect non-analytic adic space over $\Fp$ and let $X$ be a pre-adic space over $\Zp$. The natural map $\mrm{Hom}_{_\mrm{PreAd}}(Y,X)\to\mrm{Hom}(Y^\dia,X^\dia)$ is bijective. In particular, $(-)^\dia$ is fully faithful when restricted to the category of perfect non-analytic adic spaces over $\Fp$.
\end{thm2}
\Cref{thm2:non-analyticthm} allow us to recover Scholze and Weinstein's result as a particular case. Also, the olivine spectrum allows us to show that Scholze and Weinstein's functor $X\mapsto X^\diamond$ is continuous for the v-topology and admits a right adjoint $\F\mapsto \F\red$ at the level of topoi. We call this the \textit{reduction functor} and it plays a key role in the rest of our theory. In general, the objects obtained from the reduction functor might not be perfect schemes, but they are ``scheme theoretic v-sheaves" and they come equipped with an underlying topological space that agrees with the Zariski topology whenever they are representable.\\


Let us describe our approach to study v-sheaves that ``behave like'' formal schemes. We consider three layers, in each layer we get closer to capture the behavior of formal schemes. 
We first recall a more classical case. Let $\cali{X}$ be a $p$-adic separated formal scheme topologically of finite type over $\Zp$. One can associate to $\cali{X}$ a rigid analytic space over $\Qp$, that we will denote by $X_\eta$. We can also associate to $\cali{X}$ a finite type reduced scheme over $\Fp$, that we denote by $\ov{X}$. Huber's theory of adic spaces allows us to consider $X_\eta$ as an adic space and assign to it a topological space $|X_\eta|$. Moreover, one can construct a continuous map $\mrm{sp}_{\cali{X}}:|X_\eta|\to |\ov{X}|$, where $|\ov{X}|$ is the usual Zariski space underlying $\ov{X}$ \cite[Remark 7.4.12]{Bhatt} or \cite[Definition 6.4]{Lou}). A theorem of Louren\c{c}o (\cite[Theorem 18.4.2]{Ber}) says that ``nice enough" formal schemes can be recovered from the triple $(X_\eta,\ov{X},\mrm{sp}_{\cali{X}})$. We propose that v-sheaves that ``resemble" formal schemes should be those for which a specialization map can be constructed.  

Consider the following. Given a Tate Huber pair $\Hub{A}$ with pseudo-uniformizer $\varpi\in A^+$ the specialization map $\mrm{sp}_{A}:\Spa{A}\to \mrm{Spec}(A^+/\varpi)$ assigns to $x\in \Spa{A}$ the prime ideal $\mathfrak{p}_x$ of those elements $a\in A^+$ for which $|a|_x<1$. 
Observe that this construction is functorial in the category of Tate Huber pairs. We wish to exploit functoriality to descend this specialization map to more general v-sheaves. The first question is: What should the target of the specialization map be? 

One can compute directly that if $\Hub{A}$ is a uniform Tate Huber pair, then $\mrm{Spd}(A^+) \red$ is the perfection of $\mrm{Spec}(A^+/\varpi)$. This suggests that if we want to attach a specialization map to a v-sheaf $\F$ the target of this map should be $|\F\red|$. 

A key aspect that makes the specialization map for Tate Huber pairs functorial is that every map of Tate Huber pairs $\Spa{A}\to \Spa{B}$ automatically upgrades ``integrally" to a map $\Spf{A^+}\to \Spf{B^+}$. This motivates the following definition:

\begin{defi2}
	Let $\F$ be a small v-sheaf, $\Hub{A}$ be a Tate Huber pair and $f:\Spd{A}\to \F$ a map.
\begin{enumerate}
	\item We say that $\F$ \textit{formalizes} $f$ (or that $f$ is \textit{formalizable}) if there is $t:\Spdf{A^+}\to \F$ factoring $f$. 
%
	\item We say that $\F$ \textit{v-formalizes} $f$ if there is a v-cover $g:\Spa{B}\to \Spa{A}$ such that $\F$ formalizes $f\circ g$.
	\item We say that $\F$ is \textit{v-formalizing} if it v-formalizes any $f$ as above.
\end{enumerate}
\end{defi2}

Given a v-formalizing v-sheaf $\F$ one could define the specialization map $\mrm{sp}_{\F}:|\F|\to |\F\red|$ so that for any “formalized” map $f:\Spdf{A^+}\to \F$ the following diagram is commutative:
\begin{center}
	\begin{tikzcd}
		\mid \Spa{A}\mid \ar{r}\ar{d}{\mrm{sp}_{A}} & \mid \Spdf{A^+}\ar{r}{\mid f \mid} \mid & \mid \F\mid \ar{d}{\mrm{sp}_{\F}}\\
		\mid \mrm{Spec}(A^+/\varpi)^\mrm{perf}\mid \ar{rr}{\mid f\red \mid}& &\mid \F\red\mid 
	\end{tikzcd}
\end{center}
The recipe to compute the specialization map would then be as follows: given $x\in |\F|$ represent it by a map $\iota_x:\Spa{C}\to \F$, find a formalization $f_x:\Spdf{C^+}\to \F$ of $\iota_x$. Apply the reduction functor to $f_x$ to obtain a map $f_x\red:\mrm{Spec}(C^+/\varpi)^\mrm{perf}\to \F\red$. Look at the image of the closed point in $\mrm{Spec}(C^+/\varpi)$ under $|f_x\red|$. This is $\mrm{sp}_{\F}(x)$.

The natural question is whether or not this is well defined. The problem being that the map $\iota_x:\Spa{C}\to \F$ might have more than one formalization. The naive guess is that this doesn't happen if $\F$ is separated as a v-sheaf. Unfortunately, this is false. At the heart of the problem is the following pathology: although $|\Spa{C}|$ is dense within $|\Spf{C^+}|$ it is not true that $|\Spd{C}|$ is dense within $|\Spdf{C^+}|$. It is this key subtlety that requires sufficient understanding of the olivine spectrum of Huber pairs. 

\begin{defi2}
	Let $f:\F\to \G$ be a map of v-sheaves.
\begin{enumerate}
	\item We say $f$ is \textit{formally adic} if the following diagram induced by adjunction is Cartesian: 
\begin{center}
\begin{tikzcd}
	(\F\red) ^\diamond \ar[r]\ar[d] &(\G\red)^\diamond\ar[d]\\
	\F\ar[r] &\G
\end{tikzcd}
\end{center}
\item If $\F$ comes with a formally adic map to $\Zpd$ we say that $\F$ is \textit{$p$-adic}.
\item We say $f$ is \textit{formally closed} if it is a formally adic closed immersion.
\item We say $\F$ is \textit{formally separated} if the diagonal $\F\to \F\times \F$ is formally closed.
\end{enumerate}	
\end{defi2}
Using the olivine spectrum we prove that $|\Spd{C}|$ is ``formally dense" in $|\Spdf{C^+}|$. The main feature of a formally separated v-sheaf $\F$ is that a map $\iota:\Spa{A} \to \F$ has at most one formalization (if any).

Combining the two inputs we say that a v-sheaf $\F$ is \textit{specializing} if it is v-formalizing and formally separated, this is the first layer of approximation to the definition. We attach functorially to such $\F$ a continuous specialization map $|\F|\to |\F\red|$.

Now, specializing v-sheaves produce all the specialization maps we are interested in, but they are still too general to capture the behavior of formal schemes. 
\begin{defi2}
	Let $\F$ be a specializing v-sheaf. We say $\F$ is a \textit{prekimberlite} if:
	\begin{enumerate}[a)]
	\item $\F\red$ is represented by a scheme. 
	\item The map $(\F\red)^\diamond \to \F$ coming from adjunction is a closed immersion.
\end{enumerate}
If $\F$ is a prekimberlite, we let the \textit{analytic locus} be $\F^{\mrm{an}}=\F\setminus (\F\red)^\diamond$. 
\end{defi2}
We can attach \'etale and Zariski sites to a prekimberlite as follows:
\begin{defi2}
	Suppose $\F$ is a prekimberlite, we let $(\F)_{\mrm{qc}, \mrm{for\text{-}\acute{e}t}}$ be the category that has as objects maps $f:\G\to \F$ where $\G$ is a prekimberlite and $f$ is formally adic, \'etale and quasicompact. Morphisms are maps of v-sheaves commuting with the structure map. We call objects in this category the \textit{\'etale formal neighborhoods} of $\F$. If $f$ is also injective we call them open formal neighborhoods of $\F$. 
\end{defi2}

\begin{thm2}
	\label{thm2:invarianceofetalesite}
	For $\F$ a prekimberlite, the reduction functor $(-)\red:(\F)_{\mrm{qc},\mrm{for\text{-}\acute{e}t}}\to (\F\red)_{\mrm{qc},\mrm{\acute{e}t},\mrm{sep}}$ is an equivalence. Here, the target category are the maps of perfect schemes $f:Y\to \F\red$ that are quasicompact, \'etale and separated. Moreover, this functor restricts to an equivalence between open formal neighborhoods and quasicompact open immersions. 
\end{thm2}

Now, if $\frak{X}$ is a separated formal scheme locally admitting a finitely generated ideal of definition (see \Cref{conv:formalschemes} below for details), then $\frak{X}^\dia$ is a prekimberlite and $(\frak{X}^\dia)\red$ is the perfection of the reduced subscheme of $\frak{X}$. In particular, one can recover the \'etale site of $\frak{X}$ from $(\frak{X}^\dia)_{\mrm{qc},\mrm{for\text{-}\acute{e}t}}$ through \Cref{thm2:invarianceofetalesite}.

Let us describe the inverse functor, for this we consider the following construction due to Heuer \cite{heuer21}. For a perfect scheme $X$ in characteristic $p$ we let $\Heuer{X}$ denote the v-sheaf given by the analytic sheafification of the rule $(R,R^+)\mapsto X(\mrm{Spec}(R^+/\varpi)^\mrm{perf})$ where $\Hub{R}$ is affinoid perfectoid and $\varpi\in R^+$ is a pseudo-uniformizer. For a prekimberlite $\F$ we get a map of v-sheaves $\mrm{SP}_\F:\F\to \Heuer{(\F\red)}$. If $f:V\to \F\red$ is \'etale, quasicompact and separated then $\Tf{\F}{V}:=\F\times_{\Heuer{(\F\red)}}\Heuer{V}$ is the \'etale formal neighborhood of $\F$ with $(\Tf{\F}{V})\red=V$. The two key ingredients are that for perfect schemes $X$ we have an identification $(\Heuer{X})\red=X$, and if $V\to X$ is \'etale then $\Heuer{V}\to \Heuer{X}$ is formally adic and \'etale. This later statement in turn reduces to the invariance of \'etale sites under nilpotent thickenings and perfection. Heuer's construction also allows us to consider what we call \textit{formal neighborhoods}. If $\F$ is a prekimberlite and $S\subseteq \F\red$ is a locally closed subscheme we can consider $\Tf{\F}{S}:=\F\times_{\Heuer{(\F\red)}}\Heuer{S}$. We always have $\Tf{\F}{S}\subseteq \F$ and when $S$ is constructible this is even an open immersion.\footnote{These open immersions are only formally adic when $S$ is an open immersion.} This construction generalizes ``completion" of a formal scheme along a locally closed immersion.

We are ready for the third approximation.
\begin{defi2}
	Let $\F$ be a prekimberlite.
	\begin{enumerate}
		\item We say $\F$ is valuative if $\mrm{SP}_\F:\F\to \Heuer{(\F\red)}$ is partially proper.
	\item A \textit{smelted kimberlite} is a pair $\Ki=(\F,\Di)$ where $\F$ is a valuative prekimberlite, $\Di$ is a quasiseparated locally spatial diamond and $\Di \subseteq \F^{\mrm{an}}$ is open. The main cases of interest are when $\Di=\F^\mrm{an}$ or when $\Di=\F\times_\Zpd \Qpd$. 
\item We define the specialization map $\mrm{sp}_{\mathcal{K}}:|\Di|\to |\F\red|$ as the composition $|\Di|\to |\F|\xrightarrow{\mrm{sp}_{\F}} |\F\red|$.
	If the context is clear we write $\mrm{sp}_\Di$ instead of $\mrm{sp}_\Ki$.
\item We say $\F$ is a \textit{kimberlite} if $(\F,\F^\mrm{an})$ is a smelted kimberlite and $\mrm{sp}_{\F^\mrm{an}}$ is quasicompact.
	\end{enumerate}
\end{defi2}

The specialization map for kimberlites and smelted kimberlites is better behaved since it is even continuous for the constructible topology.

\begin{thm2}
	\label{pro2:spectralmapSch-Spatial}
	Let $\Ki=(\F,\Di)$ be a smelted kimberlite and $\G$ be a kimberlite, the following hold: 
	\begin{enumerate}
		\item $\mrm{sp}_\Di:|\Di|\to |\F\red|$ is a specializing, spectral map of locally spectral spaces. 
		\item $\mrm{sp}_{\G^\mrm{an}}:|\G^\mrm{an}|\to |\G\red|$ is also a closed map.
	\end{enumerate}
\end{thm2}
If $(\F,\Di)$ is a smelted kimberlite and $S\subseteq |\F\red|$ we can define analogs of Berthelot tubes, by letting $\Tup{\Di}{S}=\Tf{\F}{S}\times_\F \Di$. We call these spaces the \textit{tubular neighborhood} of $\Di$ around $S$.

Finally, to study the $p$-adic Beilinson--Drinfeld Grassmannians we introduce some ``finiteness" and ``normality" conditions.
\begin{defi2}
	Let $\Ki=(\F,\Di)$ a smelted kimberlite and $\G$ a kimberlite.
\begin{enumerate}
	\item We say $\Di$ is a \textit{cJ-diamond} (constructibly Jacobson) if rank $1$ points are dense in the constructible topology of $\Di$.
	\item We say that $\Ki$ is \textit{rich} if: $\Di$ is a cJ-diamond, $|\F\red|$ is locally Noetherian and $\mrm{sp}_{\Di}:|\Di|\to |\F\red|$ is surjective. 
	\item We say that $\G$ is \textit{rich} if: $(\G, \G^\mrm{an})$ is rich. 
	\item If $\Ki$ is rich we say it is \textit{topologically normal} if for every closed point $x\in |\F\red|$ the tubular neighborhood $\Tup{\Di}{x}$ is connected.\footnote{In a previous version of this article we had already introduced this notions, but we hadn't realized the connection to normality. Our motivation to call this ``topologically normal" comes from \cite[Proposition 2.38]{AGLR22}.}
\end{enumerate}
\end{defi2}

Let $\g$ denote a reductive over $\Zp$ and let $T\subseteq B\subseteq \g$ denote integrally defined maximal torus and Borel subgroups respectively. 
Let $\mu\in X_*^+(T_{\overline{\mathbb{Q}}_p})$ be a dominant cocharacter with reflex field $E\subseteq \overline{\mathbb{Q}}_p$. Let $O_E$ denote the ring of integers of $E$ and let $k_E$ denote the residue field. Let $\Grm{\Oe}$ denote the v-sheaf parametrizing $B^+_\mrm{dR}$-lattices with $\g$-structure whose relative position is bounded by $\mu$ as in \cite[Defintion 20.5.3]{Ber} and let $\GrWme{k_E}$ denote the Witt vector affine Grassmannian \cite{Zhu}, \cite{Witt}. 
%
%
Here is our result:

\begin{thm2}
	\label{thm2:GrassmannianisaKimberlite}
	$\Grm{O_E}$ is a topologically normal rich $p$-adic kimberlite with $(\Grm{O_E})\red=\GrWme{k_E}$. In particular, the specialization map is a closed, surjective and spectral map of spectral topological spaces. 
\end{thm2}
This result has partially been generalized in our collaboration \cite{AGLR22}. There, we prove that the local models for parahoric groups are rich $p$-adic kimberlites. Nevertheless, we only improve the ``normality" part of the result if we assume that $\mu$ is minuscule and outside certain cases in small characteristic. 

In \cite{gleason2021geometric}, we use normality of $\Grm{O_E}$ to prove normality of moduli spaces of $p$-adic shtukas which is a key step to prove the main theorem of \cite{gleason2021geometric} for the following reason. Classically, normality of formal scheme ensures that the generic fiber and special fibers have the same connected components. This also happens for rich smelted kimberlites.

	We prove \Cref{thm2:GrassmannianisaKimberlite} by using a Demazure resolution. Our key observation is that one can do the Demazure resolution using either $B^+_\mrm{dR}$-coefficients or $A_\mrm{inf}$-coefficients. The use of $A_\mrm{inf}$-coefficients makes it clear that $\Grm{O_E}$ is v-formalizing. Normality of $\Grm{O_E}$ can be deduced from the normality of the source in the Demazure resolution, which in turn can be deduced inductively from it's expression as iterated $(\bb{P}^1)^\dia$-bundles.

Let us comment on the organization of the paper.
\begin{enumerate}[I)]

	\item In the first section, we give a short review of the theory of diamonds, the v-topology and some facts about spectral topological spaces. We also review Scholze's $\diamondsuit$ functor that takes as input a pre-adic space over $\Zp$ and returns as output a v-sheaf.
	\item In the second section, we introduce and study the olivine spectrum of a Huber pair. We prove \Cref{thm2:mostHuberpairsareolivine} and \Cref{thm2:non-analyticthm}.  
	\item In the third section, we review the small diamond functor $\diamond$. We prove the continuity of $\diamond$. We introduce the reduction functor as the right adjoint to $\diamond$. We introduce and study ``formally adic" maps.	
	\item In the fourth section, we develop our theory of specialization maps. We introduce specializing v-sheaves and prekimberlites. We introduce formal neighborhoods, \'etale formal neighborhoods and we prove \Cref{thm2:invarianceofetalesite}. We introduce kimberlites, and smelted kimberlites and prove \Cref{pro2:spectralmapSch-Spatial}. We prove that formal schemes give rise to kimberlites. Finally, we introduce the finiteness and normality conditions. 
	\item In the fifth section, we study the specialization map for $p$-adic Beilinson--Drinfeld Grassmannians. We review the contruction of twisted loop groups with $B^+_\mrm{dR}$ and $A_\mrm{inf}$ coefficients. We construct the two versions of the ``integral" Demazure resolution. We prove \Cref{thm2:GrassmannianisaKimberlite}.
\end{enumerate}

\section{The v-topology}

We assume familiarity with the theory of perfectoid spaces and diamonds as discussed in \cite[\S 7]{Ber} or \cite[\S 3]{Et}. For the most part the reader can ignore the set-theoretic subtleties that arise from the theory. Nevertheless, for some of our constructions set-theoretic carefulness is necessary. 

\subsection{Recollections on diamonds and small v-sheaves}
We let $\mrm{Perfd}$ denote the category of perfectoid spaces and $\mrm{Perf}$ the subcategory of perfectoid spaces in characteristic $p$. 
Recall that we can endow $\mrm{Perfd}$ with two Grothendieck topologies, called the pro-\'etale topology and v-topology respectively \cite[Definition 7.8, Definition 8.1]{Et}.
%
The following example of a cover for the v-topology will be used repeatedly.
\begin{exa}
	\label{exa:prodpointsbasis}
	Let $\Spa{A}$ be an affinoid perfectoid space, with pseudo-uniformizer $\varpi\in A^+$. Given $x\in |\Spa{A}|$ let $\iota_x:\mrm{Spa}(k(x),k(x)^+)\to \Spa{A}$ be the residue field. By \cite[Corollary 6.7]{Perf}, each $\mrm{Spa}(k(x),k(x)^+)$ is perfectoid. Let $R^+:=\prod_{x\in|\Spa{A}|}k(x)^+$ endowed with the $\varpi$-adic topology and let $R=R^+[\frac{1}{\varpi}]$. Then $\Spa{R}$ is perfectoid and $\Spa{R}\to \Spa{A}$ is a v-cover. 

	If one replaces the role of $k(x)$ by a completed algebraic closure $C(x)$ of $k(x)$, and on considers $S^+:=\prod_{x\in|\Spa{A}|}C(x)^+$  where $C(x)^+$ denotes the integral closure of $k(x)^+$ in $C(x)$, then by letting $S=S^+[\frac{1}{\varpi}]$ we also have that $\Spa{S}$ is perfectoid and that $\Spa{S}\to \Spa{A}$ is a v-cover.  

\end{exa}
\begin{defi}
	\label{defi:prodpoints}
	Let $I$ be a set and $\{\Hub{C_i},\varpi_i\}_{i\in I}$ a collection of tuples where each $C_i$ is an algebraically closed nonarchimedean field, the $C_i^+$ are open and bounded valuation subrings of $C_i$, and $\varpi_i$ is a of pseudo-uniformizer. Let $R^+:=\prod_{i\in I} C^+_i$, let $\varpi=(\varpi_i)_{i\in I}$, endow $R^+$ with the $\varpi$-adic topology and let $R:=R^+[\frac{1}{\varpi}]$. Any space of the form $\Spa{R}$ constructed in this way will be called a product of points.
\end{defi}
\begin{rem}
	\label{rem:differentunifgivesdifferent}
	Different choices of pseudo-uniformizers $(\varpi_i)_{i\in I}$ give rise to different adic spaces. 
Also, \Cref{exa:prodpointsbasis} proves that products of points form a basis for the v-topology in the category of perfectoid spaces. 
\end{rem}
Recall the notion of totally disconnected spaces.
\begin{defi}(\cite[Definition 7.1, Definition 7.15, Lemma 7.5]{Et})
	\label{defi:totallydisconnect}
	An affinoid perfectoid space $\Spa{R}$ is totally disconnected if it splits every open cover. Moreover, it is strictly totally disconnected if it splits every \'etale cover.
\end{defi}
\begin{pro}\textup{(\cite[Lemma 7.3, Proposition 7.16, Lemma 11.27]{Et})}
	\label{pro:connectdcompcriterion}
	Let $Y$ be an affinoid perfectoid space. $Y$ is represented by a strictly totally disconnected space if and only if every connected component of $Y$ is represented by $\Spa{C}$ for $C$ an algebraically closed field and $C^+$ an open and bounded valuation subring.
\end{pro}

\begin{pro}
	Product of points are strictly totally disconnected perfectoid space. \label{pro:productofpointsarestrictlytotallydisconn} \label{pro:prodtotallydiscon}
\end{pro}
\begin{proof}
	Fix notation as in \Cref{defi:prodpoints}. The closed-opens subsets of $\Spa{R}$ are given by subsets of $I$. Every $x\in \pi_0(\Spa{R})$ is computed as $\bigcap_{U\in\cali{U}_x} U$ for some ultrafilter. This is a Zariski closed subsets cut out by an ideal of idempotents $\cali{I}_x=\langle1_V \rangle$, with $V\subseteq I$ and ${V\notin \cali{U}}$. The $\cali{O}^+$-structure sheaf of $x$ is the $\varpi$-completion of $R^+/\cali{I}_x$. Let $V=R^+/\cali{I}_x$ and $V'$ the $\varpi$-adic completion of $V$.
	By \Cref{pro:connectdcompcriterion}, it suffices to prove that $V'$ is a valuation ring with algebraically closed fraction field. 
	This easily reduces to the same claim on $V$. It is not hard to see that $\mrm{Frac}(V)=(\prod_{i \in I} C_i)/\cali{I}_x$. Moreover, the properties of being a valuation ring or being an algebraically closed field can be expressed in first order logic so these properties pass to ultraproducts, alternatively we can cite \cite[Lemma 3.27]{bhatt2021arc}. 
\end{proof}
The v-topology on $\mrm{Perfd}$ is subcanonical \cite[Corollary 8.6]{Et}. We denote a perfectoid space and the sheaf it represents with the same letter. When a distinction is needed, if $X$ denotes a perfectoid space we denote by $h_X$ the sheaf it represents. Let $Y$ be a diamond \cite[Definition 11.1]{Et}, we recall the definition of its associated underlying topological space $|Y|$. 
\begin{defi}
	\label{defi:diamondtop}
	A map $p:\Spa{K}\to Y$ is a point if $K$ is a perfectoid field in characteristic $p$ and $K^+$ is an open and bounded valuation subring of $K$. Two points $p_i:\Spa{K_i}\to Y$, $i\in\{1,2\}$, are equivalent if there is a third point $p_3:\Spa{K_3}\to Y$, and surjective maps $q_i:\Spa{K_3}\to \Spa{K_i}$ making the following commutative diagram: 
$$
\begin{tikzcd}
	&\Spa{K_1} \arrow{rd}{p_1} & \\
	\Spa{K_3}\arrow{ru}{q_1} \arrow{rd}{q_2} \arrow{rr}{p_3}	& & Y  \\
	&\Spa{K_2} \arrow{ru}{p_2}& \\
\end{tikzcd}
$$
We let $|Y|$ denote the set of equivalence classes of points of $Y$. 
\end{defi}
Scholze proves that if $Y$ has a presentation $X/R$ with $X$ and $R$ perfectoid, then there is canonical bijection between $|Y|$ and $|X|/|R|$. Moreover, the quotient topology on $|Y|$ coming from the surjection $|X|\to |Y|$ doesn't depend on the presentation \cite[Proposition 11.13]{Et}.
Also, if $X$ is a perfectoid space, then $h_X$ is a diamond and $|h_X|$ is canonically homeomorphic to $|X|$.  

We refer to sheaves on $\mrm{Perf}$ for the v-topology as v-sheaves. 
Recall that a v-sheaf is said to be small if it admits a surjection from a representable sheaf. We denote by $\topPerf$ the category of small v-sheaves. 
There's a more explicit way of defining this. 
Given a cut-off cardinal $\kappa$ (\cite[\S 4, \S 8 ]{Et} for details) denote by $\mrm{Perf}_\kappa$ the category of $\kappa$-small perfectoid spaces in characteristic $p$ and by $\topPerf_\kappa$ the topos of sheaves for the v-topology on this category. Objects in this topos are called $\kappa$-small v-sheaves. We have natural fully-faithful embeddings $\topPerf_\kappa\subseteq \topPerf_\lambda$ for $\kappa<\lambda$ and $\topPerf=\bigcup_{\kappa} \topPerf_\kappa$ as a big filtered colimit over cut-off cardinals $\kappa$. 
%
%

Scholze associates to any small v-sheaf a topological space. The definition is similar to \Cref{defi:diamondtop}, with the role of perfectoid spaces exchanged by diamonds. The key point being that if $X\to Y$ is a map of small v-sheaves with $X$ a diamond then $R=X\times_Y X$ is also a diamond and $Y=X/R$ \cite[Proposition 12.3]{Et}. Scholze then defines $|Y|$ as $|X|/|R|$ with the quotient topology and by \cite[Proposition 12.7]{Et} this is well defined.
Given a topological space $T$ we can consider a presheaf on $\mrm{Perf}$, denoted $\underline{T}$, defined as $$\underline{T}\Hub{R}=\{f:|\Spa{R}|\to T\mid f\,\text{is\,\,continuous}\}$$
This is a v-sheaf but it might not be small. There is a natural transformation, $X\to \underline{|X|}$ of v-sheaves.
A morphism of small v-sheaves $j:U\to X$ is open if it is relatively representable in perfectoid spaces and after basechange it becomes an open embedding of perfectoid spaces. Open subsheaves of $X$ are uniquely determined by open subsets of $|X|$ (\cite[Proposition 11.15, Proposition 12.9]{Et}).
%
%
The concept of closed immersion is a little more subtle. It is not a purely topological condition in the sense that closed subsheaves of $\F$ are not in bijection with closed subsets of $|\F|$. Indeed, there are more closed subsets than closed immersions.
\begin{defi}\textup{(\cite[Definition 10.7, Proposition 10.11, Definition 5.6]{Et} )}
	A map of sheaves $\F\to \G$ is a closed immersion if for every $X=\Spa{R}$ a strictly totally disconnected space and a map $X\to \G$ the pullback $X\times_\F \G\subseteq X$ is representable by a closed immersion of perfectoid spaces. 
\end{defi}
The following result characterizes closed immersions.
\begin{prop}\textup{(\cite{AGLR22})}
	\label{pro:Joaoetal}
	For a v-sheaf $\F$ we say a subset $X\subseteq |\F|$ is weakly generalizing if for any geometric point $f:\Spa{C}\to \F$ we have that $f^{-1}(X)\subseteq |\Spa{C}|$ is stable under generization. For any v-sheaf $\F$ the rule $$X\mapsto \F\times_{\underline{|\F|}} \underline{X}\subseteq \F$$ gives a bijection between weakly generalizing closed subsets of $|\F|$ and closed subsheaves of $\F$.
\end{prop}
\subsection{Spectral spaces and locally spatial diamonds}
We recall the basic theory of spectral topological spaces. This material is taken from section \cite[\S2 ]{Et} where most of the proofs can be found.
\begin{defi}
	\label{defi:Spectral-space}
	Let $S$, $T$ be topological spaces, and $f:S\to T$ a continuous map.
\begin{enumerate}
	\item $T$ is spectral if it is quasicompact, quasiseparated, and it has a basis of open neighborhoods stable under intersection that consists of quasicompact and quasiseparated subsets.
	\item $T$ is locally spectral if it admits an open cover by spectral spaces.
	\item $f$ is a spectral map of spectral spaces if $S$ and $T$ is are spectral and $f$ is quasicompact.
	\item $f$ is a spectral map of locally spectral spaces if for every quasicompact open $U\subseteq S$ and quasicompact open $V\subseteq T$ with $f(U)\subseteq V$ $f|_U:U\to V$ is spectral.
\end{enumerate}
\end{defi}
\begin{thm}\textup{(Hochster)}
	\label{thm:Hochsteter}
	For a topological space $T$ the following are equivalent:
\begin{enumerate}
	\item $T$ is spectral.
	\item $T$ is homeomorphic to the spectrum of a ring.
	\item $T$ is a projective limit of finite $T_0$ topological spaces.
\end{enumerate}
Moreover, the category of spectral topological spaces with spectral maps is equivalent to the pro-category of finite $T_0$ topological spaces.
\end{thm}
Given a spectral space $T$, we say that a subset $S$ is constructible if it lies in the Boolean algebra generated by quasicompact open subsets of $T$. For a locally spectral space $T$, a subset $S$ is constructible if for every quasicompact open subset $U\subseteq T$ the subset $S\cap U$ is constructible in $U$. The patch (or constructible) topology on $T$ is the one in which constructible subsets form a basis for the topology. A spectral space is Hausdorff and profinite for its patch topology and a locally spectral space is locally profinite for the patch topology.
\begin{pro}
	\label{pro:spectralvspatch}
	A continuous map of locally spectral spaces $f:S\to T$ is spectral if and only if it is continuous for the patch topology.
\end{pro}
\begin{defi}
	\label{defi:specialzing-generalizing}
	Let $f:S\to T$ be a continuous map of topological spaces.
\begin{enumerate}
	\item We say $f$ is generalizing if given $t_1,t_2\in T$ and $s_1\in S$ with $f(s_1)=t_1$ and such that $t_2$ generalizes $t_1$, then there exists an element $s_2$ generalizing $s_1$ with $f(s_2)=t_2$.
	\item We say $f$ is specializing if given $t_1,t_2\in T$ and $s_1\in S$ with $f(s_1)=t_1$ and such that $t_2$ specializes from $t_1$, then there exists an element $s_2$ specializing from $s_1$ with $f(s_2)=t_2$.
\end{enumerate}
\end{defi}

For a locally spectral space $T$ we say that a subset is pro-constructible if it is closed for the patch topology, or equivalently if it is an arbitrary intersection of constructible subsets. 
\begin{pro}\textup{(\cite[Lemma 2.4]{Et})}
	\label{pro:pro-cons-spec}
	Let $T$ be a spectral space and $S\subseteq T$ a pro-constructible subset. The closure $\overline{S}$ of $S$ in $T$ consists of the points that specialize from a point in $S$.
\end{pro}
\begin{cor}
	\label{cor:closedmap}
	Let $f:S\to T$ be a spectral map of spectral spaces. If $f$ is specializing then it is also a closed map.
\end{cor}
We warn the reader that the analogue of \Cref{cor:closedmap} for locally spectral spaces does not hold. 
\begin{pro}\textup{(\cite[Lemma 2.5]{Et})}
	\label{pro:generalizing-quotient}
Let $f:S\to T$ be a spectral map of spectral topological spaces. Assume $f$ is surjective and generalizing, then it is a quotient map.
\end{pro}
\begin{defi}\textup{(\cite[Definition 11.17]{Et})}
	\label{defi:spatial}
	Let $X$ be a diamond. We say that $X$ is a spatial diamond if it is quasicompact, quasiseparated and $|X|$ has a basis of open neighborhoods of the form $|U|$ where $U\subseteq X$ is a quasicompact open embedding. We say that $X$ is locally spatial if it has an open cover by spatial diamonds.
\end{defi}
The topology of spatial diamonds is spectral. Nevertheless, a diamond that has a spectral underlying topological space might not necessarily be spatial since the quasicompactness and quasiseparatedness conditions of \Cref{defi:spatial} are imposed in the topos-theoretic sense. 
\begin{pro}\textup{(\cite[Proposition 11.18, Proposition 11.19]{Et})}
	\label{pro:spatial-spectral}
	Let $X$ and $Y$ be locally spatial diamonds and $f:X\to Y$ a morphism of v-sheaves. The following assertions hold:
\begin{enumerate}
	\item $|X|$ is a locally spectral topological space.
	\item Any open subfunctor $U\subseteq X$ is a locally spatial diamond.
	\item $|X|$ is quasicompact (respectively quasiseparated) as a topological space if and only if $X$ is quasicompact (respectively quasiseparated) as a v-sheaf.
	\item The topological map $|f|$ is spectral and generalizing. In particular, if $|X|$ is quasicompact and $|f|$ is surjective then by \Cref{pro:generalizing-quotient} it is also a quotient map.  
\end{enumerate}
\end{pro}

\subsection{Pre-adic spaces as v-sheaves}
The theory of diamonds is mainly of “analytic” nature. On the other hand, we wish to consider spaces that are closer to schemes or formal schemes. The category of v-sheaves allows us to consider these three types of spaces at the same time. 
Recall that to any Huber pair $\Hub{A}$ we can associate a pre-adic space, $\mrm{Spa}^{\mrm{ind}}(A,A^+)$, as in \cite[Appendix to Lecture 3]{Ber}. One then constructs pre-adic spaces by appropriately glueing along rational covers.\footnote{One has to do this carefully since sheafiness doesn't hold in this generality.}
Every pre-adic space $X$ has an underlying topological space, and we can define the open analytic locus $|X|^{\mrm{an}}$ and the non-analytic locus $|X|^{\mrm{na}}$ in the naive way. 
That is, a point $x\in |X|$ is analytic if for every open affinoid $x\in \mrm{Spa}^{\mrm{ind}}\Hub{A}\subseteq |X|$ (equivalently one affinoid) $x$ is analytic in $\Spa{A}$.  
\begin{prop}
	\label{pro:preadicreduced}
	Given a pre-adic space $X$ there is a reduced non-analytic adic space $X^{\mrm{na}}$ and a map $X^{\mrm{na}}\to X$ which is final in the category of maps $Y\to X$ with $Y$ a reduced non-analytic adic space. Moreover, the map $|X^{\mrm{na}}|\to |X|^{\mrm{na}}$ is a homeomorphism.
\end{prop}
\begin{proof}
	In the affinoid case $\mrm{Spa}^{\mrm{ind}}(A,A^+)^{\mrm{na}}=\mrm{Spa}(A/A^{\circ \circ}\cdot A,A^+/A^{\circ \circ}\cdot A^+)$. Since $A/(A^{\circ \circ}\cdot A)$ is discrete it is sheafy. Moreover, if $\Hub{B}$ is discrete then $\mrm{Hom}(\mrm{Spa}^{\mrm{ind}}(B,B^+),\mrm{Spa}^{\mrm{ind}}\Hub{A})$ is in bijection with maps $\Hub{A}\to \Hub{B}$. Topological nilpotents map to $0$ in $B$ which proves the universal property. The claim of topological spaces is clear.
	For general pre-adic space $X$ we define $X^{\mrm{na}}$ to have underlying topological space $|X|^{\mrm{na}}$ and if $V\subseteq |X^{\mrm{na}}|$ is of the form $U\cap |X|^{\mrm{na}}$ for $U\subseteq |X|$ open and of the form $U=\mrm{Spa}^{\mrm{ind}}(A,A^+)$ we let $\cali{O}^{\mrm{ind}}_{X^{\mrm{na}}}(V):=\cali{O}^{\mrm{ind}}_X(U)/A^{\circ \circ}\cdot \cali{O}^{\mrm{ind}}_X(U)$. Since the construction $A\mapsto A/A^{\circ \circ}$ is compatible with rational localization $\cali{O}^{\mrm{ind}}_{X^{\mrm{na}}}(V)$ is well-defined and glues to a sheaf of ind-topological rings on $X^{\mrm{na}}$. Moreover, locally the ind-topological rings come from a topological ring because $A/A^{\circ \circ}$ is sheafy. This implies $X^{\mrm{na}}$ is an adic space.  
\end{proof}

Recall that to any pre-adic space $X$ over $\Zp$ one can associate a small v-sheaf $X^\dia$ over $\mrm{Spd}(\mathbb{Z}_p)$. 
This is done by letting $\mrm{Spd}(\mathbb{Z}_p)(Y)=\{(Y^\sharp,\iota)\}/_{\cong}$ 
and letting $X^\dia (Y)=\{(Y^\sharp,\iota,f)\}/_{\cong}$, where $Y^\sharp\in \mrm{Perfd}$, $\iota: (Y^\sharp)^\flat\to Y$ is an isomorphism, and $f:Y^\sharp\to X$ is a morphism of pre-adic spaces.

\begin{pro}\textup{(\cite[Lemma 18.1.1]{Ber})}
	\label{pro:Zpdvsheaf}
	For any pre-adic space $X$ over $\Zp$ (not necessarily analytic), the presheaf $X^\diamondsuit$ is a small v-sheaf. 
\end{pro}
From now on, given a Huber pair $\Hub{A}$ we denote $(\Spa{A}^{\mrm{ind}})^\diamondsuit$ by $\Spd{A}$. 
If $\Hub{R}$ is a Huber pair for which $R^+=R^\circ$ we will abbreviate $\Spa{R}$ and $\Spd{R}$ by $\Spf{R}$ and $\Spdf{R}$. For example, $\Fpd$, $\Zpd$, $\Qpd$. Given an $I$-adic ring $R$ with $I\subseteq R$ finitely generated, we will say that $\Huf{R}$ is a ``\textit{formal}" Huber pair. 

\begin{prop} 
	\label{rem:diamondconstr}
We collect some facts about $\diamondsuit$, that are either in the literature or not hard to prove. Let $\mrm{PreAd}_\Zp$ denote the category of pre-adic spaces over $\Zp$ and let $X\in \mrm{PreAd}_\Zp$.
\begin{enumerate}
	\item If $X$ is perfectoid, then $X^\diamondsuit\cong h_{X^\flat}$ \cite[Lemma 15.2]{Et}.
	\item There is a surjective map of topological spaces $|X^\diamondsuit|\to |X|$ \cite[Proposition 18.2.2]{Ber}. \label{mapofadictovsheaf}
	\item If $X$ is analytic, then $X^\diamondsuit$ is a locally spatial diamond and $|X^\diamondsuit|\cong|X|$, \cite[Lemma 15.6]{Et}.
	\item The functor $\dia:\mrm{PreAd}_\Zp\to \topPerf$ commutes with limits and colimits. More precisely, if $X_i$ is a family of pre-adic spaces and $\varinjlim_{i\in I} X_i$ (respectively $\varprojlim_{i\in I} X_i$) is represented by a pre-adic space $X$ then $X^\dia=\varinjlim_{i\in I} X_i^\dia$ (respectively $X^\dia=\varprojlim_{i\in I} X_i^\dia$). 
	 \item The structure map $\Spd{B}\to \Zpd$ is separated. 
	 \item The map $(X^{\mrm{na}})^\dia\to X^\dia$ is a closed immersion and $X^\dia\setminus (X^{\mrm{na}})^\dia=(X^{\mrm{an}})^\dia$. 
\end{enumerate}
\end{prop}

\section{The olivine spectrum}
As we will see below, the map $|X^\diamondsuit|\to |X|$ of \cref{mapofadictovsheaf} in \Cref{rem:diamondconstr} is usually not injective when $X$ has non-analytic points. Although the map is always surjective, it might not be a quotient map in pathological and drastic non-Noetherian situations. To develop a theory of specialization maps for v-sheaves, we need better understanding of $|\Spdf{A}|$ when $A$ is an $I$-adic ring over $\Zp$. To tackle this difficulty, we introduce what we call the \textit{olivine spectrum} of a Huber pair. It is a small variation of Huber's adic spectrum with a diamond-like twist. Under some mild “finiteness” conditions we prove that the olivine spectrum recovers $|\Spd{B}|$. 
For the rest of the section we fix $\Hub{B}$ a Huber pair (not necessarily over $\Zp$ and not necessarily complete).
\subsection{Review, terminology and conventions}
We assume familiarity with the construction of Huber's adic spectrum, $\Spa{B}$, but we review some definitions, some key facts, and we fix some terminology. Let $x\in \Spa{B}$ and fix a representative $|\cdot|_x:B\to \Gamma_x\cup \{0\}$.
	\begin{enumerate}
		\item The support $\mathbf{supp}(x)\subseteq B$ is the prime ideal of $b\in B$ with $|b|_x=0$. 
		\item We say $x$ is non-analytic if $\mathbf{supp}(x)$ is open in $B$, we say it is analytic otherwise.
		\item Let $H\subseteq \Gamma_x$ be a convex subgroup. We let $|\cdot|_y:B\to (\Gamma_x/H)\cup \{0\}$ with $|b|_y=|b|_x+H\in \Gamma_x/H$ when $|b|_x\neq 0$ and $|b|_y=0$ when $|b|_x=0$. Equivalence classes of valuations constructed this way are called a \textit{vertical generizations} of $x$.
		\item There is a residue field map of complete Huber pairs $\iota^*_x:\Hub{B}\to \Hub{K_x}$, where $K_x$ is either a discrete field or a complete nonarchimedean field. In both cases, $K_x^+$ is an open and bounded valuation subring of $K_x$. The induced map $\iota_x:\Spa{K_x}\to \Spa{B}$ is a homeomorphism onto the subspace of $\Spa{B}$ of continuous vertical generizations of $x$ and satisfies the universal property of maps that factor through this locus. 
		\item Residue fields relate to vertical generizations as follows. Let $K_x^\circ$ be the subring of power-bounded elements in $K_x$. If $y$ is a continuous vertical generizations of $x$ we let $K^+_y=\{b\in K_x\mid |b|_y\leq 1\}$. This gives a bijection between the set of continuous vertical generizations of $x$ and valuation rings of $K_x$ with $K_x^\circ\supseteq K_y^+\supseteq K^+_x$. Moreover, the residue field at $y$ is $(K_x,K_y^+)$.
		\item We say $x$ is \textit{trivial} if $\Gamma_x=\{1\}$. In this case, $K_x$ is discrete.
		\item We say that a valuation is \textit{microbial} if it has a non-trivial rank $1$ vertical generization. 
		\item For technical reasons we take the convention that trivial valuations have rank $0$. 
		\item The \textit{characteristic subgroup} of $|\cdot|_x$, denoted by $c\Gamma_x$, is the smallest convex subgroup of $\Gamma_x$ containing $\gamma=|b|_x$ for all $b\in B$ with $1\leq\gamma$.
		\item Given a convex subgroup $H\subseteq \Gamma_x$ containing $c\Gamma_x$, we define $|\cdot|_y:B\to H\cup \{0\}$ with $|b|_y=|b|_x$ if $|b|_x\in H$ and $|b|_y=0$ otherwise. Equivalence classes of valuations constructed in this way are called \textit{horizontal specializations} of $x$.
		\item Residue fields relate to horizontal specializations as follows. Let $K_B$ be the subring of $K_x$ generated by $K_x^+$ and the image of $B$ in $K_x$. Consider the induced map $f:\mrm{Spec}(K_B)\to \mrm{Spec}(B)$. Horizontal specializations of $x$ are in bijection with prime ideals of $B$ that are in the image of $f$. For a convex subgroup $H$ containing $c\Gamma_x$ and inducing $y$, the associated prime ideal $\mathfrak{p}_y=\{b\in B\mid |b|_x< \gamma\,\text{ for }\gamma\in H\}$. We sometimes denote $|\cdot|_y$ by $|\cdot|_x/\mathfrak{p}_y$. \label{horizontalspecializandvals}
			
		\item Given a topological space $T$ we construct a partial order on $T$ by letting $t_1\preceq_T t_2$ if $t_1\in \overline{\{t_2\}}$. We call this partial order the \textit{generization pattern} of $T$. We use $\preceq_B$ instead when $T=\Spa{B}$.
		\item The generization pattern of $\Spa{B}$ is determined by vertical generizations and horizontal specializations. More precisely, letting $(y,z)\in R$ if $z$ is a vertical generization of $y$ or if $y$ is horizontal specialization of $z$. Then $\preceq_{{B}}$ is the transitive closure of $R$.
	\end{enumerate}

\subsection{Definitions and basic properties}
\begin{defi}
	We define a topological space $\Spo{B}$ which we call the \textit{olivine spectrum} of $B$.
	\begin{enumerate}
		\item Let $\Spo{B}\subseteq \Spa{B}^2$ consist of pairs, $x:=(|\cdot|_x^h,|\cdot|_x^a)$, such that $|\cdot|_x^a$ has rank $1$ or $0$ and is a vertical generization of $|\cdot|_x^h$. 
		\item Pick $b_1,b_2\in B$ and let $\U{b_1}{b_2}=\{x\in \Spo{B}\mid \, |b_1|_x^h\leq |b_2|_x^h\neq 0\}$, we call such subsets classical localizations. 
		\item Pick $b_1,b_2\in B$ and let $\N{b_1}{b_2}=\{x\in \Spo{B}\mid \, |b_1|_x^a< |b_2|_x^a\neq 0\}$, we call such subsets analytic localizations. 
		\item We endow $\Spo{B}$ with the topology generated by classical and analytic localizations. 
	\end{enumerate}
\end{defi}
We denote by $h:\Spo{B}\to \Spa{B}$ the first coordinate projection, this map is continuous and surjective. Moreover, both $\mrm{Spo}$ and $h$ are functorial. 
\begin{defi}
	Let $x\in \Spo{B}$. 
	\begin{enumerate}
		\item We say that $x$ is \textit{discrete} if $|\cdot|_x^a$ is trivial. We say that a discrete point is \textit{microbial} if $h(x)$ is microbial. We say that a discrete point is \textit{algebraic} if $|\cdot|_x^h$ is trivial.
		\item We say that $x$ is \textit{d-analytic} if $|\cdot|_x^a$ is non-trivial. Suppose that $x$ is d-analytic, we say that it is \textit{analytic} if $h(x)$ is analytic and we say it is \textit{meromorphic} otherwise.
		\item We say that $x$ is \textit{bounded} if $|B|_x^a\leq 1$.
		\item We say that $x$ is \textit{formal} if it is bounded and d-analytic.
	\end{enumerate}
\end{defi}
For $x\in \Spo{B}$ the set $h^{-1}(h(x))$ has at most one d-analytic point and at most one discrete point. Consequently, $\mathbf{Card}(h^{-1}(h(x)))\in \{1,2\}$. Moreover, $\mathbf{Card}(h^{-1}(h(x)))=2$ if and only if $h(x)$ is discrete and $h(x)$ is microbial. The definitions are made so that $x$ is analytic if and only if $h(x)$ is, and in this way we can talk about the analytic locus. Nevertheless, with our terminology, there is no longer a dichotomy since meromorphic points are not analytic but they are also not discrete.  

We define the \textbf{bounded locus}, and denote it $\Spo{B}^\dagger\subseteq \Spo{B}$, as the subset of bounded points. This is a closed subset with complement of $\cup_{b\in B} \N{1}{b}$.
\begin{rem}
	Let us comment on the terminology chosen. By construction, the olivine spectrum has more points than Huber's adic spectrum. Algebraic points of $\Spo{B}$ are in bijection with the usual Zariski spectrum of $B/B^{\circ \circ}$. Discrete points are in bijection with the non-analytic points of Huber. Later on we will realize that when $(B,B^+)$ is defined over $\mathbb{Z}_p$ the d-analytic points of $\Spo{B}$ correspond to those points whose residue field is a diamond. Among d-analytic points only those that are analytic are in bijection with the analytic points of Huber. The terms formal and meromorphic stem from \Cref{generization-names}. The term bounded stems from the fact that the bounded locus on $\mathrm{Spd}(\mathbb{F}_p[t],\mathbb{F}_p)$ agrees with the functor sending a pair $(R,R^+)$ to the set of power-bounded elements $R^\circ$.
\end{rem}


\begin{defi}
	Let $x\in \Spo{B}$, we let $\mathbf{supp}(x):=\mathbf{supp}(h(x))$, this is the \textit{support ideal}. We let $\mathbf{sp}(x)=\{b\in B^+\mid |b|_x^h<1\}$, this is the \textit{specialization ideal}. If $x$ is bounded, we let $\mathbf{def}(x)=\{b\in B\mid |b|_x^a<1\}$, this is the \textit{deformation ideal}. 
\end{defi}

\begin{rem}
The specialization ideal will be key for us later when we discuss the specialization map for specializing v-sheaves. In contrast to Huber's theory, the discrete point that one can construct by killing the elements of the specialization ideal does not always lie in the topological closure of our original point. For this reason we also have to consider what we call the deformation ideal. 	
\end{rem}

Notice that $x$ is bounded if and only if $c\Gamma_{x^a}=\{1\}$, this only happens if $x$ is either discrete or formal. If $x$ is bounded, it is discrete whenever $\mathbf{supp}(x)=\mathbf{def}(x)$ and it is formal otherwise.

\begin{defi}
	\label{generization-names}
	Let $x$ and $y$ be two points in $\Spo{B}$.
	\begin{enumerate}
		\item  $y$ is a \textit{vertical generization} of $x$ ($x$ a \textit{vertical specialization} of $y$ respectively) if $|\cdot|_x^a=|\cdot|_y^a$ and $|\cdot|_y^h$ is a vertical generization of $|\cdot|_x^h$ in $\Spa{B}$. We abbreviate this as $y$ is v.g. of $x$ ($x$ is v.s. of $y$ respectively).
		\item $y$ is a \textit{meromorphic generization} of $x$ ($x$ a \textit{meromorphic specialization} of $y$ respectively) if $y$ is meromorphic, $x$ is discrete and $h(x)=h(y)$. We abbreviate this as $y$ is m.g. of $x$ ($x$ is m.s. of $y$ respectively).
		\item $y$ is a \textit{formal generization} of $x$ ($x$ a \textit{formal specialization} of $y$ respectively) if $y$ is formal, $x$ is discrete $\mathbf{def}(y)=\mathbf{supp}(x)$ and $|\cdot|^h_x=|\cdot|^h_y/\mathbf{def}(y)$. We abbreviate this as $y$ is f.g. of $x$ ($x$ is f.s. of $y$ respectively).
	\end{enumerate}
\end{defi}

\begin{rem}
In Huber's theory there are two distinguished types of specialization, namely vertical specializations and horizontal specializations. We consider three distinguished types of specialization. 
The vertical specializations we consider arise in the same way as Huber's vertical generizations and have the same behavior. 
In contrast, Huber's horizontal specializations are replaced by meromorphic and formal specializations. 
In very rough terms, a formal generization is what you obtain when you replace the equation $b=0$ by asking instead the condition that $b$ is a topologically nilpotent unit. 
Analogously, a meromorphic specialization is what you obtain when you replace the condition that $b$ is a topologically nilpotent unit by the condition that $|b|<1$ but for all $\epsilon\in (0,1)$ $\epsilon<|b|$. 
One can think of the locus $\{b=0\}$ as one discrete end, the locus $\{1>b>\epsilon\}$ as the opposite discrete end, and the locus where $b$ is a topologically nilpotent unit as the analytic in-between that specializes to both ends through the formal specialization and meromorphic specialization respectively. 
\end{rem}
	
Given $x\in \Spo{B}$ let $\I{x}$ denote the set of generizations of $x$ in $\Spo{B}$ and let $\Iv{x}$ denote the set of vertical generizations of $x$. If the context is clear, for a point $y\in \Spa{B}$ we will also use $\Iv{y}$ to denote the vertical generizations of $y$ in $\Spa{B}$. Let us make some easy observations and set some convenient notation: 
	\begin{enumerate}
		\item If $x$ is discrete it has a meromorphic generization (necessarily unique) if and only if $x$ is microbial. We denote this generization by $x^{\mrm{mer}}$.
		\item If $x$ is meromorphic it has a unique meromorphic specialization, we denote it by $x_{\mrm{mer}}$.
		\item If $x$ is formal it has a unique formal specialization, we denote it by $x_{\mrm{for}}$. If $x$ is discrete, we let $x^{\mrm{For}}$ denote the set of formal generizations of $x$.
	\end{enumerate}

We recommend the reader to work through the following example confirming all of the claims.

\begin{exa}
	\label{exa:thebasiccase}
	Let $B=\Fp\pot{u}$ endowed with the discrete topology, then $\Spf{B}$ consists of $3$ points: $$\Big\{\eta=|\cdot|_\eta,\,s=|\cdot|_s,\,t=|\cdot|_{t}\Big\}$$
	
	Here $|\cdot|_\eta$ is the trivial valuation with residue field $\Fp\left( \left( u \right) \right)$, $|\cdot|_s$ is the trivial valuation with residue field $\Fp$ and $|\cdot|_t$ is the $(u)$-adic valuation on $\Fp\pot{u}$ with residue affinoid field $(\Fp\left( \left( u \right) \right),\Fp\left[ \left[ u \right] \right])$. All valuations have rank $1$ or $0$. The only non-trivial vertical generization in $\Spf{B}$ goes from $|\cdot|_t$ to $|\cdot|_\eta$. 
	
	On the other hand, $Spo(B)$ has $4$ points: 
	$$\Big\{\eta:=(|\cdot|_\eta,|\cdot|_\eta),\,s:=(|\cdot|_s,|\cdot|_s),\,t^{\mrm{f}}:=(|\cdot|_t,|\cdot|_t),\,t^{\mrm{d}}:=(|\cdot|_t,|\cdot|_\eta)\Big\}$$
	Now, $\{\eta\}=\U{1}{u}$, $\{\eta,t^{\mrm{d}},t^{\mrm{f}}\}=\U{0}{u}$, $\{t^{\mrm{f}}\}=\N{u^2}{u}$ and $\{t^{\mrm{f}},s\}=\N{u}{1}$, and these are the only proper open subsets.
	Here $s$, $\eta$ and $t^{\mrm{d}}$ are discrete. Moreover, $t^{\mrm{d}}$ is microbial, and $t^{\mrm{f}}$ is both a meromorphic and formal d-analytic point. 

\begin{figure}[h!]
	\label{thedrawing}
    \centering
    \includegraphics[width=0.6\textwidth]{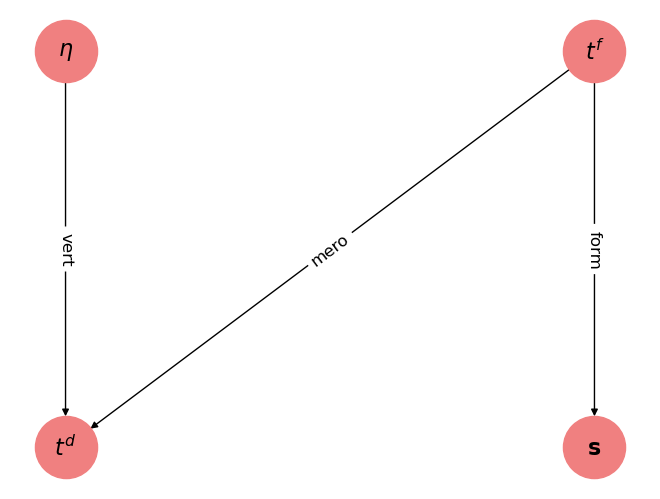}
    \caption{Generization pattern of $\mathrm{Spo}(\mathbb{F}_p\pot{u},\mathbb{F}_p\pot{u})$}
    \label{fig:graph}
\end{figure}

	The generization pattern is: $\eta$ is a vertical generization of $t^{\mrm{d}}$, $t^{\mrm{d}}$ is the meromorphic specialization of $t^{\mrm{f}}$, and $s$ is the formal specialization of $t^{\mrm{f}}$. We have $\mrm{Spo}(\Fp\pot{u})^\dagger=\mrm{Spo}(\Fp\pot{u})$.
\end{exa}

The following shows that the v.g., f.s. and m.s. determine the generization pattern in $\Spo{B}$.
\begin{pro}
	\label{pro:generizationpattern}
	Let $x\in \Spo{B}$.
	\begin{enumerate}
		\item If $x$ is d-analytic then $\I{x}=\Iv{x}$.
		\item If $x$ is discrete then $\I{x}=\Iv{x}\cup \Iv{x^{\mrm{mer}}} \cup (\bigcup_{z\in \Iv{x}} {z^{\mrm{For}}})$.
	\end{enumerate}
\end{pro}
\begin{proof}
	We prove the right to left side. Let $y\in \Iv{x}\cup \Iv{x^{\mrm{mer}}} \cup (\bigcup_{z\in \Iv{x}} {z^{\mrm{For}}})$ if $x$ is discrete and let $y\in \Iv{x}$ otherwise. Since $h$ is continuous, 
	$y$ is in every classical localization of $x$, so it suffices to check analytic localizations. Suppose $x\in \N{b_1}{b_2}$, if $y$ is a v.g. of $x$ then $|\cdot|_y^a=|\cdot|_x^a$ so $y\in \N{b_1}{b_2}$. If $x$ is discrete, then $|b_1|_x^a=0$ and $|b_2|_x^a=1$, this implies $|b_1|_{x^{\mrm{mer}}}^a=0$ and that $|b_2|_{x^{\mrm{mer}}}^a\neq 0$, so $x^{\mrm{mer}}\in \N{b_1}{b_2}$ in case $x^{\mrm{mer}}$ exists. Moreover, if $y\in x^{\mrm{For}}$ then $\mathbf{def}(y)=\mathbf{supp}(x)$ so that $|b_1|_y^a<1$, $|b_2|_y^a=1$, and $x^{\mrm{For}}\in \N{b_1}{b_2}$.

	We prove the left to right side, let $y\in \I{x}$. With classical localizations we deduce $\mathbf{supp}(y)\subseteq \mathbf{supp}(x)$, and if $x$ is d-analytic we claim that $\mathbf{supp}(y)=\mathbf{supp}(x)$. Indeed, let $b\in B$ such that $|b|_x^a\notin \{0,1\}$, and let $b_1\in \mathbf{supp}(x)$. If $|b|_x^a<1$ then $|b|_y^a<1$, which implies that $y$ is d-analytic. Additionally, the inequalities $|b_1|_y^a<|b^n|_y^a$ hold for all $n$ since $x\in \N{b_1}{b^n}$. Similarly, if $1<|b|_x^a$ then $1<|b|_y^a$ and $|b_1\cdot b^n|_y^a<|b|_y^a$ hold instead. In both cases, the archimedean property of rank $1$ valuations prove $b_1\in \mathbf{supp}(y)$. Since the only generizations of $h(x)$ in $\Spa{B}$ with the same support are v.g. we get $h(y)\in \Iv{h(x)}$ and $y\in \Iv{x}$ for $x$ d-analytic. 

	Suppose $x$ is discrete, if $\mathbf{supp}(x)=\mathbf{supp}(y)$ we can reason as above. Let $b\in \mathbf{supp}(x)\setminus \mathbf{supp}(y)$. Since $x\in \N{b}{1}$ we have $0<|b|_y^a<1$ and that $y$ is d-analytic. For all $b_1\in B$ $|b\cdot b_1^n|_y^a<1$ holds and $y$ is formal with $\mathbf{supp}(x)\subseteq \mathbf{def}(y)$. If $b_2\notin \mathbf{supp}(x)$ then $x\in \U{b}{b_2^n}$ for all $n$, giving $|b_2|_y^a=1$ and $\mathbf{def}(y)=\mathbf{supp}(x)$. Let $z=y_{\mrm{for}}$ then $\mathbf{supp}(z)=\mathbf{supp}(x)$ and it follows from the construction of horizontal specializations that $z\in \I{x}$. As above, $h(z)$ is a v.g. of $h(x)$, and since both $z$ and $x$ are discrete then $z$ is a v.g. of $x$. In other words, $z\in \Iv{x}$ and $y\in z^{\mrm{For}}$.
\end{proof}
The olivine spectrum is compatible with completion and rational localization. 
\begin{pro}
	\label{pro:completionworks}
	If $\Hub{\hat{B}}$ denotes the completion of $\Hub{B}$, then $\Spo{\hat{B}}=\Spo{B}$.
\end{pro}
\begin{proof}
	Since $\Spa{\hat{B}}=\Spa{B}$ the map $\Spo{\hat{B}}\to \Spo{B}$ is bijective and classical localizations of $\Spa{\hat{B}}$ are open in $\Spa{B}$. It suffices to prove $\N{g}{f}$ is open in $\Spo{B}$ for $f,g\in \hat{B}$. Let $x\in \N{g}{f}$ and $f_x\in B$ with $|f_x|_x^h=|f|_x^h$. We have $\U{f_x}{f}\cap\U{f}{f_x}\cap \N{g}{f}=\U{f_x}{f}\cap\U{f}{f_x}\cap \N{g}{f_x}$, so we may assume $f\in B$. Take a ring of definition $B_0\subseteq B$ and an ideal of definition $I\subseteq B_0$ with $|i_k|_x^h\leq |f|_x^h$ for a finite set of generators $\{i_1\dots i_m\}\subseteq I$. Let $g_x\in B$ such that $g-g_x\in I^2\cdot \hat{B}_0$. Then $(\bigcap_{i} \U{i}{f})\cap \N{g_x}{f}=(\bigcap_{i} \U{i}{f})\cap \N{g}{f}$ so the left hand side is open in $\Spo{B}$. 
\end{proof}
\begin{pro}
	\label{pro:localizationworks}
	Let $s,t_1,\dots,t_n\in B$ defining a rational localization $\Spa{R}:=U(\frac{t_1,\dots,t_n}{s})\subseteq \Spa{B}$. 
	Then $\Spo{R}\to \Spo{B}$ is a homeomorphism onto $h^{-1}(U(\frac{t_1,\dots,t_n}{s}))$.
\end{pro}
\begin{proof}
	It suffices to check $\N{r_1}{r_2}\subseteq \Spo{R}$ is open in $\Spo{B}$ for $r_1,r_2\in R$. By \Cref{pro:completionworks} and the construction of rational localizations we may assume $r_1,r_2\in B[\frac{1}{s}]\subseteq R$. 
	Write $r_1=\frac{b_1}{s^{n_1}}$, $r_2=\frac{b_2}{s^{n_2}}$ and let $m=n_1-n_2$. Then $\N{r_1}{r_2}=\N{b_1}{b_2\cdot s^m}\cap \Spo{R}$ when $m\geq 0$ and $\N{r_1}{r_2}=\N{b_1\cdot s^m}{b_2}\cap \Spo{R}$ otherwise.  
\end{proof}
The following example is key to the prove \Cref{lem:sotechinicalitscrazy} and \Cref{thm2:non-analyticthm}. We encourage the reader to workout this example carefully. Recalling \Cref{exa:thebasiccase} under this light might be helpful. 
\begin{exa}
	\label{exa:valuationringsarecool}
	Suppose $B^+\subseteq B$ are valuation rings with $\mrm{Frac}(B)=\mrm{Frac}(B^+)$ both with the discrete topology. We describe $\Spo{B}$ in two steps: first observe that $\Spo{B}\subseteq \Spor{B^+}$ and that it acquires the subspace topology. 
	Indeed, this follows from \Cref{pro:localizationworks} and the identification $\Spo{B}=\cap_{b\in B\setminus B^+}\U{0}{b} \subseteq \Spor{B^+}$. 

Second, we describe $\Spor{B^+}$ explicitly using that all points are bounded and admit a deformation ideal.
Since $B^+$ is a valuation ring elements of $\mrm{Spa}(B^+,B^+)$ are determined by their support and specialization ideals. Consider the map $\Spor{B^+}\to \mrm{Spec}(B^+)^3$ with $q\mapsto (\mathbf{supp}(q),\mathbf{def}(q),\mathbf{sp}(q)).$ The following hold:
\begin{itemize}
	\item The map is injective with image those triples $(q_1,q_2,q_3)$ with $q_1\subseteq q_2\subseteq q_3$, and such that $[q_1,q_2]=\{q_1,q_2\}$ where the left term is an interval for the order defined by containment. 
	\item A triple $q=(q_1,q_2,q_3)$ is meromorphic if and only if $q_1\neq q_2$. In this case, there is $b\in B^+$ with $b\in q_2\setminus q_1$ and $q_2=\sqrt{(b)}$. 
	\item $q\in \Iv{r}$ if $q_1=r_1$, $q_2=r_2$ and $q_3\subseteq r_3$.  
	\item $q=r^{\mrm{mer}}$ if $r_1=q_1=r_2$, $q_2\neq q_1$ and $r_3=q_3$. 
	\item F.g. are unique and $r$ is the f.s. of $q$ if $q_1\neq q_2$, $r_1=q_2=r_2$ and $r_3=q_3$.

\end{itemize}
We describe the open subsets.
\begin{itemize}
	\item If $\frac{f}{g}\in B^+$ then $\U{f}{g}=\U{0}{g}$ and consists of triples $(q_1,q_2,q_3)$ with $g\notin q_1$. 
	\item If $\frac{g}{f}\in B^+$ we can let $b=\frac{g}{f}$ then $\U{f}{g}=\U{1}{b}$ and it consists of triples with $b\notin q_3$. 
\item The families $\{\U{1}{b}\}_{b\in B^+}$ and $\{\U{0}{g}\}_{b\in B^+}$ are nested. In particular, finite intersections of classical localizations have the form $\U{1}{b}\cap \U{0}{g}$ for some $b,g\in B^+$.
	\item When $n=\frac{f}{g}\in B^+$ then $\N{g}{f}$ is empty and $\N{f}{g}=\U{0}{g}\cap \N{n}{1}$. 
	\item The set $\N{n}{1}$ consists of the triples $q=(q_1,q_2,q_3)$ such that $n\in q_2$. 
	\item The family of sets $\{\N{n}{1}\}_{n\in B^+}$ is nested. 
\end{itemize}
In summary, if $x\in U\subseteq \Spor{B^+}$ for $U$ an open subset there are elements $g,b,n\in B^+$ with $x\in \U{0}{g}\cap \U{1}{b}\cap \N{n}{1}\subseteq U$. Moreover, elements of $\U{0}{g}\cap \U{1}{b}\cap \N{n}{1}\subseteq U$ are explicitly described by the constraints: $g\notin q_1$, $b\notin q_3$ and $n\in q_2$.  
\end{exa}

\subsection{Olivine Huber pairs}
For the rest of the section $\Hub{B}$ denotes a complete Huber pair over $\Zp$. 
\begin{pro}
	\label{pro:Tatesaresupereasy}
	If $R$ is a Tate Huber pair, then $h:\Spo{R}\to \Spa{R}$ is a homeomorphism. 
\end{pro}
\begin{proof}
	Since $\Hub{R}$ is Tate, $\Spa{R}$ has no trivial continuous valuations and $h$ is injective. If $x^a$ is the maximal generization of $x$ in $\Spa{R}$ then $h^{-1}(x)=\{(x,x^a)\}$. It suffices to prove that $h(\N{r_1}{r_2})$ is open. But if $\varpi\in R$ is a topologically nilpotent unit, then $h(\N{r_1}{r_2})=\bigcup_{0<n} \{z\in \Spa{R}\mid |r_1^n|_z\leq |r_2^n \varpi|_z\neq 0\}$.
\end{proof}

We define a canonical map $\pi:|\Spd{B}|\to \Spo{B}$ as follows.\footnote{The topological considerations in what follows can be done purely in the context of adic spaces without the theory of perfectoid spaces. To do this one substitutes $|\Spd{B}|$ by $\Spo{B}'$ where this second space has $\Spo{B}$ as underlying set, but has the strongest topology making maps coming from Tate Huber pairs continuous.}
Given $[x] \in |\Spd{B}|$ represented by a map $x : \mathrm{Spa}(C_x^\sharp , C_x^{\sharp,+}) \to \Spa{B}$, we can define $\pi([x]) = (x(s), x(\eta)) \in \mathrm{Spa}(B, B^+ )^2$ where $s \in \mathrm{Spa}(C_x^\sharp , C_x^{\sharp,+})$ is the closed point and $\eta \in \mathrm{Spa}(C_x^\sharp , C_x^{\sharp,+})$ is the unique rank one point. 
Then $x(\eta)$ has rank $\leq 1$ and is a vertical generization of $x(s)$, so $(x(s), x(\eta)) \in \Spo{B}$.


\begin{pro}
	\label{pro:rightpointsspo}
	The map $\pi:|\Spd{B}|\to \Spo{B}$ defined above is continuous and bijective.
\end{pro}
\begin{proof}
	Continuity follows from \Cref{defi:diamondtop} and \Cref{pro:Tatesaresupereasy}. For injectivity, take points $y_1,y_2:\Spa{C_i}\to \Spd{B}$ with $\pi(y_1)=\pi(y_2)=:x$. Let $(K_{h(x)},K^+_{h(x)})$ be the residue field $\Spa{B}$. The maps $\Hub{B}\to {(C^\sharp_i,C^{\sharp,+}_i)}$ factor through $\Hub{B}\to {(K_{h(x)},K^+_{h(x)})}$. Let $s_i$ be the closed point of $\Spa{C_i}$, we show that the $s_i$ define the same point. We split our analysis in three cases.

	\textit{Case 1:} Suppose that $x$ is analytic. In this case, $s_1$ and $s_2$ map to $h(\pi(x))$ in $\Spa{B}^{\mrm{an}}$. This case follows from the bijectivity of $|X^\dia|\to |X|$ for analytic pre-adic spaces (\Cref{rem:diamondconstr}).
 
 \textit{Case 2:} Suppose that $x$ is meromorphic, then $h(x)$ is non-analytic in $\Spa{B}$. Let $K_{h(x)}^\circ:=\{k\in K_{h(x)}\mid |k|_x^a\leq 1 \}$ since $|\cdot|_x^a$ is non-trivial $K_{h(x)}^\circ\neq K_{h(x)}$. Choose $b\in B$ with $0<|b|_x^a<1$ or $|b|_x^a>1$. The subspace topology of $(K^\circ_{h(x)})\subseteq_{y_i^*} O_{C^\sharp_i}$ is either the $(b)$-adic topology or the $(\frac{1}{b})$-adic topology. After taking completion we get a commutative diagram:
	\begin{center}
		\begin{tikzcd}
			& \Spa{C_1} \ar{d}{p'_1} \arrow[bend left=10]{ddr}{y_1} \\
			\Spa{C_2} \ar{r}{p'_2} \ar[bend right=10]{rrd}{y_2} & \Spd{\hat{K}_{h(x)}}\ar{rd}{\iota_x} \\
			& & \Spd{K_{h(x)}}
		\end{tikzcd}
	\end{center}
	Now, $p'_1(s_1)=p'_2(s_2)$ in $\Spd{\hat{K}_{h(x)}}$. Since $\Spa{\hat{K}_{h(x)}}$ is analytic we may conclude as in the first case.

	\textit{Case 3:} Suppose that $x$ is discrete, in this case $h(x)$ is non-analytic in $\Spa{B}$ and $\Hub{K_{h(x)}}$ has the discrete topology. Since $|\cdot|_x^a$ is trivial, $y^*_i(K_{h(x)})\subseteq O^{\sharp,\times}_{C_i}$. After choosing pseudo-uniformizers $\varpi_i\in O_{C^\sharp_i}$ we may extend the $y_i$ to continuous adic maps of topological rings $p'^*_i:K_{h(x)}\pot{t}\to O_{C^\sharp_i}$ where $K_{h(x)}\pot{t}$ has the $(t)$-adic topology. These induce the following commutative diagram:
	\begin{center}
		\begin{tikzcd}
			& \Spa{C_1} \ar{d}{p'_1} \arrow[bend left=10]{ddr}{y_1} \\
	\Spa{C_2} \ar{r}{p'_2} \ar[bend right=10]{rrd}{y_2} & \mrm{Spd}(K_{h(x)}\rpot{t},K_{h(x)}^++ t\cdot K_{h(x)}\pot{t})\ar{rd}{\iota_x} \\
			& & \Spd{K_{h(x)}}
		\end{tikzcd}
	\end{center}
	Again, $p'_1(s_1)=p_2'(s_2)$ in $\mrm{Spd}(K_{h(x)}\rpot{t},K_{h(x)}^++ t\cdot K_{h(x)}\pot{t})$ which is also analytic. 
	
	The case by case study given above also shows that $\pi$ is surjective. 
	Indeed, we can take a completed algebraic closures of $K_{h(x)}$ ($\hat{K}_{h(x)}$, or $ K_{h(x)}\rpot{t}$ respectively) when $x$ is analytic (meromorphic or discrete respectively).
\end{proof}
\begin{defi}
	Whenever $x$ is d-analytic we let $\Hub{K_x}$ denote $\Hub{\hat{K}_{h(x)}}$, and if $x$ is discrete we let $\Hub{K_x}$ denote $(K_{h(x)}\rpot{t}, K_{h(x)}^++t\cdot K_{h(x)}\pot{t})$ as in the proof of \Cref{pro:rightpointsspo}. In both cases we call $\Hub{K_x}$ the \textit{pseudo-residue field} at $x$.
\end{defi}
\begin{rem}
	The pseudo-residue field map $\Spo{K_x}\to \Spo{B}$ is a homeomorphism onto its image. The functor $\Spd{K_x}\to \Spd{B}$ surjects onto the subsheaf of $\Spd{B}$ consisting of maps that factor through $\Iv{x}$, but when $x$ is discrete the map $\Spd{K_x}\to \Spd{B}$ is not injective. Actually, when $x$ is discrete and $|\cdot|_x^h$ is non-trivial the subsheaf of points that factor through $\Iv{x}$ is not representable by an adic space. 
\end{rem}

\begin{cor}
	\label{cor:verticalgeneriz}
	For any map of Huber pairs $m^*:\Hub{B_1}\to \Hub{B_2}$ the map $\Spof{m}$ is compatible with v.g.. More precisely, if $x\in \Spo{B_2}$, $y=\Spof{m}(x)$ and $y'$ is a v.g. of $y$ then there exist $x'$, a v.g. of $x$, with $\Spof{m}(x')=y'$.
\end{cor}
\begin{proof}
	Given $x\in \Spo{B_2}$ and $y\in \Spo{B_1}$ as in the statement we may, after making some choices if necessary, construct the following commutative diagram of pseudo-residue fields:
	\begin{center}
		\begin{tikzcd}
			\Spd{K_x} \ar{r}\ar{d} & \Spd{K_y} \ar{d}\\
			\Spd{B_2}\ar{r} & \Spd{B_1}
		\end{tikzcd}
	\end{center}
	Since the map $\Spd{K_x}\to \Spd{K_y}$ is a map of locally spatial diamonds it is generalizing and consequently surjective. But $|\Spd{K_x}|=\Iv{x}$ and analogously for $y$.
\end{proof}

\begin{lem}
	\label{lem:samegenripattern}
The topological spaces $\Spo{B}$ and $|\Spd{B}|$ have the same generization pattern.
\end{lem}
\begin{proof}

	Since $|\Spd{B}|\to \Spo{B}$ is continuous the generization pattern of $|\Spd{B}|$ is smaller than that of $\Spo{B}$, it suffices by \Cref{pro:generizationpattern} to prove that formal, meromorphic and vertical specializations are specializations in $|\Spd{B}|$. For $x\in \Spo{B}$ the pseudo-residue field map $\iota_x:\Spd{{K}_x}\to \Spd{B}$ is a bijection onto $\Iv{x}$ so v.s. are specializations in $\Spd{B}$. 
	Let $x\in \Spo{B}$ be d-analytic and let $b$ such that $|b|_x^a\notin \{0,1\}$. Let $p:\Spa{C}\to \Spa{B}$ be a geometric point mapping to $x$ and let $\varpi\in C^{\circ \circ}$ be either $p^*(b)$ or $\frac{1}{p^*(b)}$. To this choice we will associate two product of points as follows. Let $R^+=\prod_{i=1}^\infty C^+$, let $\varpi_0=(\varpi^{\frac{1}{n}})_{n=1}^\infty$ and $\varpi_\infty = (\varpi^n)_{n=1}^\infty$. Let $R^+_0$ ($R^+_\infty$ respectively) be $R^+$ endowed with the $\varpi_0$-topology ($\varpi_\infty$-topology respectively), and let $R_0=R_0^+[\frac{1}{\varpi_0}]$ ($R_\infty=R_\infty^+[\frac{1}{\varpi_\infty}]$ respectively). We have diagonal maps of rings $C^+\to R^+_\infty$ and $C\to R_\infty$, but we warn the reader that these maps are not continuous. On the other hand, the map $C^+\to R^+_0$ is continuous but $\varpi$ is not invertible in $R_0$ so the map does not extend to a map $C\to R_0$. 

	Suppose that $x$ is meromorphic. 
	The diagonal map $f:B\to K_{h(x)}\to R_\infty$ becomes continuous giving a map $\Spa{R_\infty}\to \Spa{B}$. The space $\pi_0(|\Spa{R_\infty}|)$ is the Stone--\v{C}ech compactification of $\mathbb{N}$ whose elements are ultrafilters of $\bb{N}$. Principal ultrafilters $\{\mathcal{U}_n\}_{n\in \mathbb{N}}$ define inclusions $\iota_n:\Spa{C}\to \Spa{R_\infty}$ that correspond to the $n$th-projection in the coordinate rings. The closed point of a principal connected component maps to $x$ under $\Spof{f}$. We claim that the closed point of a non-principal connected component maps to $x_{\mrm{mer}}$. It suffices to construct a commutative diagram as below:
	\begin{center}
		\begin{tikzcd}
			\Spa{C_\mathcal{U}}\arrow{r} \ar{d} &\Spa{K_{x_{\mrm{mer}}}} \ar{d} \\
			\Spa{R_\infty} \ar{r} & \Spa{K_{h(x)}} \ar{r} & \Spa{B}
		\end{tikzcd}
	\end{center}
	We claim that the natural map $K_{h(x)}\to C_\mathcal{U}$ maps to $O_{C_\mathcal{U}}$. It suffices to prove $\varpi_\infty\cdot K_{h(x)}\subseteq  O_{C_\mathcal{U}}$, 
	and since $K_{h(x)}=K^+_{h(x)}[b,\frac{1}{b}]$  it suffices to prove that $\frac{\varpi_\infty}{\varpi^n}\in O_{C_\mathcal{U}}$ for $n\in \mathbb{N}$. Clearly $\frac{\varpi_\infty}{\varpi^n}\in \prod_{i=n+1}^\infty O_C$ and since the ultrafilter is non-principal complements of finite sets are in $\mathcal{U}$, which proves the claim. 
	
	By letting $t$ map to $\varpi_\infty$ we get a map $K_{h(x)}\rpot{t}\to C_\mathcal{U}$, the intersection of $K_{h(x)}\pot{t}$ with $C_\mathcal{U}^+$ in $O_{C_\mathcal{U}}$ is $K_{h(x)}^++t\cdot K_{h(x)}\pot{t}=K^+_{x_{\mrm{mer}}}$ which gives our factorization. The set of closed points contained in a principal component are dense within the set of closed points of $|\Spa{R_\infty}|$. This gives that m.s. in $\Spo{B}$ are specializations in $|\Spd{B}|$.

	Suppose that $x$ is formal. Since $|B|_x^a\leq 1$ the map $\Hub{B}\to \Hub{C}$ factors through $(O_C,C^+)$ and $\mathbf{def}(x)=B\cap C^{\circ \circ}$. This allows us to define a map $\Spa{R_0}\to \Spa{B}$. As above, we prove that principal components of $\pi_0(\Spa{R_0})$ map to $x$ in $\Spo{B}$ while the non-principal ones map to $x_{\mrm{for}}$, which proves that f.s. are specializations.
	Let $k=O_C/C^{\circ \circ}$ and $k^+=C^+/C^{\circ \circ}$, it suffices to prove that $(O_C,C^+)\to \Hub{C_\mathcal{U}}$ factors as: 
	$$(O_C,C^+)\to \Hub{k} \to (k\rpot{t},k^++t\cdot k\pot{ t})\to \Hub{C_\mathcal{U}}$$
	Now, $\frac{\varpi}{\varpi_0^n}\in \prod_{i=n+1}^\infty O_C$ which implies that $|\varpi|_{\mathcal{U}}\leq |\varpi_0^n|_\mathcal{U}$. Since $\varpi_0$ is a pseudo-uniformizer in $C_\mathcal{U}$ this implies $|\varpi|_\mathcal{U}=0$. Clearly $k\subseteq O_{C_\mathcal{U}}$ and we may send $t$ to $\varpi_0$ to construct our factorization. 
\end{proof}
\begin{pro}
	\label{pro:strongtopoonSpo}
	Let $\Huf{B}$ be a formal Huber pair then $|\Spdf{B}|\to \Spor{B}$ is a homeomorphism.

\end{pro}
\begin{proof}
	By \Cref{pro:rightpointsspo} the map is a continuous bijection. Let $Y=\Spf{B\pot{t}}^{t\neq0}$ and recall that $|Y|=|Y^\dia|$ since this is an analytic pre-adic space. Let $U$ be open in $|\Spdf{B}|$, let $x\in U$ and let $y\in Y$ mapping to $x$. We construct a neighborhood of $x$ in $U$ open in $\Spor{B}$. Let $f:(B,B)\to (B\pot{t},B\pot{t})$ be the canonical map.
	For $\U{b_1}{b_2}$ or $\N{b_1}{b_2}$ containing $x$ we choose quasicompact neighborhoods of $y$ in $\Spf{B\pot{t}}$, that we denote $U_{b_1,b_2,y}$ and $N_{b_1,b_2,y}$, whose image in $\Spor{B}$ are contained in $\U{b_1}{b_2}$ and $\N{b_1}{b_2}$ respectively. 
	For $\U{b_1}{b_2}$ pick a finite set $S\subseteq B$ and $n\in\mathbb{N}$ such that $|s|_y\leq |b_2|_y$ for $s\in S$, that $|t^{n}|_y\leq |b_2|_y$, and that the ideal generated by $S$ is open in $B$. 
	We let $U_{b_1,b_2,y}=U(\frac{S, t^{n},b_1}{b_2})\subseteq \Spf{B\pot{t}}$. 
	Rational localizations are quasicompact and clearly $\Spof{f}(h^{-1}(U_{b_1,b_2,y}))\subseteq \U{b_1}{b_2}$. For $\N{b_1}{b_2}$ pick a finite set $S$ and $n_1, n_2\in \mathbb{N}$, such that $|b_1^{n_1}|_y\leq |b_2^{n_1}\cdot t|_y$, that $|s|_y\leq |b_2^{n_1}\cdot t|_y$ for $s\in S$, that $|t^{n_2}|_y\leq |t\cdot b_2^{n_1}|_y$ and that $S$ generates an open ideal in $B$. We let $N_{b_1,b_2,y}=U(\frac{S, t^{n_2},b_1^{n_1}}{b_2^{n_1}\cdot t})$. Since $t$ is topologically nilpotent in ${B\pot{t}}$, if $z\in \Spf{B\pot{t}}$ then $|t|_z<1$ and $\Spof{f}(h^{-1}(N_{b_1,b_2,y}))\subseteq \N{b_1}{b_2}$. Notice that $N_{b_1,b_2,y}\subseteq \Spf{B\pot{t}}^{t\neq 0}$.  
	
	Let $X=(\bigcap N_{b_1,b_2,y})\cap (\bigcap U_{b_1,b_2,y})$, then $\Spof{f}(X)\subseteq \I{x}$ and by \Cref{lem:samegenripattern}, also $\Spof{f}(X)\subseteq U$. Now, $\Spof{f}^{-1}(U)$ is open in $\Spf{B\pot{t}}^{t\neq 0}$ and the families, $\{U_{b_1,b_2,y}\cap N_{0,1,y}\}$ and $\{N_{b_1,b_2,y}\}$, consist of quasicompact open subsets of $\Spf{B\pot{t}}^{t\neq0}$. A compactness argument in the patch topology of $\Spf{B\pot{t}}$ proves that a finite intersection is contained in $\Spof{f}^{-1}(U)$. We prove that the image under $\Spof{f}$ of such a finite intersection is open in $\Spor{B}$. 	
	More generally, let $Z=\cap_{i=1}^n V_i$ with $V_i$ of the form $V_{b_{i,1},b_{i,2}}:=\{z\in \Spf{B\pot{t}}^{t\neq 0}\mid \, |b_{i,1}|_z\leq |b_{i,2}|_z\neq 0 \}$ where $b_{i,1}\in B\cup \{t^n\}_{n\in \bb{N}}$ and $b_{i,2}\in B\cup t\cdot B$, we claim that $\Spof{f}(Z)$ is open in $\Spor{B}$. If $b_{i,1},b_{i,2}\in B$ then $V_{b_{i,1},b_{i,2}}=\Spof{f}^{-1}(\U{b_{i,1}}{b_{i,2}})$ and for $Z$ as above we have $\Spof{f}(Z\cap V_{b_{i,1},b_{i,2}})=\Spof{f}(Z)\cap \U{b_{i,1}}{b_{i,2}}$, so we can reduce to the case where each $V_i=V_{b_{i,1},b_{i,2}}$ satisfy that either $b_{i,1}\in \{t^n\}_{n\in \bb{N}}$ or $b_{i,2}=b_2\cdot t$. 
	Let $T_Z^n\subseteq B$ with $b\in T_Z^n$ if either $b_{i,1}=t^n$ and $b=b_{i,2}$ or if $b_{i,1}=t^{n+1}$ and $b_{i,2}=b\cdot t$ for some $i$. Let $T_Z^{\ll}\subseteq \in B\times B$ with $(b_1,b_2)\in T_Z^{\ll}$ if $(b_1,b_2)=(b_{i,1},b_{i,2})$ for some $i$, and let $T_Z^-$ and $T_Z^+$ denote the image of $T_Z^{\ll}$ under the first and second projection maps. We prove that $\Spof{f}(Z)$ is the intersection of all the sets of the form $\U{b_1^{n}}{b_2^{n}\cdot b_3}$ where $(b_1,b_2)\in T_Z^{\ll}$ and $b_3\in T_Z^n$ and all the sets of the form $\N{b_1}{b_2}$, with $(b_1,b_2)\in T_Z^{\ll}$, which proves $\Spof{f}(Z)$ is open. 
	
	It is not hard to see $\Spof{f}(Z)$ is contained in this intersection.  
	To prove the converse, let $w$ be in the intersection, we construct a lift in $Z$. Pick a point $q:\Spa{C}\to \Spf{B}$ over $w$, the choice of $\varpi\in C^{\circ \circ,\times}$ defines a lift of $q$ to $\Spa{C}\to \Spf{B\pot{t}}^{t\neq0}$. If $w$ is discrete then $|b_1|^a_w=0$ for every $b_1\in T_Z^-$ and $|b_2|_w^a=|b_3|_w^a=1$ for every $b_2\in T_Z^+$ and $b_3\in T_Z^n$. In this case, any choice of $\varpi$ defines a lift landing inside of $Z$. 
	If $w$ is d-analytic $\varpi$ must be chosen more carefully. Since $C$ is algebraically closed we may choose $n$th-roots of $(b_3)$ for all $b_3\in T_Z^n$. For a lift of $q$ to land in $Z$, $\varpi$ must satisfy the following:  $|\varpi|_q\leq |(b_3)^{\frac{1}{n}}|_q$ for all $b_3\in T_Z^n$ and $\frac{|(b_1)|_q}{|(b_2)|_q}\leq |\varpi|_q$ for all $(b_1,b_2)\in T_Z^{\ll}$. We let $m$ be the smallest of the values in $\Gamma_q$ of the form $|b_3^\frac{1}{n}|_q$ with $b_3\in T_Z^n$ and we let $M$ be the largest of the values of the form $|\frac{b_1}{b_2}|_q$ with $(b_1,b_2)\in T_Z^{\ll}$. Since $w\in \U{b_1^n}{b_2^n\cdot b_3}$ we have $M\leq m$. Since $w\in \N{b_1}{b_2}$ for all pairs $(b_1,b_2)\in T_Z^{\ll}$ we also have $M<1$. Any $\varpi\in C$ with $|\varpi|_q<1$ and $M\leq |\varpi|_q\leq m$ defines a lift of $q$ in $Z$. 
\end{proof}

\begin{defi}
	Let $\Hub{B}$ be a complete Huber pair over $\Zp$, we say that $\Hub{B}$ is \textit{olivine} if the map $|\Spd{B}|\to \Spo{B}$ is a homeomorphism.
\end{defi}
\begin{ques}
	\label{ques:isallolivine}
Is every complete Huber pair over $\Zp$ an olivine Huber pair?	
\end{ques}

We have enough partial progress answering this question. Although we do not know what to expect in full generality, for the Huber pairs that we consider this is true. Let us clarify. By \Cref{rem:diamondconstr} Tate Huber pairs are olivine. By \Cref{pro:strongtopoonSpo} formal Huber pairs are olivine. By \Cref{pro:localizationworks} if $\Hub{B}\to \Hub{R}$ induces a locally closed immersion $\Spd{R}\subseteq \Spd{B}$ and $\Hub{B}$ is olivine then $\Hub{R}$ is olivine. Moreover, being olivine can be verified locally in the analytic topology of $\Spa{B}$. The following criterion can be used in most circumstances of interest.  
\begin{pro}
	\label{pro:mainoreprop}
	\label{pro:quotienttopologyforgeneralB}
Let $\Hub{B}$ be a complete Huber pair over $\Zp$, suppose it is topologically of finite type over a formal Huber pair $(B_0,B_0)$. Then $\Hub{B}$ is olivine. 
\end{pro}
\begin{proof}
	By definition, there is $M=\{M_i\}_{i=1}^n$ with $B_0\cdot M_i\subseteq B_0$ open and a strict surjection $f:B_0\langle T_1\dots, T_n\rangle_{M_1,\dots, M_n}\to B$.
	Let $C$ be the ring of integral elements of $B_0\langle T_1\dots, T_n\rangle_{M_1,\dots, M_n}$, then $B^+$ is the integral closure of $f(C)$ in $B$. 
	Since $\Spd{B} \to \mathrm{Spd}(B_0\langle T_1\dots, T_n\rangle_{M_1,\dots, M_n},C)$ is a closed immersion it suffices to prove the claim for $(B_0\langle T_1\dots, T_n\rangle_{M_1,\dots, M_n},C)$. We proceed by induction the base case being \Cref{pro:strongtopoonSpo}. 
	Let $\Spa{R}$ be the rational localization corresponding to $\{x\in \Spa{B} \mid |T_1|_x\leq |1|_x\neq 0\}$, then $\Hub{R}$ is olivine by induction. Indeed, $\Hub{R}=(A_0\langle T_2,\dots,T_n\rangle_{M_2,\dots,M_n}, C')$ for $(A_0,A_0)=(B_0\langle T_1\rangle_{\{1\}},B_0\langle T_1\rangle_{\{1\}})$ which is a formal. Let $\Spa{S}=\{x \in \Spa{B}\mid |1|_x\leq |T_1|_x\neq 0\}$. If we let $A_0=B_0\langle \frac{1}{T_1}\rangle_{\{1\}}$ then we may rewrite $\Spa{S}$ as the locus of points in $$\mathrm{Spa}(A_0\langle T_2,\dots,T_n\rangle_{M_2,\dots,M_n},C'')$$ such that $m\leq \frac{1}{T_1}\neq 0$ for $m\in M_1$. By induction $(A_0\langle T_2,\dots,T_n\rangle_{M_2,\dots,M_n},C'')$ is olivine, and since rational localizations preserve olivine Huber pairs we conclude $\Hub{S}$ is olivine. 
\end{proof}

\begin{rem}
	For an arbitrary Huber pair $(B,B^+)$ with $B_0$ a ring of definition we can consider the commutative diagram
\begin{center}
	\begin{tikzcd}
		\mid \Spd{B}\mid \ar{r}\ar{d} & \Spo{B}\ar{d} \\
	\varprojlim_i	\mid \mrm{Spd}(B_i,B_i^+)\mid \ar{r} &  \varprojlim_i \mrm{Spo}(B_i,B_i^+) 
	\end{tikzcd}
\end{center}
where $(B_i,B_i^+)$ ranges over all subrings of $B$ that are topologically of finite type over $B_0$. By \Cref{pro:mainoreprop} the bottom horizontal arrow is a homeomorphism and one can verify directly that the right vertical arrow is also a homeomorphism. It is not clear to us if the left vertical arrow is a homeomorphism or not since taking limits of v-sheaf does not necessarily commute with taking underlying topological spaces. Adding to the complexity of the situation the transition maps $\mrm{Spd}(B_i,B_i^+)\to \mrm{Spd}(B_j,B_j^+)$ might not be quasicompact. Counterexample to \Cref{ques:isallolivine} should come from this failure. We do not know if letting $B=\Fp[T_1,\dots,T_n,\dots]$ and $B^+=\Fp$ with the discrete topology gives an olivine Huber pair.
\end{rem}

\subsection{Some open and closed subsheaves}
By \cite[Proposition 12.9]{Et} open subsets of $\Spo{B}$ define open subsheaves of $\Spd{B}$, and when $\Hub{B}$ is olivine this association is bijective. Since the formation of $\Spd{B}$ commutes with localization in $\Spa{B}$, one can compute the open subsheaf corresponding to classical localizations. The following lemma describes, in some cases, the open subsheaf associated to analytic localizations. 
\begin{lem}
	\label{lem:representinganalyticnbhoods}
	Suppose that $B^+$ is $I$-adic and that $B\subseteq \mrm{Frac}(B^+)$. Let $b\in B$, let $B^+_b$ be the $(b,I)$-adic completion of $B^+$ and let $B_b=B\otimes_{B^+}B^+_b$. 
	If $\Hub{B_b}$ is Huber, then, $\N{b}{1}\subseteq \Spd{B}$ is represented by $\Spd{B_b}$. This condition is satisfied if $B^+=B$ or if $B$ and $B^+$ are valuation rings.
\end{lem}
\begin{proof}
Since $B\subseteq B_b$ is dense,
	$\Spd{B_b}\to \Spd{B}$ is injective.
	If $f:\Spa{R}\to \Spd{B}$ factors through $\Spd{B_b}$ then $f^*(b)$ is topologically nilpotent in $R^\sharp$. This gives $\Spof{f}(\Spa{R})\subseteq \N{b}{1}$ and since this happens for all $\Hub{R}\in \mrm{Perf}$, $\Spd{B_b}\to \Spd{B}$ factors through $\N{b}{1}$. 
Conversely, pick $f:\mrm{Spa}(R^\sharp,R^{\sharp,+})\to \Spa{B}$ with $\Spof{f}(\mrm{Spa}(R^\sharp,R^{\sharp,+}))\subseteq \N{b}{1}$. Let $x\in \mrm{Spa}(R^\sharp,R^{\sharp,+})$, let $\varpi\in {R^{\sharp,+}}$ be a pseudo-uniformizer, then $|f^*b^n|_x\leq |\varpi|_x$ for some $n$. 
Let $\Spa{R_1}=U(\frac{f^*b^n}{\varpi})\subseteq \mrm{Spa}(R^\sharp,R^{\sharp,+})$. Now, $B^+\to R_1^+$ is continuous for the $(I,b)$-topology so we get a map $(B_b,B_b^+)\to (R_1,R_1^+)$. This proves $f$ factors locally, and by injectivity it also does globally.
%
\end{proof}
In general the subsheaf $\N{b}{1}$ is not of the form $\Spd{R}$. 
Recall that $\Spo{B}^\dagger\subseteq \Spo{B}$ is the closed subset of bounded points. Observe that $\Spo{B}^\dagger$ is stable under v.g. and by \Cref{pro:Joaoetal} it defines a closed subsheaf of $\Spd{B}$. Let $\Spd{B}^\dagger$ denote this closed subsheaf.
\begin{prop}
	Let $\F:\mrm{Perf}\to \mrm{Sets}$ parametrize triples $(R^\sharp,\iota,f)$ where $(R^\sharp,\iota)$ is an untilt of $R$ and $f:\mrm{Spa}(R^{\sharp,\circ},R^{\sharp,+})\to \mrm{Spa}(B,B^+)$ is a morphism of pre-adic spaces. Then $\F=\Spd{B}^\dagger$. 
\end{prop}
\begin{proof}
	We prove $\F \to \Spd{B}$ is a closed immersion. Let $\bb{A}_\Zp^{|B|}$ parametrize tuples $(R^{\sharp},\iota,x)$ with $(R^\sharp,\iota)$ an untilt and $x:B\to R^\sharp$ a map of sets. Define $\bb{A}_\Zp^{|B|,\dagger}$ similarly with $x:B\to R^{\sharp,\circ}$. We have a basechange identity $\F=\bb{A}_\Zp^{|B|,\dagger}\times_{\bb{A}_\Zp^{|B|}} \Spd{B}$. Since limits preserve closed immersions it suffices to prove $\bb{A}^{1,\dagger}_\Zp\to \bb{A}^1_\Zp$ is a closed immersion, which can be checked after basechange. Let $f_r:\Spa{R}\to \bb{A}^1_\Zp$ defined by $r\in R^\sharp$. Then $\bb{A}^{1,\dagger}\times_{\bb{A}^1_\Zp}\Spa{R}$ is the complement in $\mrm{Spa}(R,R^{\sharp,+})$ of $\bigcup_{\varpi\in R^{\sharp,\circ\circ}} \{x\in \mrm{Spa}(R^\sharp,R^{\sharp,+})\mid |{1}|_x\leq |r\cdot {\varpi}|_x\neq 0\}$. This is and stable under v.g. as we wanted to show.

	Since $\Spd{B}^\dagger$ and $\F$ are closed immersions it suffices to prove they coincide on geometric points. This follows from the definition of the bounded locus. 
\end{proof}


\begin{lem}
	\label{lem:boundedlocusgivesqcqs}
	Let $\Hub{A}$ and $\Hub{B}$ be complete Huber pairs over $\Zp$ and $\Hub{B}\to \Hub{A}$ be an adic morphism. Then $\Spd{A}^\dagger \to \Spd{B}^\dagger$ is representable in spatial diamonds. In particular, it is qcqs.
\end{lem}
\begin{proof}
	Since the map $\Hub{B}\to \Hub{A}$ is adic we can write $(A,A^+)$ as a (completion of a) filtered colimit $\varinjlim_{i\in I} \Hub{A_i}$ where each $\Hub{A_i}$ is topologically of finite type over $\Hub{B}$, and the transition maps realize $A_i\to A_j$ as a topological subring for $i<j$. One can see that $\Spd{A}^\dagger=\varprojlim_i \Spd{A_i}^\dagger$ and by \cite[Lemma 12.17]{Et} it suffices to prove that $\Spd{A_i}^\dagger\to \Spd{B}^\dagger$ is representable in spatial diamonds. 
	A presentation of $A_i$ as a topologically of finite type $B$-algebra gives a closed immersion $\Spd{A_i}^\dagger\to \mrm{Spd}(B\langle (T_k)_{k=1}^n\rangle_{M_k})^\dagger$. 
	Since closed immersions are representable in spatial diamonds we may assume $A_i=B\langle T_1\rangle_{M_1}$. There is an open immersion $\Spdf{B\langle T_1\rangle_{M_1}} \to \bb{A}^1_B$ and $\Spd{B\langle T_1\rangle_{M_1}}\cap \bb{A}^{1,\dagger}_B=\Spd{B\langle T_1\rangle_{M_1}}^\dagger$. Clearly, $\bb{A}_B^{1,\dagger}\to \Spd{B}^\dagger$ is representable in locally spatial diamonds, we need to verify it is quasicompact. This can be done after basechanges by affinoid perfectoid. But the basechange by a map $\Spa{R}\to \Spd{B}^\dagger$ is representable by $\mrm{Spd}(R^\sharp\langle T\rangle, R')$ where $R'$ is the minimal ring of integral elements containing $R^{\sharp,+}$. 
\end{proof}

The following statement says that at least the bounded locus of a Huber pair is always olivine. 

\begin{prop}
	\label{pro:bounededsialwaysolivine}
	Suppose that $\Hub{B}$ is a complete Huber pair over $\Zp$. The natural map $$|\Spd{B}^\dagger|\to \Spo{B}^\dagger$$ is a homeomorphism.
\end{prop}
\begin{proof}
	Let $B_0\subseteq B^+$ be a ring of definition and express $\Hub{B}$ as a filtered colimit $\varinjlim_{i\in J} (B_i,B_i^+)$ with both $B_i$ and $B_i^+$ of finite type over $B_0$, then $\Spd{B}^\dagger= \varprojlim \Spd{B_i}^\dagger$. By \Cref{pro:mainoreprop} each $\Hub{B_i}$ is olivine and by \Cref{lem:boundedlocusgivesqcqs} the transition maps are representable in spatial diamonds. Let $\pi_i:\bb{D}^{\times}_{B^\dagger_i}\to \Spd{B_i}^\dagger$ denote the punctured open unit disc over $\Spd{B_i}^\dagger$. Observe that $\pi_i$ is open. 
	Now, $\bb{D}^{\times}_{B_0}$ is a locally spatial diamond represented by $(\Spf{B_0\pot{t}}^{t\neq 0})^\dia$. In particular, $\bb{D}^{\times}_{B^\dagger_i}$ is also a locally spatial diamond and since the transition maps $\bb{D}^{\times}_{B^\dagger_i}\to \bb{D}^{\times}_{B^\dagger_j}$ are qcqs we see that by \cite[Lemma 12.17]{Et} $|\bb{D}^{\times}_{B^\dagger}|=\varprojlim |\bb{D}^{\times}_{B^\dagger_i}|$. It suffices to prove that $\pi: |\bb{D}^{\times}_{B^\dagger}| \to \Spo{B}^\dagger$ is a quotient map. Let $S\subseteq \Spo{B}^\dagger$ with $\pi^{-1}(S)$ open. For every point $y\in \pi^{-1}(S)$ there is an index $j_y\in J$ and an open subset of $U_y\subseteq  \bb{D}^{\times}_{B^\dagger_i}$ whose preimage in $\bb{D}^{\times}_{B^\dagger}$ is contained in $\pi^{-1}(S)$ and contains $y$. Now, $\pi_{j_y}(U_y)$ is open in $|\Spd{B_i}^\dagger|$ and since $\Hub{B_i}$ is olivine it is also open in $\Spo{B_i}^\dagger$. The preimage of $\pi_{j_y}(U_y)$ in $\Spo{B}^\dagger$ contains $\pi(y)$, is open and it is contained in $h^{-1}(S)$. 
\end{proof}

\subsection{Discrete Huber pairs in characteristic $p$}
For the rest of the subsection $A$ denotes a discrete perfect ring in characteristic $p$ and $A^+\subseteq A$ is integrally closed.
\begin{prop}
	\label{pro:quotientmapishfordiscrete}
	Let $\Hub{A}$ be as above. The projection map $\Spo{A}^\dagger \to \Spa{A}$ is surjective. Moreover, if $\Si\subseteq \Spa{A}$ is stable under arbitrary generization and $h^{-1}(\Si)$ is open in $\Spo{A}^\dagger$ then $\Si$ is open in $\Spa{A}$. 
\end{prop}
\begin{proof}
	The complement of the bounded locus consists of d-analytic points. Since $A$ has the discrete topology every d-analytic point is meromorphic. If $x\in \Spo{A}$ is meromorphic, then $y:=x_{\mrm{mer}}$ is bounded and satisfies $h(x)=h(y)$. Consequently, $h(\Spo{A})=h(\Spo{A}^\dagger)$. 
	
	Now, observe that $\mrm{Spd}(A\rpot{t},A^++t\cdot A\pot{t})\to \Spd{A}$ surjects onto $\Spd{A}^\dagger$ and represents the punctured open unit ball over it. Consider, $f:\mrm{Spa}(A\rpot{t},A^++t\cdot A\pot{t})\to \Spa{A}$, it suffices to prove that if $\Si$ is stable under generization and $f^{-1}(\Si)$ is open, then $\Si$ is open. The rest of the argument is a variant of the proof of \Cref{pro:strongtopoonSpo}, using only classical localizations. In this case, one exploits the constructible topology of $\mrm{Spa}(A\rpot{t},A^++t\cdot A\pot{t})$. We omit the details.
\end{proof}

\begin{prop}
	\label{prop:Valuationringsareolivine}
	If $A$ and $A^+$ are valuation rings with the same fraction field then $\Hub{A}$ is olivine. 	
\end{prop}
\begin{proof}
	If $\Spo{A}^\dagger=\Spo{A}$, then \Cref{pro:bounededsialwaysolivine} proves that $\Hub{A}$ is olivine. Suppose $x\in \Spo{A}\setminus \Spo{A}^\dagger$, then $x$ is meromorphic and there is $\pi\in A$ with $1<|\pi|^a_x$. We must have $\frac{1}{\pi}\in A^+$ since $A^+$ is a valuation ring and $\pi\notin A^+$. Let $b=\frac{1}{\pi}$, we claim that $A=A^+[\frac{1}{b}]$. By the archimedean property of $|\cdot|^a_x$ for every $a'\in A$ there is a big enough $n\in \bb{N}$ with $|b^n\cdot a' |_x^a<1$. Since $A^+$ is a valuation ring either ${a'}\cdot b^n\in A^+$ or $\frac{1}{a'\cdot b^n}\in A^+$, but the second case contradicts that $|\cdot|_x^a\in \Spa{A}$. 

	By \Cref{pro:strongtopoonSpo}, $\Huf{A^+}$ is olivine and since $\Spo{A}\subseteq \Spor{A^+}$ is the open locus in which $b\neq 0$ we conclude by \Cref{pro:localizationworks} that $\Hub{A}$ is also olivine.
\end{proof}

\begin{lem}
	\label{lem:redZp2}
	There is a unique map $\Spd{A}\to \Zpd$, it is given by $\Spd{A}\to \Fpd\to \Zpd$.
\end{lem}
\begin{proof}	
	It suffices to prove that $\Spa{C}\to \Spd{A}\to \Zpd$ factors through $\Fpd$ for geometric points. Consider, $\Spa{R_\infty}\to \Spd{A}$ as in the proof of \Cref{lem:samegenripattern}, with $R_\infty^+=\prod_{i=1}^\infty C^+$ and $\varpi_\infty=(\varpi^{p^i})$. The map $\Spa{R_\infty}\to \Spd{A}\to \Zpd$ defines an untilt of $R_\infty$ given by $\xi=p+(\varpi_\infty)^\frac{1}{p^k}\cdot \alpha$ with $\alpha\in W(R_\infty^+)$. For any $i\in \bb{N}$ the projection $\iota_i:R_\infty \to C$ gives an untilt of $C$. Since $\Hub{A}\to \Hub{R_\infty} \xrightarrow{\iota_i}\Hub{C}$ is independent of the projection, all of these untilts agree. This says that the ideal $I_i$ generated by $\iota_i(\xi)$ in $W(C^+)$ agree, we call this ideal $I$. Since $\iota_i(\xi)=p-\varpi^{\frac{p^i}{p^k}}\iota_i(\alpha)$ the sequence $\iota_i(\xi)$ converges to $p$ in the $(p,\varpi)$-adic topology. But the ideal associated to an untilt is closed, so $p\in I$ and $\Spa{C}\to \Zpd$ factors through $\Fpd$.
\end{proof}

\begin{lem}
	\label{lem:discreteperfecthubpairs1}
	Let $\Hub{A}$ be as above and let $\Hub{B}$ be a complete Huber pairs over $\Zp$. Then every morphism of v-sheaves $\Spd{A}\to \Spd{B}$ comes from a unique morphism of Huber pairs $\Hub{B}\to \Hub{A}$.
\end{lem}
\begin{proof}
	Given a map $g:\Spd{A}\to \Spd{B}$ we first construct a map $m:\Spa{A}\to \Spa{B}$, and then prove that $m^\dia=g$. 
	Let $R=A\rpot{t\pthr}$, $R^+=A^++(t\pthr)A\pot{t\pthr}$ and $X=\Spa{R}$. 
	The natural map $X\to \Spd{A}$ surjects onto $\Spd{A}^\dagger$. 
	Since $X$ representable, any $f:X\to \Spd{B}$ is given by an untilt $R^\sharp$ and a map $f^*:(B,B^+)\to (R^\sharp,R^{\sharp,+})$. Let $f$ be induced by $g$, by \Cref{lem:redZp2} the untilt must be $R$.
	Since $f$ factors through $g$, $f^*(B)$ is invariant under automorphism of $R$ over $A$. 
	Take $b\in B$, we show $f^*(b)\in A\subseteq R$. Now, $t^{p^n}\cdot f^*(b)$ is topologically nilpotent for big enough $n$. Replacing by $t\mapsto t^{\frac{1}{p^m}}$ we conclude that $t^{p^n}f^*(b)$ is topologically nilpotent for $n\in \bb{Z}$. This proves that $f^*(b)$ is power-bounded so that $f^*(b)\in A\pot{t\pthr}$. Write $f^*(b)=a_0+t^{\frac{1}{p^m}}q$ with $a_0\in A$ and $q\in A\pot{ t\pthr}$. Now, $t^{\frac{1}{p^m}}q$ converges to $0$ under $t\mapsto t^{p^n}$ so $f^*(b)=a_0$. We get a ring map $m^*:B\to A$. The subspace topology of $A$ in $R$ is discrete, this gives continuity of $m^*$. Moreover, $R^+\cap A=A^+$. We have constructed $m:\Spa{A}\to \Spa{B}$ with $m^\dia=g$ over $\Spd{A}^\dagger$. 

	Consider $(g,m^\dia):\Spd{A}\to \Spd{B}\times \Spd{B}$ we show that $\Spd{A}$ factors through the diagonal $\Delta:\Spd{B}\to \Spd{B}\times \Spd{B}$. We can check this on geometric points $x:\Spd{C}\to \Spd{A}$. Since the maps agree on $\Spd{A}^\dagger$ we can assume that $x$ is meromorphic. Pick a pseudo-uniformizer $\varpi\in C$ and consider $R_\infty$ as in \Cref{lem:samegenripattern}. Consider $\Spa{R_\infty}\to \Spd{A}$ given by the diagonal morphism $A\to \prod_{i=1}^\infty C$. Recall, $C\subseteq_\Delta R_\infty \subseteq \prod_{i=1}^\infty C,$
	and that although $C\subseteq_\Delta R_\infty$ is not continuous the composition $A\to R_\infty$ is. 
	Now, $\Spa{R_\infty}\to \Spd{A}\to \Spd{B}\times \Spd{B}$ gives two maps $f_1,f_2:B\to R_\infty$, and both have to factor through the diagonal $C\subseteq_\Delta R_\infty$. 
	By the proof of \Cref{lem:samegenripattern} the residue field at a non-principal ultrafilter $\cali{U}$ maps to $x_{\mrm{mer}}$. Since $\Spa{C_\cali{U}}\to \Spd{B}\times \Spd{B}$ factors through $\Spd{A}^\dagger$ (being discrete in $\Spo{A}$), it also factors through the diagonal. These ring maps are the compositions $f_i:B\to C\to R_\infty \to C_\cali{U}$. We can conclude $f_1=f_2$ since $C\to C_\cali{U}$ is injective.  
\end{proof}

\begin{thm}
	\label{lem:discreteperfecthubpairs}
	\label{pro:perfectschemesvsadicultimate}
	Let $Y$ be a perfect discrete adic space over $\Fp$ and let $X$ be a pre-adic space over $\Zp$. The natural map $\mrm{Hom}_{_\mrm{PreAd}}(Y,X)\to \mrm{Hom}(Y^\dia,X^\dia)$ is bijective. In particular, $\dia$ is fully faithful when restricted to the category of perfect discrete adic spaces over $\Fp$.
\end{thm}
This theorem says, intuitively speaking, that (up to perfection) one does not get new morphisms of v-sheaves when the source is a discrete adic space. 
\begin{proof}
	It is not hard to prove injectivity. For surjectivity, the hard part is to prove that morphisms $g:Y^\dia\to X^\dia$ induce a map of topological spaces $f:|Y|\to |X|$ making the following diagram commute: 
	\begin{center}
		\begin{tikzcd}
			\mid Y^\dia\mid\ar{r}{g}\ar{d} & \mid X^\dia\mid \ar{d}\\
			\mid Y\mid\ar{r}{f} & \mid X\mid
		\end{tikzcd}
	\end{center}
	Indeed, if this holds true one can reduce to \Cref{lem:discreteperfecthubpairs1} by standard glueing arguments.
	Verifying that $g:|Y^\dia|\to|X^\dia|$ descends to $f:|Y|\to|X|$ can be done locally on $|Y|$, we may assume $Y=\Spa{A}$. Let $y\in |Y|$ and $z\in \Spo{A}$ with $h(z)=y$, we define $f(y):=h(g(z))$. We must verify that this doesn't depend on $z$ and that it is continuous. The map $f$ is well defined if and only if $h(g(z))=h(g(z_{\mrm{mer}}))$ when $z$ is meromorphic, and by \Cref{pro:quotientmapishfordiscrete} to prove continuity  it suffices to prove that if $\Si \subseteq |X|$ is open then $f^{-1}(\Si)$ is stable under arbitrary generization in $\Spa{A}$. Let $w\in \Spa{A}$ be a horizontal generization of $y$. Let $\Hub{k_y}$ and $\Hub{k_w}$ denote the affinoid residue fields of $w$ and $y$ and let $K_w$ denote the smallest ring containing $k_w^+$ and $A$ as in \cref{horizontalspecializandvals}. It suffices to prove that $|\mrm{Spd}(K_w,k_w^+)|\to |X^\dia|$ and $|\mrm{Spd}(k_y,k^+_y)|\to |X^\dia|$ descend to continuous maps $|\mrm{Spa}(K_w,k_w^+)|\to |X|$ and $|\mrm{Spa}(k_y,k^+_y)|\to |X|$. In summary, we have reduced to the case where $Y=\Spa{A}$ with $A^+\subseteq A$ two valuation rings with the same fraction field. We deal with this case in \Cref{lem:sotechinicalitscrazy} below.
	\end{proof}
	\begin{lem}
		\label{lem:sotechinicalitscrazy}
		Let $X$ and $Y=\Spa{A}$ as above, with $A^+\subseteq A\subseteq \mrm{Frac}(A^+)$ both valuation rings. Let $g:\Spd{A}\to X^\dia$ be a map. Let $h(c_\mrm{min})\in \Spa{A}$ denote the unique closed point and let $c_\mrm{min}\in \Spo{A}$ denote the unique discrete point mapping to $h(c_\mrm{min})$. If $h(g(c_\mrm{min}))\in |X|$ lies in $\Spa{B_1}\subseteq X$ then $g$ factors through a map $\Spd{A}\to \Spd{B_1}\subseteq X^\dia$. In particular, $g$ is coming from a map of pre-adic spaces $\Spa{A}\to \Spa{B_1}\subseteq X$.
	\end{lem}
	\begin{proof}
		Suppose to get a contradiction that there is an “exotic” $g$ that does not satisfy this property. By \Cref{prop:Valuationringsareolivine}, $|\Spd{A}|=\Spo{A}$ and we work with the later. Let $U_1\subseteq \Spo{A}$ be pullback of $\Spd{B_1}$, this a proper open subset. Let $Z=\Spo{A}\setminus U_1$, it is quasicompact and by \cite[Lemma 18.3.2]{Ber} we may find the largest prime $\mathfrak{p}_m\in \mrm{Spec}(A)$ of the form $\mathbf{supp}(z)$ with $z\in Z$. Replacing $A$ and $A^+$ by $A/\mathfrak{p}_m$ and $A^+/\mathfrak{p}_m$ we may assume that if $z\in Z$ then $\mathbf{supp}(z)=0$. Let $K=\mrm{Frac}(A)$, then $Z\subseteq \mrm{Spo}(K,A^+)$. Since $Z$ is a closed it contains the unique closed point $q_{\mrm{min}}\in \mrm{Spo}(K,A^+)$.\footnote{This is the unique discrete point such that $h(q_{\mrm{min}})$ is closed in $\mrm{Spa}(K,A^+)$.}
		Now, $U_1$ contains all analytic localizations $\N{n}{1}$ with $n\neq 0$ and $n\in \mathbf{supp}(c_\mrm{min})$. Indeed, if $z\in Z\cap \N{n}{1}$ then $|n|_z^a<1$ and either $z$ or $z_{\mrm{for}}$ have non-trivial support giving a contradiction. 
	
		Also, $\mrm{Spd}(K,A^+)\to X^\dia$ factors through another open affine subsheaf $\Spd{B_2}$, since it has a unique closed point. Let $U_2$ be the pullback of $\Spd{B_2}$. By \Cref{exa:valuationringsarecool}, there is an open with $q_{\mrm{min}}\in \U{0}{b}\cap \U{1}{b'}\cap \N{n}{1}\subseteq U_2$. Moreover, in the notation of \Cref{exa:valuationringsarecool} $q_{\mrm{min}}=(0,0,\frak{m})$ with $\frak{m}$ the maximal ideal of $A^+$. This gives that $n=0$ and that $b'\in A^{+}\setminus \frak{m}$ so $\U{1}{b'}=\Spo{A}$ and $\N{n}{1}=\Spo{A}$. In summary, $q_{\mrm{min}}\in\U{0}{b}\subseteq U_2$.
	
		We have found neighborhoods $\N{b}{1}\subseteq U_1$ and $\U{0}{b}\subseteq U_2$. Observe that $\Spo{A}=\N{b}{1}\cup \U{0}{b}$. Let$A_b^+$ denote the $(b)$-adic completion and $A_b=A^+_b\otimes_{A^+}A$. \Cref{lem:representinganalyticnbhoods} shows that $\N{b}{1}$ is represented by $\Spd{A_b}$. Also, $\U{0}{b}$ is represented by $\mrm{Spd}(A[\frac{1}{b}],A^+)$ and $\N{b}{1}\cap \U{0}{b}$ is represented by $\mrm{Spa}(A_b[\frac{1}{b}], A^+_b)$. Notice that this is a perfectoid field. Let $q_b$ be the closed point of $\N{b}{1}\cap \U{0}{b}$.
	Since the morphisms glue, there is $\Spa{B_3}\subseteq \Spa{B_1}\times_{X} \Spa{B_2}$ and $\mrm{Spa}(A_b[\frac{1}{b}], A^+_b)\to \Spd{B_3}$ making the following diagram commute: 
	\begin{center}
		\begin{tikzcd}
			
			\mrm{Spa}(A_b[\frac{1}{b}],A^+_b)\ar{rr} \ar{dd}\ar{rd}		&  &	\mrm{Spd}(A[\frac{1}{b}],A^+) \ar{d} \\
			& \Spd{B_3}\ar{r} \ar{d}			& \Spd{B_2}\ar{d}	\\	
			\mrm{Spd}(A_b,A^+_b)\ar{r}		& 	\Spd{B_1}\ar{r}&	X^\dia
		\end{tikzcd}
	\end{center}
	By \Cref{lem:discreteperfecthubpairs1} the map $\mrm{Spd}(A[\frac{1}{b}],A^+)\to \Spd{B_2}$ is given by a map of Huber pairs $\Hub{B_2}\to (A[\frac{1}{b}],A^+)$. The pullback of $\Spd{B_3}$ to $\mrm{Spo}(A[\frac{1}{b}],A^+)$ has the form $h^{-1}(U_3)$ for some $U_3\subseteq \mrm{Spa}(A[\frac{1}{b}],A^+)$. Moreover, $h(q_b)$ is the closed point of $\mrm{Spa}(A[\frac{1}{b}],A^+)$. This proves that $\mrm{Spd}(A[\frac{1}{b}],A^+)$ factors through $\Spd{B_3}$ and consequently through $\Spd{B_1}$ contradicting our assumption.
\end{proof}
We now study perfect discrete Huber pairs of the form $\Huf{A}$. 
\begin{pro}
	\label{pro:schematicmapsgiveschematicsets}
	Let $A$ be discrete ring and $f^*:\Hub{B}\to \Huf{A}$ a map. The following hold:
	\begin{enumerate}
		\item $f(\Spor{A})= h^{-1}(f(\Spf{A}))$.
		\item $f(\Spf{A})$ is stable under horizontal specializations in $\Spa{B}$.
		\item $f(\Spf{A})$ is stable under vertical generizations in $\Spa{B}$.
	\end{enumerate}
\end{pro}
\begin{proof}
	For the first claim let $y\in \Spor{A}$ and let $x=\Spof{f}(y)$. If $x$ is d-analytic, $y$ is meromorphic and $y_{\mrm{mer}}$ maps to $x_{\mrm{mer}}$, giving $h^{-1}(h(x))\subseteq Im(\Spof{f})$.
	Suppose that $x$ is discrete and that $x^{\mrm{mer}}$ exists. In this case, $y^{\mrm{mer}}$ might not exist and even if it does it might not map to $x^{\mrm{mer}}$. 
	Consider instead $h(y)\in \Spf{A}$ and its residue field map $\iota_{h(y)}:\Spa{K_{h(y)}}\to \Spf{A}$. Notice that $\iota_{h(y)}$ factors through $g:\Spf{K_{h(y)}^+}\to \Spf{A}$. We prove that $x^{\mrm{mer}}$ is in the image of $\Spof{f\circ g}$. Take $b\in B$ with $|b|^a_{x^{\mrm{mer}}}\notin \{0,1\}$ and replace it with its inverse in $K_{h(y)}$, if necessary, so that $b\in K_{h(y)}^+$. Define $K^+$ as the $(b)$-adic completion of $K_{h(y)}^+$, and let $K=K^+[\frac{1}{b}]$. We get a map $\Spa{K}\to  \Spor{K^+_{h(y)}} \to \Spo{B}$. 
	One can verify that the closed point of $\Spa{K}$ maps to $x^{\mrm{mer}}$. 

	The proof of the second claim also follows from observing that the residue field map $\iota_{h(y)}$ factors through $g$. Indeed, we get the following commutative diagram of adic spaces: 
	\begin{center}
		\begin{tikzcd}
			\Spa{K_{h(y)}}\ar{r} \ar{d}& \Spf{K_{h(y)}^+}\ar{r}{g} \ar{d} & \Spf{A}\ar{d} \\
			\Spa{K_{h(x)}}\ar{r} & \Spf{K_{h(x)}^+}\ar{r}{g'}  & \Spa{B} \\
		\end{tikzcd}
	\end{center}
	Where we use that $K^+_{h(x)}=K_{h(y)}^+\cap K_{h(x)}$ to define $g'$. Moreover, the vertical map on the left is surjective since $h(x)=f(h(y))$ and one can deduce that the vertical map in the middle column is also surjective because the map of valuation rings is local. A prime ideal of $J\subseteq K^+_{h(x)}$ determines a horizontal specializations of $|\cdot|_{h(x)}$, namely $|\cdot|_{h(x)}/J$, and every horizontal specialization of $h(x)$ can be constructed in this way. For $J$ as above we let $K_J^+=K^+_{h(x)}/J$ and $K_J=\mrm{Frac}(K_J^+)$, we get the following commutative diagram: 
	\begin{center}
		\begin{tikzcd}
			\Spa{K_J}\ar{r} \ar[bend right=10]{drr} &\Spf{K^+_J} \ar{dr} \ar{rr}& & \Spf{K^+_{h(y)}} \ar{dl} \\
 &			& \Spa{B} & 
		\end{tikzcd}
	\end{center}
	The closed point of $\Spa{K_J}$ maps to the horizontal specialization of $h(x)$ associated to the ideal $J$.

The third claim follows from \Cref{cor:verticalgeneriz} and from the first claim.
\end{proof}
\begin{defi}
	\label{defi:schematicsubs}
	We say that a subset of $\Spo{B}$ is a \textit{schematic subset} if it is a union of sets of the form $\Spof{m}(\Spor{A})$ where $\Huf{A}$ is given the discrete topology and $m^*:\Hub{B}\to \Huf{A}$ is a map of Huber pairs. 
\end{defi}

The following statement is the key result that allow us to construct a well-defined specialization map.

\begin{pro}
	\label{pro:closedschematicallyreduced}
	Suppose that $Z\subseteq \Spo{B}$ is a schematic closed subset. Let $\sigma:\Spo{B}\to \mrm{Spec}(B)$ denote the map $x\mapsto \mathbf{supp}(x)$ attaching to every point of $\Spo{B}$ its support ideal. Notice that $\sigma=\mathbf{supp}\circ h$ where $\mathbf{supp}:\Spa{B}\to \mrm{Spec}(B)$ also attaches the support ideal. Then $Z=\sigma^{-1}(V(I))$ for some prime ideal $I\subseteq B$ open for the topology in $B$. 
\end{pro}
\begin{proof}
	Any map $m^*:\Hub{B}\to \Huf{A}$ with $A$ a discrete ring factors through $(B/B^{\circ \circ},B^+/B^{\circ \circ})$, so we may assume that $B$ has the discrete topology. By \Cref{pro:schematicmapsgiveschematicsets}, $Z=h^{-1}(h(Z))$ and by \Cref{cor:verticalgeneriz}, $Z$ is closed under v.g. Moreover, since $Z$ is closed in $\Spo{B}$ it is also stable under vertical specialization. This implies that $Z=\sigma^{-1}(\sigma(Z))$. 
	We prove $\sigma(Z)$ is closed. Since $B$ has the discrete topology, the support map admits a continuous section $\mrm{Triv}:\mrm{Spec}(B)\to \Spa{B}$ that assigns to a prime ideal $\mathfrak{p}\subseteq B$ the trivial valuation with support $\mathfrak{p}$. We have $\sigma(Z)=\mrm{Triv}^{-1}(h(Z))$ so we may prove $h(Z)$ is closed instead. By \Cref{pro:schematicmapsgiveschematicsets}, $h(Z)$ is also closed under horizontal specialization, this gives that the complement of $h(Z)$ in $\Spa{B}$ is stable under (arbitrary) generization. Now, $\Spo{B}\setminus Z=h^{-1}(\Spa{B}\setminus h(Z))$ and by \Cref{pro:quotientmapishfordiscrete} the set $\Spa{B}\setminus h(Z)$ is open.
\end{proof}

\section{The reduction functor}
\subsection{The v-topology for perfect schemes}

In this section, set-theoretic carefulness is necessary. We advise the reader to review the definition and basic properties of cut-off cardinals \cite[\S 4]{Et}. 

Denote by $\mrm{PCAlg}^\mrm{op}_\Fp$ the category of perfect affine schemes over $\Fp$. If $\kappa$ is a cut-off cardinal we let $\mrm{PCAlg}^\mrm{op}_{\Fp,\kappa}$ be the category of perfect affine schemes over $\Fp$ whose underlying topological space and whose ring of global sections have cardinality bounded by $\kappa$. Given $S=\mrm{Spec}(A)\in \mrm{PCAlg}^\mrm{op}_\Fp$ we associate to it a v-sheaf in $\mrm{Perf}$ given by:
$S^\diamond(\Hub{R})=\{f:A\to R^+|f\text{\,\,is\,a\,\,morphism\,of\,rings}\}$.
\begin{rem}
	Notice that $\mrm{Spec}(A)^\diamond =\Spdf{A}$ when $A$ is given the discrete topology. 
\end{rem}
\begin{pro}
	\label{pro:fullyperfsch}
		If $\kappa$ is a cut-off cardinal and $S\in \mrm{PCAlg}^\mrm{op}_{\Fp,\kappa}$ then $S^\diamond$ is a $\kappa$-small v-sheaf.
\end{pro}
\Cref{pro:fullyperfsch} gives rise to functors $\diamond_\kappa:\mrm{PCAlg}_{\Fp,\kappa}^\mrm{op}\to \topPerf_{\kappa}$ that are compatible when we vary $\kappa$ and give rise to a functor $\diamond:\mrm{PCAlg}_{\Fp}^\mrm{op}\to \topPerf$
\begin{pro}
	\label{pro:diamcommuteswithlim}
	The functors $\diamond: \mrm{PCAlg}^\mrm{op}_\Fp\to \topPerf$ and $\diamond_\kappa: \mrm{PCAlg}^\mrm{op}_{\Fp,\kappa}\to \topPerf_\kappa$ are fully-faithful and commute with finite limits.
\end{pro}
\begin{proof}

This is a direct consequence of \Cref{lem:discreteperfecthubpairs}. 
\end{proof}

After embedding $\mrm{PCAlg}^\mrm{op}_\Fp$ in $\topPerf$ one can define a Grothendieck topology on $\mrm{PCAlg}^\mrm{op}_\Fp$ by considering a small family of maps of affine schemes, $(S_i\to T)_{i\in \F}$, to be a cover if the map $\coprod_{i\in \F}S_i^\diamond\to T^\diamond$ is a surjective map of v-sheaves. However, there is an intrinsic way of defining this topology which we now discuss. 
\begin{defi}\textup{(\cite[Definition 2.1]{Witt})}
		\label{defi:v-top-sch}
		\begin{enumerate}
			\item A morphisms of qcqs schemes $S\to T$, is said to be universally subtrusive (or a v-cover) if for any valuation ring $V$ and a map $\mrm{Spec}(V)\to T$ there is an extension of valuation rings $V\subseteq W$ (\cite[Tag 0ASG]{Stacks}) and a map $\mrm{Spec}(W)\to S$ making the following diagram commutative:
	\begin{center}
		\label{diag:univesubtrus}
	\begin{tikzcd}
		\mrm{Spec}(W) \arrow[d]\arrow[r]& S\arrow[d]\\
		\mrm{Spec}(V) \arrow[r]& T
	\end{tikzcd}
	\end{center}
\item A small family of morphisms in $\mrm{PCAlg}^\mrm{op}_\Fp$, $(S_i\to T)_{i \in \F}$, is said to be universally subtrusive (or a v-cover) if there is a finite subset $\F'\subseteq \F$ for which $\coprod_{i\in \F'} S_i\to T$ is universally subtrusive.
		\end{enumerate}
\end{defi}
\begin{lem}\textup{(\cite[Remark 2.2]{Witt})}
	\label{lem:subtrusivevsadic}
	A morphism $f:\mrm{Spec}(B)\to \mrm{Spec}(A)$ of affine schemes (not necessarily over $\Fp$) is universally subtrusive if and only if the map of topological spaces $|f^{\mrm{ad}}|:|\Spf{B}|\to |\Spf{A}|$ is surjective.
\end{lem}

\begin{lem}
	\label{lem:schemes-quasicompact}
	Let $f:S\to T$ be a morphism of perfect affine schemes over $\Fp$. The map $f^\diamond:S^\diamond\to T^\diamond$ is a quasicompact map of v-sheaves.
\end{lem}
\begin{proof}
	Observe that for a perfect discrete ring $A$ we have the identity $\Spdf{A}^\dagger=\Spdf{A}$. We can apply \Cref{lem:boundedlocusgivesqcqs}.
\end{proof}
   
\begin{pro}
		\label{pro:two-v-covers}
		\begin{enumerate}
			\item Let $f:S\to T$ be a morphism of perfect affine schemes over $\Fp$. The map $f$ is universally subtrusive if and only if $f^\diamond:S^\diamond \to T^\diamond$ is a surjective map of v-sheaves.
			\item A family of morphisms $(S_i\to T)_{i\in \F}$ is universally subtrusive if and only if $(\coprod_{i\in \F}S_i^\diamond)\to T^\diamond$ is a surjective map of v-sheaves.
		\end{enumerate}
\end{pro}
\begin{proof}
	Since $f^\diamond:S^\diamond \to T^\diamond$ is quasicompact, by \cite[Lemma 12.11]{Et} it is surjective if and only if $|f^\diamond|$ is surjective. By \Cref{pro:rightpointsspo} and \Cref{lem:subtrusivevsadic}, it suffices to prove that $\Spor{B}\to \Spor{A}$ is surjective if and only if the map $\Spf{B}\to \Spf{A}$ is. Surjectivity of $h$ proves one direction, the converse is a consequence of \Cref{pro:schematicmapsgiveschematicsets}.
	The second claim, follows easily from the first. 
\end{proof}
\begin{rem}
	\label{rem:prodpointssch}
	One can discuss the analogue of \Cref{exa:prodpointsbasis}. Given an index set $I$ and $\{V_i\}_{i\in I}$ a family of perfect valuation rings over $\Fp$, we let $R=\prod_{i\in I}V_i$. We call the affine schemes constructed in this way a scheme-theoretic product of points. They form a basis for the v-topology on $\mrm{PCAlg}^\mrm{op}_\Fp$ \cite[Lemma 6.2]{Witt}.
\end{rem}
Given a cut-off cardinal $\kappa$ we let $\topSch_\kappa$ be the topos associated to the site $\mrm{PCAlg}^\mrm{op}_{\Fp,\kappa}$ with the v-topology, and we will refer to an object in this topos as a $\kappa$-small scheme-theoretic v-sheaf. For any pair of cut-off cardinals $\kappa<\lambda$ we have a continuous fully-faithful embedding of sites $\iota_{\kappa,\lambda}^*:\mrm{PCAlg}^\mrm{op}_{\Fp,\kappa}\to \mrm{PCAlg}^\mrm{op}_{\Fp,\lambda}$, which induces a morphism of topoi $\iota_{\kappa,\lambda}:\topSch_\lambda\to \topSch_\kappa$.
\begin{pro}
	\label{pro:bigcardinalembedding}
	The functor $\iota_{\kappa,\lambda}^*:\topSch_\kappa\to \topSch_\lambda$ is fully-faithful \cite[Proposition 8.2]{Et}.
\end{pro}
\begin{proof}
	It is enough to prove that the adjunction $\F\to \iota_{\kappa,\lambda,*}\iota_{\kappa,\lambda}^* \F$ is an isomorphism. 
	Define a presheaf $S\mapsto \G(S)$ constructed as follows. 
	Let $\mathcal{C}^\kappa_S$ denote the category of maps of affine schemes $S\to T$ with $T\in \mrm{PCAlg}^\mrm{op}_{\Fp,\kappa}$. This category is cofiltered and there is a $\lambda$-small set of objects $I^\kappa_S\subseteq \mathcal{C}^\kappa_S$, that is cofinal in $\mathcal{C}^\kappa_S$. We let $\G(S)=\varinjlim_{T\in I^\kappa_S} \F(T)$, for any choice of $I^\kappa_S$. Unraveling the definitions we see that $\iota_{\kappa,\lambda}^*\F$ is the sheafification of $\G$. 

	We claim that $\G$ is already a sheaf. Indeed, since filtered colimits are exact it suffices to prove that v-covers $S'\to S$ in $\mrm{PCAlg}^\mrm{op}_{\Fp,\lambda}$ are filtered colimits of v-covers in $\mrm{PCAlg}^\mrm{op}_{\Fp,\kappa}$. Let $S=\mrm{Spec}(A)$ and let $S'=\mrm{Spec}(B)$, write $A=\varinjlim_{i\in I^\kappa_S} A_i$ and $B=\varinjlim_{j\in I^\kappa_{S'}} B_j$ with $A_i$ and $B_j$ $\kappa$-small rings, we may assume that the transition maps are injective. By \Cref{lem:kappasmallcovers} below we may assume that the $\mrm{Spec}(A)\to \mrm{Spec}(A_i)$ are v-covers. Consequently, $S'\to S\to \mrm{Spec}(A_i)$ are v-covers and when $S'\to \mrm{Spec}(A_i)$ factors through $\mrm{Spec}(B_j)\to \mrm{Spec}(A_i)$ this later one is also a v-cover. Replacing our index sets $I^\kappa_S$ and $I^\kappa_{S'}$ by a common index set $I$ and replacing $B_j$ by the smallest subring of $B$ containing $B_j$ and $A_i$ for some $i\in I^\kappa_S$ we can ensure $(\mrm{Spec}(B_i)\to \mrm{Spec}(A_i))_{i\in I}$ is defined for all $i\in I$ and is a v-cover. We get our desired expression $$(S'\to S)=\varprojlim_{i\in I}(\mrm{Spec}(B_i)\to \mrm{Spec}(A_i))_{i\in I}.$$

	Once we know $\iota_{\kappa,\lambda}^*\F=\G$, we compute $\iota_{\kappa,\lambda,*}\iota_{\kappa,\lambda}^*\F(S)=\F(S)$  since the identity is cofinal in $\mathcal{C}^\kappa_S$. 
\end{proof}
\begin{lem}
	\label{lem:kappasmallcovers}
	Let $\kappa$ be a cut-off cardinal, $S\in \mrm{PCAlg}^\mrm{op}_{\Fp}$ and $T\in \mrm{PCAlg}^\mrm{op}_{\Fp,\kappa}$. Given a morphism $g:S\to T$, there is $T'\in \mrm{PCAlg}^\mrm{op}_{\Fp,\kappa}$ together with morphisms $f:S\to T'$ and $h:T'\to T$ such that $f$ is a v-cover and $g=h\circ f$.
\end{lem}
\begin{proof}
	Let $S=\mrm{Spec}(B)$ and $T=\mrm{Spec}(A)$. By replacing $A$ by its image in $B$ we may assume $g^*:A\to B$ to be injective. We construct recursively a countable sequence of $\kappa$-small subrings $$A=A_0\subseteq \dots\subseteq A_n\subseteq A_{n+1}\subseteq\dots B$$ such that the image of $\Spf{B}\to \Spf{A_n}$ coincides with that of $\Spf{A_{n+1}}\to \Spf{A_n}$. Assume $A_n$ is defined and let $Z_n\subseteq \Spf{A_n}$ be the image of $\Spf{B}$ in $\Spf{A_n}$. If $x\in \Spf{A_n}\setminus Z_n$ the valuation $|\cdot|_x:A_n\to \Gamma_x$ can't be extended to a valuation $|\cdot|:B\to \Gamma$. A compactness argument proves there are finitely many elements $\{a_1,\dots a_m\}$ such that $|\cdot|_x$ does not extend to $A_n[a_1,\dots, a_m]\subseteq B$. Since $\Spf{A_n}\setminus Z_n$ is $\kappa$-small, there is $\lambda<\kappa$ and a set $\{a_i\}_{i\in \lambda}\subseteq B$ such that $A_n[a_i]_{i\in \lambda}$ does not extend any $x\in \Spf{A_n}\setminus Z_n$. We let $A_{n+1}=A_n[a_i\pthr]_{i\in \lambda}$. 
	
	We let $A_\infty=\varinjlim_{i\in \mathbb{N}}A_i$, it is $\kappa$-small and we claim that the map $\mrm{Spec}(B)\to \mrm{Spec}(A_\infty)$ is a v-cover. 
	We use \Cref{lem:subtrusivevsadic} to prove instead that $\Spf{B}\to \Spf{A_\infty}$ is surjective. One verifies that $\Spf{A_\infty}=\varprojlim_{i\in \bb{N}}\Spf{A_i}$. Given a compatible sequence $x_i\in \Spf{A_i}$ let $M_i$ be the preimage of $x_i$ in $\Spf{B}$. This gives a sequence $\Spf{B}\supseteq M_0\supseteq M_1\dots $
	Since the maps $\Spf{B}\to \Spf{A_i}$ are spectral, each $M_i$ is compact in the patch topology. Any element in this intersection maps to $x_\infty$. 
\end{proof}
We define $\topSch$ as the big colimit $\bigcup_\kappa \topSch_\kappa$ along all cut-off cardinals and the fully-faithful embeddings $\iota_{\kappa,\lambda}^*$. Objects in $\topSch$ are called small scheme-theoretic v-sheaves.

The general formalism of topoi, specifically (\cite[IV 4.9.4]{SGA4}), allows us to promote $\diamond_\kappa: \mrm{PCAlg}^\mrm{op}_{\Fp,\kappa}\to \topPerf_\kappa$ to a morphism of topoi $f_\kappa:\topPerf_\kappa\to \topSch_\kappa$ for which $f_\kappa^*|_{\mrm{PCAlg}^\mrm{op}_{Fp,\kappa}}=\diamond_\kappa$. 
\begin{pro}
	\label{pro:commutetopoi}
Given two cut-off cardinals $\kappa<\lambda$ we have a commutative diagram of morphism of topoi:	
\begin{center}
	\begin{tikzcd}
		\topPerf_\lambda \ar{r}{f_\lambda}\ar{d}{\iota_{\kappa,\lambda}} &	\topSch_\lambda \ar{d}{\iota_{\kappa,\lambda}}\\
		\topPerf_\kappa \ar{r}{f_\kappa }&	\topSch_\kappa
	\end{tikzcd}
\end{center}
Moreover, the natural morphism $\iota_{\kappa,\lambda}^*\circ f_{\kappa,*}\to f_{\lambda,*}\circ \iota_{\kappa,\lambda}^*$ is an isomorphism.

\end{pro}
\begin{proof}
	The commutativity of morphism of topoi follows formally from the similar commutativity of continuous functors.
For the second claim, given an element $S\in \mrm{PCAlg}^\mrm{op}_{\Fp,\lambda}$ we let $I^\kappa_S$ be an index set category as in the proof of \Cref{pro:bigcardinalembedding}. If $S=\mrm{Spec}(A)$ we let $X=\mrm{Spa}(A\rpot{t\pthr},A\pot{t\pthr})$ and $Y=X\times_{S^\diamond}X$. In a similar way, for $T\in I^\kappa_S$ with $T=\mrm{Spec}(B)$ we let $X_T=\mrm{Spa}(B\rpot{t\pthr},B\pot{t\pthr})$ and $Y_T=X_T \times_{T^\diamond} X_T$. The family of perfectoid spaces $(X_T)_{T\in I^\kappa_S}$ ($(Y_T)_{T\in I^\kappa_S}$ respectively) is cofinal in the category $\mathcal{C}^\kappa_X$ of maps $X\to X'$ with $X'$ a $\kappa$-small perfectoid space ($\mathcal{C}^\kappa_Y$ respectively). We get the following chain of isomorphisms:
\begin{align}
	\iota_{\kappa,\lambda}^* f_{\kappa,*}\F (S) & =  \varinjlim_{T\in I^\kappa_S} \mrm{Hom}_{\topSch_\kappa}(h_T, f_{\kappa,*}\F) \\
	& =  \varinjlim_{T\in I^\kappa_S} \mrm{Hom}_{\topPerf_\kappa}(f^*_\kappa h_T, \F) \\
	& =  \varinjlim_{T\in I^\kappa_S} \mrm{Hom}_{\topPerf_\kappa}(T^{\diamond_\kappa}, \F) \\
	& =  \varinjlim_{T\in I^\kappa_S} \mrm{Eq}_{\topPerf_\kappa}(\mrm{Hom}(X_T, \F)\rightrightarrows \mrm{Hom}(Y_T, \F)) \\
	& =  \mrm{Eq}_{\topPerf_\lambda}(\varinjlim_{T\in I^\kappa_S} \mrm{Hom}(X_T, \F)\rightrightarrows \varinjlim_{T\in I^\kappa_S}\mrm{Hom}(Y_T, \F)) \\
	& =  \mrm{Eq}_{\topPerf_\lambda}(\mrm{Hom}(X_S, \iota_{\kappa,\lambda}^*\F)\rightrightarrows \mrm{Hom}(Y_S, \iota_{\kappa,\lambda}^*\F)) \\
	& =  \mrm{Hom}_{\topPerf_\lambda}(S^{\diamond_\lambda}, \iota_{\kappa,\lambda}^*\F) \\
	& =  \mrm{Hom}_{\topSch_\lambda}(h_S, f_{\lambda,*}\iota_{\kappa,\lambda}^*\F) \\
	& = f_{\lambda,*}\iota_{\kappa,\lambda}^*\F(S)
\end{align}
\end{proof}

Recall that a morphism of topoi consists of a pair of adjoint functors $(f^*,f_*)$ such that $f^*$ commutes with finite limits. By \Cref{pro:commutetopoi} above we can gather all of the morphisms of topoi $f_\kappa: \topPerf_\kappa\to \topSch_\kappa$ into a pair of adjoint functors $(f^*,f_*):\topPerf\to \topSch$ such that $f^*$ commutes with finite limits. This is not a morphism of topoi because $\topPerf$ and $\topSch$ are not topoi, but they behave as such. 
\begin{defi}
	\label{defi:reduced}
		Let $(f^*,f_*)$ be the pair of adjoint functors described above, given $\F\in \topSch$ we will denote $f^*\F$ by $\F^\diamond$ and given $\G\in \topPerf$ we will denote $f_*\G$ by $(\G)\red$. We refer to $(-)\red$ as the reduction functor.
\end{defi}

\begin{rem}
	\label{rem:diamond+ajunct}
	By adjunction $\F\red(S)=\mrm{Hom}_{\topPerf}(S^\diamond,\F)$. 
	We could have simply defined it in this way, but it is useful to know that “reduction” preserves smallness.
\end{rem}

We can endow small scheme-theoretic v-sheaf with a topological space in a similar fashion to \Cref{defi:diamondtop}. Given $\Si\in \topSch$ we let $|\Si|$ denote the set of equivalence classes of maps $\mrm{Spec}(k)\to \Si$, where $k$ is a perfect field over $\Fp$. Two maps $p_1$, $p_2$ are equivalent if we can complete a commutative diagram as below: 
\begin{center}
\begin{tikzcd}
	&\mrm{Spec}(k_1) \arrow{rd}{p_1} & \\
	\mrm{Spec}(k_3)\arrow{ru}{q_1} \arrow{rd}{q_2} \arrow{rr}{p_3}	& & \Si  \\
	&\mrm{Spec}(k_2) \arrow{ru}{p_2}& \\
\end{tikzcd}
\end{center}
 
\begin{pro}
			\label{pro:topologyschematicvsheaf}
Let $\Si\in \topSch$ the following hold:

	\begin{enumerate}
		\item There is a pair of cut-off cardinals $\kappa<\lambda$ and a $\lambda$-small family $\{S_i\}_{i\in I}$ of objects in $\mrm{PCAlg}^\mrm{op}_{\Fp, \kappa}$ together with a surjective map $X=(\coprod_{i\in I} S_i)\to \Si$.
		\item The small scheme-theoretic v-sheaf $R=X\times_\Si X$ has a similar cover $Y=(\coprod_{j\in J} T_j)\to R$, there is a natural map $|X|\to |\Si|$ which induces a bijection $|\Si|\cong |X|/|Y|$. We endow $|\Si|$ with the quotient topology induced by this bijection.
		\item The topology on $|\Si|$ does not depend on the choices of $X$ or $Y$.
		\item Any map of small v-sheaves $\Si_1\to \Si_2$ induces a continuous map of topological spaces $|\Si_1|\to |\Si_2|$.
	\end{enumerate}
\end{pro}

\subsection{Reduction functor and formal adicness}

\begin{defi}
	Let $\F\in \topSch$, we say it is \textit{reduced} if $\F\to (\F^\diamond)\red$ is an isomorphism. 
\end{defi}
\begin{pro}\textup{(\cite[Proposition 18.3.1]{Ber})}
	\label{pro:schemesarereduced}
	\label{pro:reducedgivesfullyfaithful}
	\begin{enumerate}
		\item If $S$ is a perfect scheme over $\Fp$ then the Yoneda functor $h_S$ is reduced. 
		\item The functor $\diamond:\topSch \to \topPerf$ is fully-faithful when restricted to small reduced v-sheaves.	
	\end{enumerate}
\end{pro}
\begin{proof}
The first claim follows from \Cref{lem:discreteperfecthubpairs}. The second claim follows from adjunction. Indeed, 
$\mrm{Hom}_\topPerf(\G^\diamond,\F^\diamond) = 	\mrm{Hom}_\topSch(\G,(\F^\diamond)\red)= \mrm{Hom}_\topSch(\G,\F)$.
\end{proof}

Intuitively, the reduction functor kills all topological nilpotent elements and removes analytic points. One can think of reduction functor as taking the underlying reduced subscheme of a formal scheme.
\begin{lem}
	\label{lem:redZp}
	The scheme-theoretic v-sheaf $\Zpd\red$ is represented by $\mrm{Spec}(\Fp)$.
\end{lem}
\begin{proof}	
	This is a direct consequence of \Cref{lem:redZp2}.
\end{proof}	

For an f-adic ring ${A}$ over $\Zp$, we let $\ovr{A}=(A/(A\cdot A^{\circ \circ}))^{\mrm{perf}}$ where $A\cdot A^{\circ \circ}$ is the ideal generated by the topological nilpotent elements. The following statement generalizes \Cref{lem:redZp} 
\begin{pro}	
	\label{pro:tworeds}	
	\label{pro:globaltworeds}	
 Let $X$ be a pre-adic space over $\Zp$ and let $X^{\mrm{na}}$ be the reduced adic space associated to the non-analytic locus of \Cref{pro:preadicreduced}. The following hold: 
	\begin{enumerate}
		\item The map $(X^{\mrm{na},\diamond})\red\to (X^\diamond)\red$ is an isomorphism. 
		\item If $X=\Spa{A}$ for $\Hub{A}$ a Huber pair over $\Zp$, then $\Spd{A}\red$ is represented by $\mrm{Spec}(\ovr{A})$. 	
	\end{enumerate}
\end{pro}	
   \begin{proof}	           

	   By \Cref{lem:discreteperfecthubpairs} if $S=\mrm{Spec}(R)\in \mrm{PCAlg}_\Fp^\mrm{op}$ then morphisms $S^\diamond\to X$ are given by maps of pre-adic spaces $f:\Spf{R}\to X$. These factor through the non-analytic locus.
	   The non-analytic locus of $\Spa{A}$ is represented by the Huber pair $(A/A^{\circ \circ}\cdot A,A^{\circ \circ}\cdot A^+)$. Since $R$ is perfect the map $f^*:A/A\cdot A^{\circ\circ}\to R$ factors uniquely through its perfection.
   \end{proof}

\begin{pro}
	\label{pro:redofdiamond} If $Y$ is a quasiseparated diamond, then $Y\red=\emptyset$.
\end{pro}
\begin{proof}
	It suffices to prove that there are no maps $f:S^\diamond \to Y$ for $S=\mrm{Spec}(k)$ and $k$ an algebraically closed field. Suppose $f$ exists and let $y\in|Y|$ be the unique point in the image of $|f|$. Consider $Y_y$ the sub-v-sheaf of points that factors through $y$. By \cite[Proposition 11.10]{Et} it is a quasiseparated diamond and $|Y_y|$ consists of one point. Using \cite[Proposition 21.9]{Et} we write $Y_y=\mrm{Spa}(C,O_C)/\underline{G}$ with $C$ a nonarchimedean algebraically closed field over $\Fp$ and $\underline{G}$ a profinite group acting continuously and faithfully on $C$. 
	
	Consider the v-cover $S'=\mrm{Spa}(K_1,O_{K_1})\to \mrm{Spec}(k)^\diamond$ where $K_1$ is an algebraic closure of $k\rpot{t\pthr}$. Similarly, let $T=\mrm{Spa}(K_2,O_{K_2})$ where $K_2$ is an algebraically closed nonarchimedean field containing $k$ discretely and whose value group $\Gamma_{K_2}\subseteq \bb{R}^{> 0}$ has at least two $\bb{Q}$-linearly independent elements. By hypothesis on $K_2$, we can find two embeddings $\iota^*_i:K_1\to K_2$ with $\iota^*_1(K_1)\cap  \iota^*_2(K_1)=k$. 
	
	The composition of $[g]:\Spa{K_1}\to S^\diamond\to  Y_y$ satisfies $[g]\circ \iota_1= [g]\circ \iota_2$. Since $\Spa{K_1}$ and $\Spa{K_2}$ are algebraically closed the maps to $Y_y$ are given by $G$-orbits of maps to $\mrm{Spa}(C,O_C)$.
	Let $g^*:(C,O_C)\to (K_1,O_{K_1})$ represent $[g]$ in $\mrm{Hom}(\Spa{K_1},Y_y)$, we get maps $\iota_i^*\circ g^*:(C,O_C)\to (K_2,O_{K_2})$ and since $[g]\circ \iota_1=[g]\circ \iota_2$ we have $\iota^*_1\circ g^*(C)=\iota^*_2\circ g^*(C)\subseteq k$. The incompatibility of topology between $k$ and $C$ gives the contradiction. 
\end{proof}
Recall that a morphism of adic spaces $X\to Y$ is said to be adic if the image of an analytic point is again an analytic point. For v-sheaves we can define a related notion. 
\begin{defi}
	\label{defi:adicmorph}
	A morphism $\F\to \G$ is \textit{formally adic} if the following diagram is Cartesian: 
\begin{center}
\begin{tikzcd}
	(\F\red) ^\diamond \ar[r]\ar[d] &(\G\red)^\diamond\ar[d]\\
	\F\ar[r] &\G
\end{tikzcd}
\end{center}
\end{defi}
Although the notion of a morphism of adic spaces being adic is related to the morphism of v-sheaves being formally adic neither of this notions implies the other. 
   \begin{exa}	           
   	\label{exa:adicnotformaladic}	           
	Endow $\Fp\rpot{t}$ with the discrete topology, then $\mrm{Spa}(\Fp\rpot{t},\Fp\pot{t})\to \mrm{Spa}(\Fp,\Fp)$ is adic. Nevertheless, $\mrm{Spd}(\Fp\rpot{t},\Fp\pot{t}) \to \mrm{Spd}(\Fp,\Fp)$ is not formally adic. 
	Observe that $\mrm{Spo}(\Fp\rpot{t},\Fp\pot{t})$ has an unbounded meromorphic point.	           
   \end{exa}	           
   \begin{exa}	           
   	\label{exa:formaladicnotadic}	           
	Let $K$ be a perfect nonarchimedean field and consider $\mrm{Id}:\mrm{Spa}(K_1,O_{K_1})\to \mrm{Spa}(K_2,O_{K_2})$ where $K_2=K$ given the discrete topology and $K_1=K$ given the norm topology. This morphism is not adic, but the reduction diagram is Cartesian. Indeed, it looks like this:	           
   \begin{center}	           
   \begin{tikzcd}	           
   \emptyset \arrow{r} \arrow{d} & \mrm{Spec}(K_2)^\diamond \arrow{d} \\	           
   \mrm{Spd}(K_1,O_{K_1}) \arrow{r} & \mrm{Spd}(K_2,O_{K_2})	           
   \end{tikzcd}	           
   \end{center}	           
   \end{exa}	           
Although formal adicness does not capture adicness in general, it does in important situations: 
\begin{pro}
	\label{pro:twoadics}
	Let $\Huf{A}$ and $\Huf{B}$ be formal Huber pairs over $\Zp$ with ideals of definition $I_A$ and $I_B$ respectively. Then $\Spf{A}\to \Spf{B}$ is adic if and only if $\Spdf{A}\to \Spdf{B}$ is formally adic.
\end{pro}
\begin{proof}
The reduction diagram looks as follows:
\begin{center}
\begin{tikzcd}
	\mrm{Spec}(\ovr{A})^\diamond \ar[rrd, bend left] \ar[ddr, bend right] \ar{rd} & & \\
&	(\mrm{Spec}(A/I_B)^{\mrm{perf}})^\diamond \arrow{r} \arrow{d} & \mrm{Spec}(\ovr{B})^\diamond \arrow{d} \\
&	\Spdf{A} \arrow{r} &\Spdf{B} 
\end{tikzcd}
\end{center}
Continuity of the morphism $B\to A$ ensures that $I_B^n\subseteq I_A$ for some $n$. The morphism is adic if and only if $I_A^m\subseteq A\cdot I_B$ for some $m$. If the morphism is adic, then ${A}/I_A$ and $(A/A\cdot I_B)$ become isomorphic after taking perfection which gives formal adicness. Conversely, if the morphism is formally adic, by hypothesis the rings $(A/I_B)^{\mrm{perf}}$, and $\ovr{A}$ are isomorphic with the isomorphism being induced by the natural surjective ring map with source $(A/p)^{\mrm{perf}}$. This implies that the ideals $I_A$ and $I_B$ define the same Zariski closed subset in $\mrm{Spec}(A)$. In particular, the elements of $I_A$ are nilpotent in $A/I_B$, and since $I_A$ is finitely generated $I_A^m\subseteq I_B$ for some $m$.
\end{proof}

\begin{pro}
	\label{pro:preservedbycomposit}
	Let $\F$, $\G$ and $\Hi$ be small v-sheaves.
	\begin{enumerate}
		\item If $\F\to \Hi$ and $\Hi\to \G$ are formally adic, the composition $\F\to \G$ is formally adic.
		\item If $\F\to \Hi$ is formally adic, the basechange $\G\times_\Hi\F \to \G$ is formally adic.
	\end{enumerate}
\end{pro}
\begin{proof}
The first claim is clear. The second follows from the commutativity of $(-)\red$ and $(-)^\diamond$ with finite limits. 
\end{proof}
\begin{defi}
	\label{defi:p-adic}
	We say that a v-sheaf $\F$ over $\Zpd$ is \textit{formally $p$-adic} (or just \textit{$p$-adic} when the context is clear) if the morphism $\F\to \Zpd$ is formally adic.
\end{defi}
Over $\Zp$ the situation of \Cref{exa:formaladicnotadic} does not happen.
\begin{pro}
	\label{pro:formallypadicispadic}
	Suppose we have a Huber pair $\Hub{A}$ and a map $f:\Spa{A}\to \Spf{\Zp}$, if $f^\diamond$ is formally adic then $f$ is adic (as a morphism of adic spaces). 
\end{pro}
\begin{proof}
	Let $U\subseteq \Spa{A}$ the open subset of analytic points. 
	It follows from \Cref{pro:redofdiamond} and \Cref{pro:globaltworeds} that $U^\diamond \to \Spd{A}$ is formally adic. 
	By \Cref{pro:preservedbycomposit}, $U^\diamond\to \Zpd$ is formally adic and the map must factor through $\Qpd$. This proves proves that $f$ is adic.
\end{proof}

Recall that a v-sheaf $\F$ is said to be separated if the diagonal $\F\to \F\times\F$ is a closed immersion \cite[Definition 10.7]{Et}. We need the following related notion:
\begin{defi}
	\label{defi:formallysep}
	Let $\F$ and $\G$ be small v-sheaves.
	\begin{enumerate}
		\item We say $\F\to \G$ is \textit{formally closed} if it is a formally adic closed immersion.
		\item We say that a v-sheaf is \textit{formally separated} if the diagonal map $\F\to \F\times\F$ is formally closed. 
	\end{enumerate}
\end{defi}

\begin{lem}
	\label{lem:Zpdseparated}
The v-sheaf $\Zpd$ is formally separated.
\end{lem}
\begin{proof}
	To prove that the diagonal $\Zpd\to \Zpd\times \Zpd$ is a closed immersion observe that the basechanges by maps $\Spa{R}\to \Zpd\times \Zpd$, with $\Spa{R}\in \mrm{Perf}$ define the locus on which two untilts agree in $|\Spa{R}|$. Each untilt is individually cut out of $\Spf{W(R^+)}\setminus\{V([\varpi])\}$ as a closed Cartier divisor \cite[Proposition 11.3.1]{Ber}. The intersection defines a Zariski closed subset in each of the untilts and these are represented by a perfectoid space. 

	We compute directly $(\Zpd\times\Zpd)\red=\Fp$ since $(-)\red$ commutes with limits. On the other hand,  $\Fpd\times_{\Zpd^2} \Zpd=\Fpd$, which proves that the diagonal is formally adic.
\end{proof}

\begin{pro}
	\label{pro:formallypadicgivesformaldiagona}
If $\F$ is formally $p$-adic, then the diagonal $\F\to \F\times \F$ is formally adic.
\end{pro}
\begin{proof}
We have a formally adic map $\F\to \Zpd$, and since formal adicness is preserved by basechange and composition we get a formally adic map $\F\times_\Zpd \F\to \Zpd$. By a general property of Cartesian diagrams, the diagonal map $\F\to \F\times_\Zpd \F$ is also formally adic. Now, $\F\times_\Zpd \F$ is the basechange of the diagonal $\Zpd\to \Zpd\times \Zpd$ by the projection $\F\times \F\to \Zpd\times \Zpd$. This gives that $\F\times_\Zpd \F\to \F\times \F$ and by composition that $\F\to\F\times\F$ are also formally adic.
\end{proof}

\begin{lem}
	\label{rem:diamondofreduced}
	The diagonal $\F\to \F\times \F$ is formally adic if and only if the adjunction morphism $(\F\red)^\diamond\to \F$ is injective. Let $\Spa{A}\in \mrm{Perf}$ and $m\in \F\Hub{A}$. Then $m\in (\F\red)^\diamond\Hub{A}$ if and only if $\Spa{A}$ admits a v-cover $\Spa{R}\to \Spa{A}$ and a map $\mrm{Spec}(R^+)^\diamond\to \F$ making the following diagram commutative:
\begin{center}
	\begin{tikzcd}
		\Spa{R} \ar{r}\ar{d} &  \mrm{Spec}(R^+)^\diamond \ar{d} \\
		\Spa{A} \ar{r}{m} & \F
	\end{tikzcd}
\end{center}
\end{lem}
\begin{proof}
	In general, a map of sheaves $\G\to \F$ is injective if and only if $(\G\times \G)\times_{\F\times \F} \F=\G$. 
	Now, $(\F\red)^\diamond$ is the sheafification of $\Hub{R} \mapsto \mrm{Hom}(\mrm{Spec}(R^+)^\diamond,\F)$. The description given in the statement above is what one gets from taking sheafification and assuming injectivity of $(\F\red)^\diamond\to \F$. 
\end{proof}

The following lemma will be key for our theory of specialization, it roughly says that formally adic closed immersions behave as expected: 
\begin{lem}
	\label{lem:formallyclosedisZariski}
	Let $(A,A)$ be a formal Huber pair and let $\F\to \Spdf{A}$ be formally adic closed immersion. Then $(\F\red)^\diamond=\mrm{Spec}(A/J)^\diamond$ for some open ideal $J\subseteq A$.
\end{lem}
\begin{proof}
	$|\F|\subseteq \Spor{A}$ is closed and we get an expression $\F=\Spdf{A}\times_{\underline{|\Spd{A}|}} \underline{|\F|}$. 
	By \Cref{pro:tworeds}, $(\Spdf{A}\red)^\diamond=\mrm{Spec}(\ovr{A})^\diamond$ which is closed in $\Spdf{A}$. By formal adicness $(\F\red)^\diamond$ is closed in $\Spdf{A}$ determined by $|\F|\cap |\mrm{Spec}(\ovr{A})^\diamond|$. By \Cref{rem:diamondofreduced}, a map $\Spa{R}\to \F$ factors through $(\F\red)^\diamond$ if after possibly replacing $R$ by a v-cover it factors through $\mrm{Spec}(R^+)^\diamond \to \F\cap \mrm{Spec}(\ovr{A})^\diamond$. This proves that $|(\F\red)^\diamond|$ is a schematic closed subset of $\Spor{A}$ as in \Cref{defi:schematicsubs}. By \Cref{pro:closedschematicallyreduced}, it is a Zariski closed subset corresponding to an open ideal $J\subseteq A$.
\end{proof}


We will often use implicitly the following easy result. 
\begin{lem}
	\label{lem:basechangerepresentimpliesadic}
	Let $\F$ and $\G$ be two small v-sheaves, and $f:\F\to \G$ a map between them. Suppose that the adjunction map $(\G\red)^\diamond \to \G$ is injective and that $\F\times_{\G} (\G^\mrm{red})^\diamond$ is representable by a reduced scheme-theoretic v-sheaf, then $f$ is formally adic.
\end{lem}
\begin{proof} 
	Let $T\in \topSch$ be reduced and such that $T^\diamond=\F\times_{\G} (\G^\mrm{red})^\diamond$. By hypothesis $(\G\red)^\diamond \to \G$ is a monomorphism and since $(-)\red$ is a right adjoint $((\G\red)^\diamond)\red \to \G\red$ is also a monomorphism. Recall that for any pair of adjoint functors $(L,R)$ the compositions $R\to R\circ L\circ R\to R$ and $L\to L\circ R\circ L\to L$ are the identity. This implies that $((\G\red)^\diamond)\red \to \G\red$ is an isomorphism. We compute directly:
	\begin{align*}
		(T^\diamond	)\red	&=(\F\times_{\G} (\G\red)^\diamond)\red	\\
					&= \F\red\times_{\G\red} ((\G\red)^\diamond)\red\\
					&= \F\red\times_{\G\red} \G\red \\
					&= \F\red
	\end{align*}
and
\begin{align*}
	(\F\red)^\diamond &=((T^\diamond)\red)^\diamond\\
			  &= T^\diamond
\end{align*}

\end{proof}

\section{Specialization}
\subsection{Specialization for Tate Huber pairs}
\begin{defi}
	\label{defi:specializationTate}
	Given a Tate Huber pair $\Hub{A}$ over $\Zp$ and a pseudo-uniformizer $\varpi\in A$, we define the specialization map $\mrm{sp}_{A}:|\Spa{A}|\to |\mrm{Spec}(\ovr{A^+})|$ by sending a valuation $|\cdot|_x\in |\Spa{A}|$ to the ideal $\frak{p}\subseteq A^+$ given by $\frak{p}=\{a\in A^+ \mid |a|_x<1\}$.
\end{defi}
These maps of sets are functorial in the category of Tate Huber pairs. We thank David Hansen for providing reference and an explanation of the following statement.

\begin{pro}\textup{(\cite[Theorem 8.1.2]{Bhatt})}
	\label{pro:specialisniceforTate}
	The specialization map $\mrm{sp}_{A}:|\Spa{A}|\to |\mrm{Spec}(\ovr{A^+})|$ is a continuous, surjective, spectral and closed map of spectral topological spaces.
\end{pro}

\begin{pro}
	\label{pro:totallydisconnectedspecialhomeomorphism}
	For a strictly totally disconnected space $\Spa{R}$, the specialization map $\mrm{sp}_{R}$ is a homeomorphism.
	\label{pro:totallydisconhomeo}
\end{pro}
\begin{proof}
	By \Cref{pro:specialisniceforTate} the map is surjective and a quotient map so it suffices to prove injectivity. One first proves that if $\mrm{sp}_{R}(x)=\mrm{sp}_{R}(y)$, then $x$ and $y$ are in the same connected component of $|\Spa{R}|$. In this way one reduces to prove injectivity component by component. Using \Cref{pro:connectdcompcriterion} we can assume $R=C$, and this case follows from generalities of valuation rings.  
\end{proof}

\begin{rem}
One can also prove \Cref{pro:specialisniceforTate} if we knew already that \Cref{pro:totallydisconhomeo} holds.  
\end{rem}

\subsection{Specializing v-sheaves}
We now discuss the specialization map for v-sheaves. The idea is to descend the specialization map from the case of formal Huber pairs.
\begin{defi}
	We say that a small v-sheaf $\F$ is \textit{v-locally formal} if there is a set $I$, a family $\Huf{B_i}_{i\in I}$ of formal Huber pairs over $\Zp$ and a surjective map of v-sheaves $\coprod_{i\in I}\Spdf{B_i}\to \F$.
\end{defi}
\begin{defi}
	\label{defi:v-formalizing}
	Let $\F\in \topPerf$, $\Hub{A}$ be a Tate Huber pair and $f:\Spd{A}\to \F$ a map.
\begin{enumerate}
	\item We say that $\F$ \textit{formalizes} $f$ (or that $f$ is \textit{formalizable}) if there is $t:\Spdf{A^+}\to \F$ factoring $f$. 
%
	\item We say that $\F$ \textit{v-formalizes} $f$ if for some v-cover $g:\Spa{B}\to \Spa{A}$, $\F$ formalizes $f\circ g$.
	\item We say that $\F$ is \textit{formalizing} if it formalizes maps with source an affinoid perfectoid space.
	\item We say that $\F$ is \textit{v-formalizing} if it v-formalizes any $f$ as above.
\end{enumerate}
\end{defi}

We use this extensively because it gives an abstract way to verify that a v-sheaf is v-locally formal.
\begin{lem}
	\label{lem:Zpdvformalizing}
	The following statements hold:
\begin{enumerate}
	\item The v-sheaf $\Zpd$ is formalizing.
	\item $\Spdf{B}$ is formalizing for any formal Huber pair over $\Zp$. 
 \item A small v-sheaf $\F$ is v-formalizing if and only if it is v-locally formal.
\end{enumerate}
\end{lem}
\begin{proof}
	Let $\Spa{R}\in \mrm{Perf}$ in characteristic $p$ and an untilt $\iota:(R^\sharp)^\flat \to R$. 
	Let $\xi=p+[\varpi]\alpha$ be a generator of the kernel of $W(R^+)\to (R^\sharp)^+$. The image of $\xi$ under $W(R^+)\to W(A^+)$ defines an untilt of $\Spa{A}$ for every map $\Spa{A}\to \Spdf{R^+}$, and gives a map $\Spdf{R^+}\to \Zpd$. 
	Now, given $\mrm{Spa}(R^\sharp,R^{\sharp,+})\to \Spf{B}$ we get $B\to R^{\sharp,+}$ and $\Spdf{R^{\sharp,+}}\to \Spdf{B}$.
	For the third claim, assume that $\F$ is v-formalizing. Since it is small there is a set $I$ and a surjective map by a union of affinoid perfectoid spaces $\coprod_{i\in I}\Spa{R_i}\to \F$. After refining this cover we may assume that each $\Spa{R_i}\to \F$ formalizes to $\Spdf{R_i^+}\to \F$. Then $\coprod_{i\in I} \Spdf{R_i^+}\to \F$ is surjective, so $\F$ is v-locally formal. If $\F$ is v-locally formal a map $\Spa{R}\to \F$ will v-locally factor through a map $\Spa{R}\to \Spdf{B_i}$. By the second claim, this map formalizes $\Spdf{R^+}\to \Spdf{B_i}$. 
\end{proof}
\begin{pro}
The following properties are easy to verify.
	\label{rem:propertiesofformalizing}
\begin{enumerate}
\item If $f:\F\to \G$ is a surjective map of small v-sheaves and $\F$ is v-formalizing then $\G$ is v-formalizing. 
\item If $\mrm{Spec}(R)\in \mrm{PCAlg}^\mrm{op}_\Fp$ then $\mrm{Spec}(R)^\diamond$ is formalizing.
\item If $X\in \topSch$ then $X^\diamond$ is v-formalizing by \Cref{rem:diamondofreduced}.
\item Non-empty v-formalizing v-sheaves have non-empty reduction. Quasi-separated diamonds are not v-formalizing.
\item If $\F$ formalizes $f:\Spa{A}\to \F$ then $\F$ formalizes $f\circ g$ for any map $g:\Spa{B}\to \Spa{A}$.
\end{enumerate}
\end{pro}

\begin{pro}
	\label{pro:separateduniqueform}
	Let $\F$ be a small v-sheaf, and $f:\Spa{R}\to \F$ a map with $\Spa{R}$ affinoid perfectoid in characteristic $p$. If $\F$ is formally separated then $f$ admits at most one formalization.
\end{pro}
\begin{proof}
	Pick two maps $g_i:\Spdf{R^+}\to \F$ that agree on $\Spa{R}$. Consider $(g_1,g_2):\Spdf{{R}^+}\to \F\times \F$, and the pullback along $\Delta_\F:\F\to \F\times \F$ to get $\G\subseteq \Spdf{R^+}$ a formally closed subsheaf. We prove $\G=\Spdf{R^+}$, it suffices to show $|\G|=|\Spdf{R^+}|$. Moreover, since $|\Spa{R}|\subseteq |\G|$ and $|\Spdf{R^+}|=|\Spa{R}|\cup |\mrm{Spec}(\ovr{R^+})^\diamond|$  it suffices to prove $|(\G\red)^\diamond|= |\mrm{Spec}(\ovr{R^+})^\diamond|$. 
	We first assume $\Hub{R}=\Hub{C}$ for $C$ is a nonarchimedean field and $C^+\subseteq C$ an open and bounded valuation subring. Let $k^+=\ovr{C^+}$ and $k=\mrm{Frac}(k^+)$, then $\mrm{Spec}(k^+)=\Spdf{C^+} \red$ and by \Cref{lem:formallyclosedisZariski} $(\G\red)^\diamond=\mrm{Spec}(k^+/I)^\diamond$ for some ideal $I$. Since $\Spa{C}\subseteq \G$ and $|\G|$ is closed, $|\G|$ contains the formal specialization of $\mrm{Spa}(C,O_C)$, which is the image of $\mrm{Spec}(k)^\diamond$. By formal adicness $|(\G\red)^\diamond|=|\G|\cap |\mrm{Spec}(k^+)^\diamond|$ and we can conclude that $\mrm{Spec}(k)^\diamond\subseteq (\G\red)^\diamond$. This proves $I=\{0\}$ and $(\G\red)^\diamond=\mrm{Spec}(k^+)^\diamond$ in this case. 

	In the general case, for every map $\Spa{C}\to \Spa{R}$ the canonical formalization $\Spdf{C^+}\to \Spdf{R^+}$ factors through $\G$. In particular, after taking reduction, the map $\mrm{Spec}(k^+)\to \mrm{Spec}(\ovr{R^+})$ factors through $\G\red$. This says that $|\G\red|$ contains every point of $|\mrm{Spec}(\ovr{R^+})|$ in the image of the specialization map. By \Cref{lem:formallyclosedisZariski} $\G\red\to \mrm{Spec}(\ovr{R^+})$ is a closed immersion and by \Cref{pro:specialisniceforTate} the specialization map is surjective, these two imply that $\G\red=\mrm{Spec}(\ovr{R^+})$. 
\end{proof}
\begin{pro}
	\label{pro:v-formalizing+separatedstableprod}
	The following statements hold:
\begin{enumerate}
\item Given two maps of v-sheaves $\F\to \Hi$, $\G\to \Hi$ if $\F$ and $\G$ are v-formalizing and $\Hi$ is formally separated then $\F\times_\Hi \G$ is v-formalizing.
\item	The subcategory of v-sheaves that are v-formalizing and formally separated is stable under fiber product and contains $\Zpd$.
\end{enumerate}
\end{pro}
\begin{proof}
	Given a map $\Spa{A}\to \F\times_\Hi \G$ we can find a cover $\Spa{B}\to \Spa{A}$ for which the compositions with the projections to $\F$ and $\G$ are both formalizable. By formal separatedness any pair of choices of formalizations $\Spdf{B^+}\to \G$ and to $\Spdf{B^+}\to \F$ define the same formalization to $\Hi$ and a map to $\F\times_\Hi \G$. The second claim follows from the stability of separatedness by basechange and composition, from \Cref{lem:Zpdseparated} and from \Cref{rem:diamondofreduced}. Indeed, we need to prove that $(\F\red)^\diamond\times_{(\Hi\red)^\diamond} (\G\red)^\diamond$ is a subsheaf of $\F\times_\Hi \G$, but this follows from knowing that $\F\red$ (respectively $\Hi$, $\G$) is a subsheaf of $\F$ (respectively $\Hi$, $\G$).  
\end{proof}
\begin{defi}
	\label{defi:specializingsheaf}
	Let $\F\in\topPerf$, we say it is \textit{specializing} if it is formally separated and v-locally formal. 
\end{defi}
\begin{defi}
	\label{defi:specializationforvsheaf}
	Let $\F$ be a specializing v-sheaf and let $f:\coprod_{i\in I}\Spdf{B_i}\to \F$ be a surjective map. The specialization map for $\F$, denoted $\mrm{sp}_{\F}$, is the unique map $\mrm{sp}_{\F}:|\F|\to |\F\red|$ making the following diagram commutative:
	\begin{center}
		\begin{tikzcd}
			\coprod_{i\in I} \mid  \Spdf{B_i} \mid \ar{r}{f}\ar{d}{\mrm{sp}_{B_i}} & \mid \F \mid \ar{d} \\
			\coprod_{i\in I} \mid  \mrm{Spec}(\ovr{(B_i)}) \mid  \ar{r}{|f\red|} & \mid \F\red \mid  
		\end{tikzcd}
	\end{center}
\end{defi}
\begin{rem}
	\label{rem:constructionofpoints}
	We use \Cref{pro:separateduniqueform} to prove that this map of sets is well defined and does not depend on the choices taken. Indeed, given $[x]\in |\F|$ we take a formalizable representative $x:\Spa{K_x}\to \F$. Consider, its unique formalization $\Spdf{K^+_x}\to \F$ and apply reduction to obtain $\mrm{Spec}(\ovr{(K^+_x)})\to \F^{\mrm{red}}$. The maximal ideal of $\ovr{(K^+_x)}$ maps to $\mrm{sp}_{\F}([x])$. 
\end{rem}

\begin{pro}
	\label{pro:specializationvsheaves}
	For any specializing v-sheaf $\F$ the specialization map $\mrm{sp}_{\F}:|\F|\to |\F\red|$ is continuous. Moreover, this construction is functorial in the category of specializing v-sheaves. 	
\end{pro}
\begin{proof}
	Functoriality follows from uniqueness of formalizations, functoriality of the reduction functor and \Cref{rem:constructionofpoints}.
	For continuity, take a cover $f:\coprod_{i\in I} \Spdf{R^+_i}\to\F$. We get the following commutative diagram:
\begin{center}
\begin{tikzcd}

	\mid \coprod_{i\in I} \Spdf{R^+_i}\mid \arrow{r}{f} \arrow{d}{\mrm{sp}_{R^+_i}} & \mid \F \mid \arrow{d}{\mrm{sp}_{\F}} \\
	\mid \mrm{Spec}(\ovr{(R^+_i)}) \mid \arrow{r}{f\red} & \mid \F\red\mid
\end{tikzcd}
\end{center}
Now, $f\red$ is continuous by \Cref{pro:topologyschematicvsheaf}, $f$ is continuous and a quotient map, and the maps $\mrm{sp}_{R^+_i}$ are continuous by \Cref{pro:specialisniceforTate}. Since the diagram is commutative, the map $\mrm{sp}_{\F}$ is also continuous.
\end{proof}

\subsection{Pre-kimberlites, formal schemes and formal neighborhoods}
\begin{defi}
	\label{defi:pre-Kimberlites}
	Let $\F$ be a specializing v-sheaf. We say $\F$ is a \textit{prekimberlite} if:
	\begin{enumerate}
	\item $\F\red$ is represented by a scheme. 
	\item The map $(\F\red)^\diamond \to \F$ coming from adjunction is a closed immersion.
\end{enumerate}
If $\F$ is a prekimberlite, we let the \textit{analytic locus} be $\F^{\mrm{an}}=\F\setminus (\F\red)^\diamond$. 
\end{defi}
 


In what follows, we prove that the v-sheaf associated to separated formal schemes are prekimberlites.  
For this we fix a convention of what we mean by a ``formal scheme''. We follow \cite[\S 2.2]{SW}.
\begin{conv}
	\label{conv:formalschemes}
	Denote by $\mrm{Nilp}_\Zp$ the category of algebras in which $p$ is nilpotent, and endow $\mrm{Nilp}^\mrm{op}_\Zp$ with the structure of a site by giving it the Zariski topology. By a formal scheme $\frak{X}$ over $\Zp$ we mean a Zariski sheaf on $\mrm{Nilp}^\mrm{op}_\Zp$ which is Zariski locally of the form $\mrm{Spf}(A)$. Here $A$ is a topological ring given the $I$-adic topology for a finitely generated ideal of $A$ containing $p$, and $\mrm{Spf}(A)$ denotes the functor $\mrm{Spec}(B)\mapsto \varinjlim_n \mrm{Hom}(A/I^n,B)$.

	For a formal scheme $\frak{X}$ over $\Zp$ we let $\ovr{\frak{X}}$ denote its reduction in the sense of formal schemes (\cite[Tag 0AIN]{Stacks}). Recall that this is a sheaf in $\mrm{Nilp}^\mrm{op}_\Zp$ which is representable by a scheme. Moreover, the map $\ovr{\frak{X}}\to \frak{X}$ is relatively representable in schemes, it is a closed immersion and for any open $\mrm{Spf}(A)\subseteq \frak{X}$ the pullback to $\ovr{\frak{X}}$ is given by the reduced subscheme of $\mrm{Spec}(A/I)$ (for an ideal of definition $I\subseteq A$).

	We say that $\frak{X}$ is separated if $\ovr{\frak{X}}$ is a separated scheme (\cite[Tag 0AJ7]{Stacks}).
\end{conv}

Recall the following result of Scholze and Weinstein.\footnote{What is called adic spaces in \cite{SW} is what we call pre-adic spaces here and in \cite{Ber}.}
\begin{prop}\textup{(\cite[Proposition 2.2.1]{SW})}
	\label{pro:SWfullyfaithfulformal}
	The functor $\mrm{Spf}(A)\mapsto \mrm{Spa}(A,A)$ extends to a fully faithful functor $\frak{X}\mapsto \frak{X}^{\mrm{ad}}$ from formal schemes over $\Zp$ as in \Cref{conv:formalschemes} to the category of pre-adic spaces. 
\end{prop}

\begin{prop}
	\label{exa:formalispre-Kimberlite}
	If $\frak{X}$ is a separated formal scheme over $\Zp$, then $(\frak{X}^{\mrm{ad}})^\diamondsuit$ is a prekimberlite.
\end{prop}
\begin{proof}
	Let $X=\frak{X}^{\mrm{ad}}$ and let $W=X^{\mrm{na}}$, then $W=(\frak{X}_{\mrm{red}})^{\mrm{ad}}$. Clearly $X^\diamondsuit$ is v-locally formal. 
	By \Cref{pro:tworeds} we have $(W^\diamondsuit)\red=(X^\diamondsuit)\red$ which is the perfection of $\frak{X}_\mrm{red}$. The adjunction morphism agrees with the map $W^\diamondsuit\to X^\diamondsuit$ which by \Cref{rem:diamondconstr} is a closed immersion. 
	
	The only thing left to prove is that $X^\diamondsuit\to \Zpd$ is separated, we first prove that $X^\diamondsuit$ is quasiseparated. Let $Z=\Spa{R}$ be a strictly totally disconnected space and take a map $f:Z\to X^\diamondsuit\times_\Zpd X^\diamondsuit$. Since $Z$ splits any open cover we may assume that $f$ factors through an open neighborhood of the form $\Spdf{B_1}\times_\Zpd \Spdf{B_2}$ for an open subset $\mrm{Spf}(B_1)\times_{\mrm{Spf}(\Zp)} \mrm{Spf}(B_2)\subseteq \mathfrak{X}\times_{\mrm{Spf}(\Zp)}\frak{X}$. Consider the following basechange diagrams, where $Y=\frak{Y}^{\mrm{ad}}$
	\begin{center}
		\begin{tikzcd}
			\frak{Y}\ar{r}\ar{d} & \mrm{Spf}(B_1)\times_{\mrm{Spf}(\Zp)} \mrm{Spf}(B_2)\ar{d} &
			{Y}\ar{r}\ar{d} & \Spf{B_1}\times_{\Zp} \Spf{B_2}\ar{d} \\

			\frak{X} \ar{r} & \frak{X}\times_{\mrm{Spf}(\Zp)} \frak{X} &
			X \ar{r} & X\times_{\Zp} X
		\end{tikzcd}
	\end{center}
	Since $\frak{X}$ is separated $\frak{Y}$ is quasicompact. This implies that $Y$ admits a finite open cover of the form $\coprod_{i=1}^n \Spf{A_i}\to Y$. Moreover, the diagonal map $X\to X\times_{\Zp} X$ is adic. 
	By \Cref{lem:boundedlocusgivesqcqs} the maps $\Spdf{A_i}\to \Spdf{B_1}\times_\Zp \Spdf{B_2}$ are quasicompact, which proves that $Y^\diamondsuit\to \Spdf{B_1}\times_\Zp \Spdf{B_2}$ and any basechange of it is also quasicompact. Now we may use the valuative criterion of separatedness \cite[Proposition 10.9]{Et}. Given $\mrm{Spa}(K,O_K)\to X^\diamondsuit$ we must show there is at most one extension to $\mrm{Spa}(K,K^+)\to X^\diamondsuit$ where $K^+\subseteq O_K$ is an open and bounded valuation subring. Maps $\mrm{Spa}(K,K^+)\to X^\diamondsuit$ are in bijection with maps $\mrm{Spf}(K^+)\to \frak{X}$. On the other hand, maps $g:\mrm{Spf}(K^+)\to \frak{X}$ are in bijection with pairs $(g_\eta,g_s)$ where $g_\eta:\mrm{Spf}(O_K)\to \frak{X}$, $g_s:\mrm{Spec}(K^+/K^{\circ \circ})\to \frak{X}_\mrm{red}$ and such that $g_\eta=g_s$ when we restrict the maps to $\mrm{Spec}(O_K/K^{\circ \circ})$. 
	At this point we may use the valuative criterion of separatedness of $\frak{X}_\mrm{red}$.
\end{proof}

\begin{defi}
	\label{defi:tubularneigh}
	Let $\F$ be a prekimberlite and let $S\subseteq \F\red$ be a locally closed immersion of schemes. We let $\Tubf{\F}{S}$, the \textit{formal neighborhood} of $S$ on $\F$, be the subsheaf given by the Cartesian diagram:
\begin{center}
\begin{tikzcd}
	\Tubf{\F}{S}  \ar{rr} \ar{d}  & &   \und{\mid S \mid} \ar{d} \\
	\F \ar{r} & \und{\mid \F \mid} \ar{r}{\mrm{sp}_{\F}} &   \und{\mid \F\red \mid}\\
\end{tikzcd}
\end{center}
If $S\subseteq \F\red$ is open we call it \textit{open formal neighborhood}.
\end{defi}
\begin{pro}
	\label{pro:twoformalneighb}
	Suppose $\Huf{A}$ is a formal Huber pair over $\Zp$ with ideal of definition $I$. Let $J\subseteq A$ be a finitely generated ideal containing $I$ and $B$ the completion of $A$ with respect to $J$. The closed immersion of schemes, $S=\mrm{Spec}(\ovr{B})\to \mrm{Spec}(\ovr{A})$, induces an identification $\Tubf{\Spdf{A}}{S}=\Spdf{B}$. 
\end{pro}
\begin{proof}
	Let $S=\mrm{Spec}(\ovr{B})$ and $T=\mrm{Spec}(\ovr{A})$. The reduction of $\Spdf{B}\to \Spdf{A}$ induces $S\to T$. Since specialization is functorial, points coming from $\Spdf{B}$ specialize to $S$. Consequently, the map factors as $\Spdf{B}\to \Tubf{\Spdf{A}}{S}\to \Spdf{A}$. Since $A$ is dense in $B$, this map is an injection. To prove surjectivity onto $\Tubf{\Spdf{A}}{S}$, let $f:A\to R^+$ be a map such that $f:\mrm{Spec}(\ovr{R^+})\to \mrm{Spec}(\ovr{A})$ factors through $|S|$. If $a\in J$, then $f(a)$ is nilpotent in $\mrm{Spec}(R^+/\varpi^n)$. Since $J$ is finitely generated there is an $m$ for which $J^m\subseteq (\varpi^n)$ in $R^+$. This proves that the map $f:A\to R^+$ is continuous for the $J$-adic topology on $A$. Since $R^+$ is complete the map $f:A\to R^+$ factors through $B$. 
\end{proof}
\begin{pro}
	\label{pro:pullbackformal}
	Let $f:\G\to\F$ be a map of prekimberlites and $S\subseteq |\F\red|$ a locally closed subscheme. Let $T=S\times_{\F\red} \G\red$, then $\Tubf{\F}{S}\times_\F \G=\Tubf{\G}{T}$. In particular, $\G\to \F$ factors through $\Tf{\F}{S}$ if and only if $\G\red\to \F\red$ factors through $S$.
\end{pro}
\begin{proof}
	Since $S$ is a locally closed $|T|=|S|\times_{|\F\red|} |\G\red|$. The rest is a standard diagram chase. 
%
\end{proof}

\begin{pro}
	\label{pro:tubularalmostpre-Kimb}
	Let $\F$ be a prekimberlite and let $S\subseteq |\F\red|$ a locally closed subset, then $\Tubf{\F}{S}$ is a prekimberlite and $(\Tubf{\F}{S})\red=S$. 
\end{pro}
\begin{proof}

	The formula $(\Tubf{\F}{S})\red=S$ follows from observing that by \Cref{pro:pullbackformal} a map $\mrm{Spec}(A)^\diamond\to \F$ factors through $\Tf{\F}{S}$ if and only if the adjunction map $\mrm{Spec}(A)\to \F\red$ factors through $S$. 
	Since $\Tubf{\F}{S}$ is a subsheaf of a separated v-sheaf it is separated as well. The map $S^\diamond\to \Tubf{\F}{S}$ is injective, so $\Tubf{\F}{S}$ is formally separated. Take a map $\Spa{R}\to \Tubf{\F}{S}\subseteq \F$. After replacing $\Spa{R}$ by a v-cover we get a formalization $\Spdf{R^+}\to \F$. By \Cref{pro:pullbackformal} this factors through $\Tf{\F}{S}$ since $\mrm{Spec}(\ovr{R^+})\to \F\red$ factors through $S$. 
	We have proved $\Tubf{\F}{S}$ is specializing, we prove $S^\diamond\to \Tubf{\F}{S}$ is a closed immersion. 
	Consider $\G=\Tubf{\F}{S}\times_\F (\F\red)^\diamond$. 
	Now, $\G\to \Tubf{\F}{S}$ is a closed immersion, and by \Cref{pro:pullbackformal} $\G$ is $\Tubf{((\F\red)^\diamond)}{S}$. 
	If $S$ is a closed subscheme of $\F\red$, then $S^\diamond\to (\F\red)^\diamond$ is proper, so $S^\diamond \to \Tf{((\F\red)^\diamond)}{S}$ is a closed immersion. 
	If $S$ is an open in $\F\red$, then $\Tubf{((\F\red)^\diamond)}{S}=S^\diamond$. 
	By definition a locally closed subset $S\subseteq |\F\red|$ can regarded as a composition $S\to U\to \F\red$ where $U\to \F$ is an open immersion and $S\to U$ is a closed immersion. 
	In this case $\Tubf{(\Tubf{\F}{U})}{S}=\Tubf{\F}{S}$, applying the argument once to $U\subseteq |\F\red|$ and once to $S\subseteq U$ we get the result.
\end{proof}
\begin{pro}
	\label{pro:constgivesopen}
	Let $\F$ be a prekimberlite, $S\subseteq |\F\red|$ a locally closed constructible subset, then the map $\Tf{\F}{S}\to \F$ is an open immersion.
\end{pro}
\begin{proof}
	The question is Zariski local in $\F\red$. Indeed, an open cover $\coprod_{i\in I} U_i\to \F\red$ induces an open cover $\coprod_{i\in I} \Tf{\F}{U_i}\to \F$. We may assume that $\F\red=\mrm{Spec}(A)$ and that $S$ is closed and constructible in $\mrm{Spec}(A)$. Write $S=\mrm{Spec}(A/I)$ for $I\subseteq A$ an ideal, by constructibility we may assume that $I$ is finitely generated. Pick $\{i_1,\dots, i_n\}$ a list of generators for $I$, $\Hub{R}\in \mrm{Perf}$ and a map $\Spdf{R^+}\to \F$. Let $X:=\Spdf{R^+}\times_\F \Tf{\F}{S}$, and let $\varpi\in R^+$ be a pseudo-uniformizer. Let $\{j_1,\dots,j_n\}$ be a list of lifts of $\{i_1,\dots i_n\}$ to $R^+$. Then $X$ is the open subsheaf of $\Spdf{R^+}$ defined by $\bigcap_{k=1}^n \N{j_k}{1}$. 
	Indeed, this follows from \Cref{pro:pullbackformal}, \Cref{pro:twoformalneighb} and \Cref{lem:representinganalyticnbhoods}. 
	Since $\F$ is v-formalizing every map $\Spa{R}\to \F$ factors through $\Spdf{R^+}$ after replacing $\Spa{R}$ by a v-cover. By \cite[Proposition 10.11]{Et}  $\Tf{\F}{S}\to \F$ is open.
\end{proof}

\subsection{Heuer's specialization map and \'etale formal neighborhoods}
In \cite{heuer21}, Heuer considers certain specialization maps. These are maps of v-sheaves rather than a map of topological spaces. We discuss his construction and use it to enhance our theory.

\begin{defi}\textup{(\cite[Definition 5.1]{heuer21})}
	Let $X$ be a scheme over $\Fp$. We attach a presheaf $\Heuer{X}$ defined by the (analytic sheafification of the) formula $\Hub{R}\mapsto X(\mrm{Spec}(\ovr{R^+}))$, where $\ovr{R^+}=R^+/R^{\circ \circ}$.
\end{defi}

In \cite[Lemma 5.2]{heuer21}, Heuer proves that $\Heuer{X}$ is a v-sheaf and that when $X$ is affine the sheafification is not necessary. There is an evident map $X^\diamond\to \Heuer{X}$.
\begin{prop}
	\label{pro:correctreductionHeuer}
	If $X$ is a perfect scheme over $\Fp$, then $X=(\Heuer{X})\red$. Moreover, if $\Spa{R}$ is a totally disconnected perfectoid space then $\Heuer{X}(\Spa{R})=\Heuer{X}(\Spdf{R^+})$. 
\end{prop}
\begin{proof}
	Let $X=\mrm{Spec}(A)$. 
	The inclusion $X\subseteq (\Heuer{X})\red$ of scheme-theoretic v-sheaves is easy to verify. 
	Now, $\Heuer{X}(\mrm{Spec}(R)^\diamond) \subseteq \Heuer{X}(\mrm{Spd}(R\rpot{t\pthr},R\pot{t\pthr})$ and this latter is by \cite[Lemma 5.2]{heuer21} the set of maps $A\to R$. So $(\Heuer{X})\red = X$.
	Moreover, let $\Spa{R}\in \mrm{Perf}$, with pseudo-uniformizer $\varpi$. Consider $U=(\Spdf{R^+\pot{t\pthr}})^{\mrm{an}}$ with its cover $\Spa{R_1}=U(\frac{\varpi}{t})$ and $\Spa{R_2}=U(\frac{t}{\varpi})$. Let $\Spa{R_3}=U(\frac{t}{\varpi})\cap U(\frac{\varpi}{t})$. One computes explicitly that $\Heuer{X}(\Spa{R_i})=\Heuer{X}(\Spa{R})$ for $i\in \{1,2,3\}$, so $\Heuer{X}(U)=\Heuer{X}(\Spa{R})$. This proves $\Heuer{X}(\Spdf{R^+})= \Heuer{X}(\Spa{R})$ when $X$ is affine. Using the techniques of \Cref{pro:perfectschemesvsadicultimate} we can glue and prove the general case. Indeed, open subschemes $f:U\subseteq X$ induce open subsheaves $\Heuer{f}:\Heuer{U}\subseteq \Heuer{X}$ and the delicate part of the glueing happens on $|\mrm{Spec}(R)^\diamond|$ for test objects $R$, which is easily reduced to $R=V$ a valuation ring. 
Following the proof loc. cit. we get to a similar diagram:
	\begin{center}
		\begin{tikzcd}
			
			\mrm{Spa}(V_b[\frac{1}{b}],V_b)\ar{rr} \ar{dd}\ar{rd}		&  &	\mrm{Spd}(V[\frac{1}{b}],V) \ar{d} \\
			& \Heuer{\mrm{Spec}(B_3)}\ar{r} \ar{d}			& \Heuer{\mrm{Spec}(B_2)}\ar{d}	\\	
	 		\mrm{Spd}(V_b,V_b)\ar{r}		& 	\Heuer{\mrm{Spec}(B_1)}\ar{r}&	X^\dia
		\end{tikzcd}
	\end{center}
	Now, $\Heuer{\mrm{Spec}(B_1)}( \mrm{Spa}(V_b[\frac{1}{b}],V_b)=\Heuer{\mrm{Spec}(B_1)}(\mrm{Spa}(V_b,V_b))$ proves $\Spdf{V_b}$ factors through $\Heuer{\mrm{Spec}(B_3)}$ and $\mrm{Spec}(V)^\diamond$ factors through $\Heuer{\mrm{Spec}(B_2)}$. This proves $(\Heuer{X})\red=X$. It also proves that for perfectoid fields $\Hub{K}$ and a map $\Spdf{K^+}\to \Heuer{X}$ if $\Spa{K}\to \Heuer{X}$ factors through an affine then $\Spdf{K^+}$ also does. Since totally disconnected perfectoid spaces split open covers we can conclude $\Heuer{X}(\Spa{R})=\Heuer{X}(\Spdf{R^+})$. 
\end{proof}

Suppose $X$ is a prekimberlite and $\Spa{R}\in \mrm{Perf}$. Let $f:\Spa{R}\to X$ be a formalizable map, applying reduction to the formalization we obtain a map $\mrm{Spec}(\ovr{R^+})\to X\red$, or in other words a map $\Spa{R}\to \Heuer{(X\red)}$. Overall, we obtain a natural transformation $\mrm{SP}_X:X\to \Heuer{(X\red)}$. This is the type of specialization map that Heuer considers. With this switch of perspective we can reinterpret formal neighborhoods: if $S\subseteq |X\red|$ is a locally closed subset we get a map $\Heuer{S}\to \Heuer{(X\red)}$ and $\Tf{\F}{S}=X\times_{\Heuer{(X\red)}}\Heuer{S}$. Under this light, \Cref{pro:constgivesopen} is simply proving that $\Heuer{S}\to \Heuer{(X\red)}$ is an open immersion when $S$ is constructible. Moreover, this leads to a good notion of ``\'etale formal neighborhoods" of a prekimberlite.  

\begin{lem}
	\label{pro:etalepreservesformaladic}
	If $V\to X$ is a quasicompact, separated \'etale map of perfect schemes over $\Fp$, then $\Heuer{V}\to \Heuer{X}$ is separated, quasicompact, formally adic and \'etale.
\end{lem}
\begin{proof}
	The formation of $\Heuer{X}$ commutes with finite limits and preserves open immersions. If $V\to X$ is separated then $V\to V\times_X V$ is open and closed, the same holds for $\Heuer{V}\to \Heuer{V}\times_{\Heuer{X}} \Heuer{V}$ which proves separatedness. 
	We prove formal adicness. Let $Y=\Heuer{V}\times_{\Heuer{X}} X^\diamond$, we get a map $V^\diamond \to Y$. We prove this map is an isomorphism after basechange by any map $\Spa{R}\to X^\diamond$.
	Let $Y_R:=Y\times_{X^\diamond}\Spa{R}$, we may assume $m$ factors through $\mrm{Spec}(R^+)^\diamond\to X^\diamond$ since this happens v-locally. Now, before sheafification $Y_R\Hub{S}$ parametrizes lifts of $\mrm{Spec}(\ovr{S^+})\to \mrm{Spec}(\ovr{R^+})\to X$ to $V$. By invariance of the \'etale site under perfection (respectively nilpotent thickenings), this is the same as parametrizing lifts of $\mrm{Spec}(S^+/\varpi)\to X$ to $V$ (respectively of $\mrm{Spf}(S^+)\to X$ to $V$). In other words, $Y_R$ fits in the following Cartesian diagram:
\begin{center}
\begin{tikzcd}
	Y_R \ar{r}\ar{d} & \Spa{R} \ar{d} \\
	\mrm{Spec}(R^+)^\diamond \times_{X^\diamond} V^\diamond \ar{r} & \mrm{Spec}(R^+)^\diamond
\end{tikzcd}
\end{center}
But this is precisely $V^\diamond\times_{X^\diamond}\Spa{R}$. For quasicompactness and \'etaleness we can argue locally and assume $X=\mrm{Spec}(A)$ and $V=\mrm{Spec}(B)$. Indeed, this follows from the quasicompactness of the specialization map for Tate Huber pairs. Arguing as above, for a map $\Spa{R}\to \Heuer{X}$ the basechange is represented by $\Spa{S}$ where $S^+$ is the unique \'etale $R^+$-algebra lifting $\ovr{R^+}\otimes_A B$.   
\end{proof}

\begin{defi}
	Suppose $X$ is a prekimberlite, we let $(X)_{\mrm{qc}, \mrm{for\text{-}\acute{e}t}}$ be the category that has as objects maps $f:T\to X$ where $T$ is a prekimberlite and $f$ is formally adic, \'etale and quasicompact. Morphisms are maps of v-sheaves commuting with the structure map. We call objects in this category the \textit{\'etale formal neighborhoods} of $X$. 
\end{defi}
Morphisms in $X_{\mrm{for\text{-}\acute{e}t}}$ are automatically quasicompact, formally adic, \'etale and separated. Given a perfect scheme $S$ we let $(S)_{\mrm{\acute{e}t},\mrm{qc},\mrm{sep}}$ denote the category of schemes \'etale, quasicompact and separated over $S$.
\begin{thm}
	\label{thm:invarianceofetalesite}
	For $X$ a prekimberlite, reduction $(-)\red:(X)_{\mrm{qc},\mrm{for\text{-}\acute{e}t}}\cong (X\red)_{\mrm{qc},\mrm{\acute{e}t},\mrm{sep}}$ is an equivalence.
\end{thm}
\begin{itemize}
	\item[Step 1:] If a morphism $f:Y\to W$ in $(X)_{\mrm{qc},\mrm{for\text{-}\acute{e}t}}$ induces an isomorphism $f^{\mrm{red}}:Y^{\mrm{red}}\to W^{\mrm{red}}$ then $f$ is an isomorphism.
\begin{proof}
	The sheaf-theoretic image of $Y$ in $W$ is an open subsheaf of $W$ containing $W^\mrm{red}$ so it must be $W$. This proves surjectivity. Since the map is qcqs we may prove injectivity on geometric points. We reduce to the case where $W=\mrm{Spd}(C^+,C^+)$, and the geometric point is given by the inclusion $\mrm{Spa}(C,C^+)\subseteq \mrm{Spd}(C^+)$. In this case, $Y^\mrm{an}$ has the form $\coprod^n_{i=1} \mathrm{Spa}(C,C'^+)$ with $C^+\subseteq C'^+\subseteq O_C$. To count connected components we may restrict to $W=\mrm{Spd}(O_C,O_C)$ so that $Y^\mrm{an}=\coprod^n_{i=1} \mathrm{Spa}(C,O_C)$ and $Y^\mrm{red}=\mathrm{Spec}(O_C/C^{\circ \circ})=W^\mrm{red}$. Replacing $C$ by a v-cover we may assume that $\mrm{Spa}(C,O_C)\to Y$ is formalizable. The map $\mrm{Spd}(O_C,O_C)\to Y$ is \'etale, and since its image is open $n=1$. 
\end{proof}

\item[Step 2:] $(-)\red:(X)_{\mrm{qc},\mrm{for\text{-}\acute{e}t}}\to (X\red)_{\mrm{qc},\mrm{\acute{e}t},\mrm{sep}}$ is fully-faithful.
	\begin{proof}
		Maps from $Z$ to $Y$ over $X$ are in bijection with open and closed subsheaves $W\subseteq Z\times_X Y$ whose projection to $Z$ is an isomorphism. Identically, maps $Z\red\to Y\red$ are in bijection with open and closed subschemes of $Z\red\times_{X\red} Y\red$ whose projection to $Z\red$ is an isomorphism. Since we can check isomorphisms by passing to reduction we get a bijection $W\mapsto W\red$ with inverse $V\mapsto \Tf{(Z\times_X Y)}{V\red}$.
	\end{proof}
\item[Step 3:] $(-)\red:(X)_{\mrm{qc},\mrm{for\text{-}\acute{e}t}}\to (X\red)_{\mrm{qc},\mrm{\acute{e}t},\mrm{sep}}$ is essentially surjective.
	\begin{proof}
		If $V\in (X\red)_{\mrm{qc},\mrm{\acute{e}t},\mrm{sep}}$ we let $\Tf{X}{V}:=X\times_{\Heuer{(X\red)}}\Heuer{V}$, where the map $X\to \Heuer{(X\red)}$ is $\mrm{SP}_X$. By \Cref{pro:etalepreservesformaladic}, $\Tf{X}{V}\to X$ is quasicompact, formally adic, \'etale and separated. Since $X$ is formally separated $\Tf{X}{V}$ also is. Although they are not formally separated, by \Cref{pro:correctreductionHeuer}, $\Heuer{(X\red)}$ and $\Heuer{V}$ still formalize uniquely totally disconnected spaces, this proves $\Tf{X}{V}$ is also v-formalizing and a specializing v-sheaf. It is a prekimberlite since $(\Tf{X}{V})\red=V$, and the map $V^\diamond\to \Tf{X}{V}$ is closed. 
	\end{proof}
\end{itemize}
If $X$ is a prekimberlite and $V\to X\red$ is an \'etale map, \Cref{thm:invarianceofetalesite} associates to $V$ the \'etale formal neighborhood $\Tf{X}{V}:=X\times_{\Heuer{X}}\Heuer{V}$.

\begin{cor}
Let $X$ be a prekimberlite and let $Y\in (X)_{\mrm{qc},\mrm{for\text{-}\acute{e}t}}$. If $X=\frak{X}^\dia$ for a formal scheme $\frak{X}$, then there is a formal scheme $\frak{Y}$ with $Y=\frak{Y}^\dia$. 
\end{cor}
\begin{rem}
	\label{rem:naivenearbycycles}
	If $\Ki=(\F,\Di)$ is a smelted kimberlite, by \Cref{thm:invarianceofetalesite}, we obtain a morphism of sites $\Psi':\Di_{\mrm{\acute{e}t}}\to (\F\red)_{\mrm{\acute{e}t}}$. This allows us to form a ``naive nearby cycles functor" $\mrm{R}\Psi'$. If $j:\Di\to \F \leftarrow (\F\red)^\diamond:i$ denote the inclusions, Scholze's $6$-functor formalism give us already a nearby cycles functor $i^*\mrm{R}j_*$. It is an interesting question to understand the relation between these two functors. We have partial progress in answering this question. We will report our findings on a future work. 
\end{rem}

Observe that $\mrm{SP}_X:X\to \Heuer{(X\red)}$ is separated. This allows us to make the following definition.
\begin{defi}
	A prekimberlite is \textit{valuative} if $\mrm{SP}_X:X\to \Heuer{(X\red)}$ is partially proper. 
\end{defi}
\begin{prop}
	If $\frak{X}$ is a separated formal scheme over $\Zp$ as in \Cref{conv:formalschemes}, then $\frak{X}^\dia$ is valuative.
\end{prop}
\begin{proof}
	By \cite[Proposition 18.6]{Et}, one can verify partial properness open locally on the target, this reduces to the case $\frak{X}=\Spdf{B}$. The valuative criterion asks if given a map $B\to R^{\sharp,\circ}$ such that $\ovr{B}\to \ovr{R^{\sharp,\circ}}$ factors through $\ovr{R^{\sharp,+}}$ then $B\to R^{\sharp,\circ}$ factors through $R^{\sharp,+}$. This follows from $R^{\sharp,+}=R^{\sharp,\circ}\times_{\ovr{R^{\sharp,\circ}}} \ovr{R^{\sharp,+}}$. 
\end{proof}
\begin{prop}
	\label{pro:partiallyproperisvaluative}
Let $f:X\to Y$ be a map of prekimberlites. If $f$ is partially proper and $Y$ is valuative then $X$ is valuative. 
\end{prop}
\begin{proof}
	It follows from two facts: $X\to \Heuer{(Y\red)}$ is partially proper and $X\to \Heuer{(X\red)}$ is separated. 	
\end{proof}
\begin{prop}
	\label{pro:valuativegivesspecializing}
	If $\F$ is a valuative prekimberlite then $\mrm{sp}_{\F}:|\F|\to |\F\red|$ is specialising as a map of topological spaces. 
\end{prop}
\begin{proof}
Let $r\in |\F|$, $x=\mrm{sp}_{\F}(r)$ and $y\in |\F\red|$ specializing from $x$. We construct $q$ specializing from $r$ that maps to $y$. Pick a formalizable representative $f_r:\Spa{C}\to \F$. Let $K=O_C/C^{\circ \circ}$ and $K^+=C^+/C^{\circ \circ}$, then $x$ is the image of closed point under $f_x:\mrm{Spec}(K^+)\to \F\red$. Let $R$ be the local ring obtained by intersecting the closure of $x$ and the localization at $y$. Let $k=K^+/\frak{m}_{K^+}$, so that $R\subseteq k$. By \cite[Tag 00IA]{Stacks}, we have a valuation subring $R\subseteq V\subseteq k$ such that $\mrm{Frac}(V)=k$ and $V$ dominates $R$. This induces a valuation subring $K'^+\subseteq K^+$  and a map $f_y:\mrm{Spec}(K'^+)\to \F\red$ whose closed point maps to $y$. In turn, this induces a valuation subring $C'^+\subseteq C^+$ with $C'^+/C^{\circ \circ}=K'^+$. Now, $f_y$ induces a map $\mrm{Spa}(C,C'^+)\to \Heuer{(\F\red)}$ extending $\mrm{SP}_\F\circ f_r$. By partial properness this lifts to a map $f_q:\mrm{Spa}(C,C'^+)\to \F$ and clearly $\mrm{sp}_{\F}(q)=y$. 
\end{proof}
 
\begin{prop}
	\label{pro:preservevaluative}
Formal neighborhoods and \'etale formal neighborhoods of valuative prekimberlites are valuative prekimberlites.
\end{prop}
\begin{proof}
This is immediate from their expression as a basechange.	
\end{proof}

\subsection{Kimberlites and smelted kimberlites}

\begin{defi}
	Let $\F$ be a valuative prekimberlite. 
	\begin{enumerate}
		\item A \textit{smelted kimberlite} $\Ki$ is a pair $\Ki:=(\F,\Di)$, where $\Di \subseteq \F^{\mrm{an}}$ is an open subsheaf such that $\Di$ is a quasiseparated locally spatial diamond, and such that the map $\Di\to \F$ is partially proper. 
		\item We define the specialization map $\mrm{sp}_{\mathcal{K}}:|\Di|\to |\F\red|$ as the composition $|\Di|\to |\F|\xrightarrow{\mrm{sp}_{\F}} |\F\red|$. If the context is clear, we write $\mrm{sp}_\Di$ instead of $\mrm{sp}_\Ki$.
		\item We say $\G$ is a \textit{kimberlite} if $\Ki_{\G}:=(\G,\G^\mrm{an})$ is a smelted-kimberlite and $\mrm{sp}_{\Ki_\G}$ is quasicompact. 
	\end{enumerate}
\end{defi}

\begin{rem}
	Given a valuative prekimberlite $\F$ one is mostly interested in smelted kimberlites $(\F,\Di)$ where $\Di=\F^\mrm{an}$ or where $\Di=\F\times_{\mrm{Spd}(O_K)}\mrm{Spd}(K,O_K)$ when $\F$ comes with a map $\F\to \mrm{Spd}(O_K)$ for $O_K$ a complete rank $1$ valuation ring. Notice that $\F^\mrm{an}\to \F$ is always partially proper if $\F$ is a prekimberlite.
\end{rem}

\begin{rem}
	Let $\F$ be a kimberlite. The quasicompactness hypothesis of $\mrm{sp}_{\F^\mrm{an}}$ is equivalent to asking that $\mrm{sp}_{\F^\mrm{an}}^{-1}(U)$ is a spatial diamond for all $U\subseteq \F\red$ with $U$ affine. 
\end{rem}

\begin{defi}
	Let $\Ki=(\F,\Di)$ be a smelted kimberlite. Given a constructible locally closed subset $S\subseteq \F\red$ we let $\Tup{\Di}{S}=\Tf{\F}{S}\times_\F \Di$. We call this subsheaf the \textit{tubular neighborhood} of $\Di$ around $S$. If $\G$ is a kimberlite we write $\Tup{\G}{S}$ for $\Tf{\G}{S}\times_\G \G^\mrm{an}$.
\end{defi}
\begin{prop}
	\label{tubularsmeltedkimberlites}
	If $\Ki=(\F,\Di)$ is a smelted kimberlite, then $(\Tf{\F}{S},\Tup{\Di}{S})$ is a smelted kimberlite. Moreover, $\Tup{\Di}{S}$ is the open subdiamond corresponding to the interior of $\mrm{sp}_{\Ki}^{-1}(S)$ in $|\Di|$.
\end{prop}
\begin{proof}
	By \Cref{pro:preservevaluative} $\Tf{\F}{S}$ is a valuative prekimberlite, and by basechange $\Tup{\Di}{S}\to \Tf{\F}{S}$ is partially proper open. Moreover, \Cref{pro:constgivesopen} $\Tup{\Di}{S}\to \Di$ is an open immersion, so $\Tup{\Di}{S}$ is a quasiseparated locally spatial diamond.  
	For the second claim, let $T\subseteq \mrm{sp}_{\Ki}^{-1}(S)$ be the largest subset stable under generization. It suffices to prove $T\subseteq \Tup{\Di}{S}$ since by \Cref{pro:constgivesopen} we already have $\Tup{\Di}{S}\subseteq (\mrm{sp}_{\Ki}^{-1}(S))^\mrm{int}\subseteq T$. 
	Take $x\in T$ and a formalizable geometric point $\iota_x:\Spa{C_x}\to \F$ over $x$. Since every generization of $x$ is in $\mrm{sp}_{\Ki}^{-1}(S)$ the map $\mrm{Spec}(\ovr{(C_x^+)})\to \F\red$ factors through $S$, so $\iota_x$ factors through $\Tup{\Di}{S}=\Tf{\F}{S}\cap\Di$ by \Cref{pro:pullbackformal}. 
\end{proof}


\begin{thm}
	\label{pro:spectralmapSch-Spatial}
	Let $\Ki=(\F,\Di)$ be a smelted kimberlite and $\G$ be a kimberlite, the following hold: 
	\begin{enumerate}
		\item $\mrm{sp}_\Di:|\Di|\to |\F\red|$ is a specializing, spectral map of locally spectral spaces. 
		\item $\mrm{sp}_{\G^\mrm{an}}:|\G^\mrm{an}|\to |\G\red|$ is a closed map.
	\end{enumerate}
\end{thm}
\begin{proof}
	Being specializing follows from \Cref{pro:valuativegivesspecializing} and the hypothesis that $\Di\to \F$ is partially proper. The second claim follows easily from \Cref{cor:closedmap}, from the first claim and from quasicompactness of $\mrm{sp}_{\G^\mrm{an}}$. We need to prove that $\mrm{sp}_\Di$ is continuous for the constructible topology. We can argue on an open cover of $|\Di|$, so we may assume that $\Di$ is spatial. Find a formalizable cover $X\to\Di$ with $X=\Spa{A}\in \mrm{Perf}$, and consider the diagram:
\begin{center}
\begin{tikzcd}

	\mid \Spa{A}\mid ^{\mrm{cons}}\arrow{r}{g} \arrow{d}{\mrm{sp}_{X}} & \mid \Di \mid^{\mrm{cons}}\arrow{d}{\mrm{sp}_{\Di}} \\
	\mid \mrm{Spec}(\ovr{A^+}) \mid ^{\mrm{cons}}\arrow{r}{g\red} & \mid \F\red \mid^{\mrm{cons}}
\end{tikzcd}
\end{center}
Since $\F\red$ is represented by a scheme, by \Cref{pro:reducedgivesfullyfaithful} $g\red$ is continuous for the patch topology. Similarly, $\mrm{sp}_{X}$ is continuous and since $\Di$ is spatial by \Cref{pro:spatial-spectral} $g$ is also continuous. Moreover, $g$ is a surjective map of compact Hausdorff spaces and consequently a quotient map. Since the diagram commutes, $\mrm{sp}_{\Ki}$ is continuous for the patch topology.
\end{proof}

\begin{pro}
	\label{lem:closedsubschemeisevident}
	Let $f:\G\to \F$ be a formally closed immersion. The following hold:
	\begin{enumerate}
		\item If $\F$ is a specializing v-sheaf, then $\G$ is a specializing v-sheaf.
		\item If $\F$ is a prekimberlite, then $\G$ is a prekimberlite.
		\item If $(\F,\Di)$ forms a smelted kimberlite then $(\G, \G\cap \Di)$ forms a smelted kimberlite. 
		\item If $\F$ is a kimberlite, then $\G$ is a kimberlite.
	\end{enumerate}
\end{pro}
\begin{proof}
	Suppose $\F$ is specializing, since $\G$ is a subsheaf of $\F$ it is separated, and by formal adicness $(\G\red)^\diamond\to \G$ is injective. Pick $\Spa{R}\in \mrm{Perf}$ and a map $\Spdf{R^+}\to \F$, the basechange $X:=\G\times_\F \Spdf{R^+}$ is a formally closed subsheaf of $\Spdf{R^+}$. Reasoning as in \Cref{pro:separateduniqueform} we conclude $X=\Spdf{R^+}$ when $\Spa{R}\to \F$ factors through $\G$. This proves that $\G$ is v-formalizing and a specializing sheaf. Suppose $\F$ is a prekimberlite. 
%
	By formal adicness $(\G\red)^\diamond\to \G$ is a closed immersion. Now, $(\G\red)^\diamond\to (\F\red)^\diamond$ is also a formally closed immersion and by \Cref{lem:formallyclosedisZariski} $\G\red$ is represented by a closed subscheme of $\F$. This proves that $\G$ is a prekimberlite. Since closed immersions are partially proper the map $\F\to \Heuer{(\G\red)}$ is partially proper. Consequently, the same holds for $\F\to \Heuer{(\F\red)}$.
That $(\G, \G\cap \Di)$ is a smelted kimberlite follows from \cite[Proposition 11.20]{Et}. 
If $\F$ is a kimberlite, then $\G^{\mrm{an}}=\F^{\mrm{an}}\times_\F \G$ so $(\G,\G^\mrm{an})$ is a smelted kimberlite. The quasicompactness hypothesis on $\mrm{sp}_{\G^\mrm{an}}$ follows from that of $\mrm{sp}_{\F^\mrm{an}}$ and that closed immersions are quasicompact.
\end{proof}

\begin{pro}
	\label{pro:formaletalegiveskimberlites}
	Let $\F$ be a prekimberlite and $V\to \F\red$ a quasicompact, separated and \'etale map. The following hold:
	\begin{enumerate}
		\item If $(\F,\Di)$ is a smelted kimberlite, then $(\Tf{\F}{V},\Di\times_\F \Tf{\F}{V})$ is a smelted kimberlite.
		\item If $\F$ is a kimberlite $\Tf{\F}{V}$ is a kimberlite.
	\end{enumerate}
\end{pro}
\begin{proof}
	This follows from \Cref{pro:etalepreservesformaladic}, form \Cref{pro:preservevaluative}, and \cite[Corollary 11.28]{Et}.  
\end{proof}

\subsection{Finiteness and normality}
In this section, we discuss a finiteness and a normality hypothesis imposed on kimberlites to tame them. These notions are ad hoc, but they turned out to be useful for applications.

Let us give some motivation. Suppose $\cali{X}$ is a formal scheme topologically of finite type over $\Zp$, let $X_\eta$ denote the generic fiber considered as an adic space and let $X\red$ denote the reduced special fiber. In this situation, we have a specialization map $\mrm{sp}_{X_\eta}:|X_\eta|\to |X\red|$, and for a fixed closed point $x\in |X\red|$ we have the following chain of inclusions $|\Tup{\cali{X}}{x}|\subseteq \mrm{sp}_{X_\eta}^{-1}(x)\subseteq |X_\eta|$. These inclusions satisfy: 
\begin{enumerate}
	\item $\mrm{sp}_{X_\eta}^{-1}(x)$ is a closed subset.
	\item $|\Tup{\cali{X}}{x}|$ is the interior of $\mrm{sp}_{X_\eta}^{-1}(x)$ in $|X_\eta|$.
	\item $|\Tup{\cali{X}}{x}|$ is dense in $\mrm{sp}_{X_\eta}^{-1}(x)$
\end{enumerate}
The first two conditions generalize, by \Cref{tubularsmeltedkimberlites}, to the case of kimberlites whose closed points are constructible. Our finiteness condition is sufficient to make a kimberlite have the third property as well. As the next example shows finiteness hypothesis are necessary for this third property to hold. 
\begin{exa}
	\label{exa:valringnotcJ}
	Let $C$ be a nonarchimedean field and $C^+\subseteq C$ an open and bounded valuation subring whose rank is strictly larger than $1$. Then $\mrm{sp}_{C}$ is a homeomorphism between $\Spa{C}$ and $\mrm{Spec}(C^+/C^{\circ\circ})$. If $x$ denotes the closed point of $\mrm{Spec}(C^+/C^{\circ\circ})$ then $y=\mrm{sp}_{C}^{-1}(x)$ is the closed point in $\Spa{C}$. The interior of $\{y\}$ is empty.
\end{exa}
\begin{defi}
	\label{defi:cJ}
	We say that a locally spatial diamond $X$ is \textit{constructibly Jacobson} if the subset of rank $1$ points are dense for the constructible topology of $|X|$. We refer to them as \textit{cJ-diamonds}.
\end{defi}
\begin{pro}
	\label{pro:densenbhood}
	Suppose that $\Ki=(\F,\Di)$ is a smelted kimberlite with $\Di$ a cJ-diamond, let $S\subseteq |\F|$ a constructible subset. Then $|\Di\cap \Tf{\F}{S}|$ is dense in $\mrm{sp}_{\Ki}^{-1}(S)$.
\end{pro}
\begin{proof}
	By the proof of \Cref{pro:constgivesopen}, $|\Di\cap \Tf{\F}{S}|$ is the largest subset of $\mrm{sp}_{\Ki}^{-1}(S)$ stable under generization. Now, $\mrm{sp}_{\Ki}^{-1}(S)$ is open in the patch topology of $|\Di|$ and by assumption rank $1$ points are dense in it. Since rank $1$ points are stable under generization, they belong to $|\Di\cap \Tf{\F}{S}|$. This proves that $|\Di\cap \Tf{\F}{S}|$ is dense in $\mrm{sp}_{\Ki}^{-1}(S)$.
\end{proof}
\begin{pro}
	\label{pro:basicpropCJ}
	Let $f:X\to Y$ be a morphism of locally spatial diamonds the following hold:
\begin{enumerate}
	\item If $|f|$ is surjective and $X$ is a cJ-diamond, then $Y$ is a cJ-diamond. 
	\item If $f$ is an open immersion and $Y$ is a cJ-diamond, then $X$ is a cJ-diamond.
	\item If $f$ realizes $X$ as a quasi-pro-\'etale $\underline{J}$-torsor over $Y$ for a profinite group $J$ and $X$ is a cJ-diamond, then $Y$ is a cJ-diamond.
	\item If $f$ is \'etale and $Y$ is a cJ-diamond, then $X$ is a cJ-diamond.
\end{enumerate}
\end{pro}
\begin{proof}
	Maps of locally spatial diamonds induce spectral and generalizing maps of locally spectral spaces. The first claim follows easily from this. 
	Suppose that $Y$ is a cJ-diamond. If $f$ is an open immersion, any open in the patch topology of $X$ is also open in the patch topology of $Y$. This proves the second claim. Moreover, this allow us to argue locally, so we can assume $X$ and $Y$ are spatial.
	If $f$ is \'etale, by \cite[Lemma 11.31]{Et}, locally we can write $f$ as the composition of an open immersion and a finite \'etale map. Spaces that are finite \'etale over a fixed spatial diamond form a Galois category and using the first claim we may reduce to the case in which $f$ is Galois with finite Galois group $G$. In this way, the fourth claim follows from the third. 
	For the third claim, we prove that $f$ is an open map for the patch topology. This would finish the proof since the fibers of rank $1$ points of a quasi-pro-\'etale map are also rank $1$. 
	
	Let $J=\varprojlim_i J_i$ with $J_i$ a cofiltered family of finite groups and denote by $f_i:X_i\to Y$ the induced $J_i$-torsors. We get continuous action maps $J_i\times |X_i|^{\mrm{cons}}\to |X_i|^\mrm{cons}$. Moreover, if $S\subseteq |X_i|$ then $f_i^{-1}(f_i(S))=J_i\cdot S$. Now, the formation of the patch topology on a spectral space commutes with limits along spectral maps. This gives a continuous action map $J\times |X|^{\mrm{cons}}\to |X|^\mrm{cons}$. Let $U\subseteq X$ be open in the patch topology, then $f^{-1}(f(U))=J\cdot U$ which is also open. We can conclude since $|f|^\mrm{cons}:|X|^\mrm{cons}\to |Y|^\mrm{cons}$ is a quotient map. 
\end{proof}
We recall a theorem of Huber. His statement is stronger, but we only need this weaker form.
\begin{thm}\textup{(\cite[Theorem 4.1]{Cont})}
	\label{thm:TatealgebrascJ}
	Let $K$ be a complete nonarchimedean field, and let $A$ be a topologically of finite type $K$-algebra. Then, $\mrm{Max}(A)\subseteq \Spac{A}$ is dense for the patch topology. 
\end{thm}
\begin{cor}
	\label{cor:rigidcJ}
If $X$ is an adic space topologically of finite type over $\mrm{Spa}(K,O_K)$, with $K$ a complete nonarchimedean field over $\Zp$. Then $X^\dia$ is a cJ-diamond.
\end{cor}

\begin{exa}
	\label{lem:perfectoidball}
	The perfectoid unit ball $\bb{B}_n=\mrm{Spa}(C\langle T_1\pthr\dots T_n\pthr\rangle,O_C\langle T_1\pthr\dots T_n\pthr\rangle)$ over a perfectoid field $C$ over $\Fp$, is a cJ-diamond. Indeed, we have an equality $\mrm{Spa}(C\langle T_1\cdots T_n\rangle,O_C\langle T_1\cdots T_n\rangle)^\dia=\bb{B}_n$.
\end{exa}

\begin{defi}
	\label{defi:enoughballs}
	Let $C$ be a characteristic $p$ perfectoid field and $X$ a locally spatial diamond over $\mrm{Spa}(C,O_C)$. We say that $X$ has \textit{enough facets} over $C$ if it admits a v-cover of the form $\coprod_{i\in I} \mrm{Spd}(A_i,A_i^\circ)\to X$ where each $A_i$ is an algebra topologically of finite type over $C$.
\end{defi}

\begin{pro}
	\label{pro:enoughballsprop}
	Let $X$ and $Y$ be two locally spatial diamonds with enough facets over $C$, let $C^\sharp$ be an untilt of $C$, and $C\to C'$ a map of perfectoid fields. The following hold:
\begin{enumerate}
	\item The base change $X\times_\C \Cp$ has enough facets over $C'$.
	\item The fiber product $X\times_\C Y$ has enough facets over $C$.
	\item $X$ is a $cJ$-diamond.
	\item If $X=\mrm{Spd}(A,A^\circ)$ for a smooth and topologically of finite type $C^\sharp$-algebra $A$, then $X$ has enough facets. 
\end{enumerate}
\end{pro}
\begin{proof}
	Since being topologically of finite type is stable under products and change of the ground field the first two claims follows. The third claim follows from \Cref{thm:TatealgebrascJ} and \Cref{pro:basicpropCJ}.
	For the last claim, let $\bb{T}_{C^\sharp}^n$ denote $\mrm{Spa}(C^\sharp\langle T_1^{\pm},\dots T_n^{\pm}\rangle,O_{C^\sharp}\langle T_1^{\pm},\dots T_n^{\pm}\rangle)$, and let $\widetilde{\bb{T}}_{C^\sharp}^n=\varprojlim_{T_i\mapsto T_i^p} \bb{T}^n_{C^\sharp}$ analogously for $\bb{T}^n_C$ and $\widetilde{\bb{T}}^n_C$. For $x\in \mrm{Spa}(A,A^\circ)$ there is $U$ a neighborhood of $x$ and an \'etale map $\eta:U\to \bb{T}_{C^\sharp}^n$. Let $\widetilde{U}$ be the pullback of $\eta$ along $\widetilde{\bb{T}}_{C^\sharp}^n\to \bb{T}_{C^\sharp}^n$, we get an \'etale map $\widetilde{U}^\flat\to \widetilde{\bb{T}}^n_C$. By the invariance of the \'etale site under perfection (\cite[Lemma 15.6]{Et}) $\widetilde{U}^\flat=U'^\dia$ for an adic space $U'$ \'etale over $\bb{T}^n_C$. Now, $U'$ admits an open cover of the form $\coprod_{i\in I}\mrm{Spa}(A_i,A_i^\circ)\to U'$ with each $A_i$ topologically of finite type over $C$. This gives a cover,	$\coprod_{i\in I}\mrm{Spd}(A_i,A_i^\circ)\to \widetilde{U}^\flat \to U^\dia$.
\end{proof}
\begin{defi} 
\label{defi:kimberlite} 
Let $\G$ be a kimberlite and $\Ki=(\F, \Di)$ a smelted kimberlite.
\begin{enumerate}
	\item We say that $\Ki$ is \textit{rich} if: $\F$ is valuative, $\Di$ is a cJ-diamond, $|\F\red|$ is locally Noetherian and $\mrm{sp}_{\Di}:|\Di|\to |\F\red|$ is surjective. 
	\item We say that $\G$ is \textit{rich} if: $(\G, \G^\mrm{an})$ is rich. 
	\item If $\Ki$ is rich we say it is \textit{topologically normal} if for every closed point $x\in |\F\red|$ the tubular neighborhood $\Tup{\Di}{x}$ is connected.\footnote{This definition is motivated by \cite[Proposition 2.38]{AGLR22}, where the authors observed that normality ``could be captured" in terms of tubular neighborhoods.}
\end{enumerate}
\end{defi}

	\begin{lem}
		\label{lem:partproperpreKimb}
		If $\Ki=(\F,\Di)$ is a rich smelted kimberlite, then $\mrm{sp}_{\Di}$ is a quotient map. 
	\end{lem}
	\begin{proof}
		We can argue locally on $|\F\red|$, so we may assume $\F\red=\mrm{Spec}(A)^\diamond$. We first prove the case that $|\F\red|$ is irreducible. Let $g$ be the generic point of $|\F\red|$, let $r\in|\Di|$ with $\mrm{sp}_{\Di}(r)=g$ represented by formalizable map $f_r:\C\to \Di$. Let $C_g^\mrm{min}$ be the minimal ring of integral elements of $C$ such that $C_g^\mrm{min}/C^{\circ \circ}$ contains $A$. This defines a map $\mrm{Spa}(C,C^\mrm{min}_g)\to \Heuer{\mrm{Spec}(A)}$. By valuativity, we get a lift $\mrm{Spa}(C,C_g^\mrm{min})\to \Di$. The map $f^\mrm{min}:|\mrm{Spa}(C,C^\mrm{min})|\to |\F\red|$ is specializing, surjective and a spectral map of spectral spaces. By \Cref{cor:closedmap}, $f^\mrm{min}$ is a closed map and consequently a quotient map of topological spaces.
		The case in which $|\F\red|$ has a finite number of irreducible components is analogous. Since we are allowed to work locally and we assumed $|\F\red|$ is locally Noetherian, every point has an affine neighborhood with finitely many irreducible components.
	\end{proof}

	\begin{prop}
		\label{pro:neighborhoodspreserverich} 
		Constructible formal neighborhoods and \'etale formal neighborhoods preserve rich smelted kimberlites.
	\end{prop}
	\begin{proof}
This follows from \Cref{pro:basicpropCJ}, \Cref{pro:constgivesopen} and \Cref{pro:correctreductionHeuer}.
	\end{proof}

	The following lemma was the starting point of our theory of specialization and a key input for our work on connected components of moduli spaces of $p$-adic shtukas \cite{gleason2021geometric}.   
\begin{lem}
	\label{pro:mainproKimberlites}
	Let $\Ki=(\F,\Di)$ be a topologically normal rich smelted kimberlite. Then $\pi_0(\mrm{sp}_{\Di}):\pi_0(|\Di|)\to \pi_0(|\F\red|)$ is bijective.
\end{lem}
\begin{proof}
	Let $U,V\subseteq |\Di|$ be two non-empty closed-open subsets with $V\cup U=|\Di|$. Clearly, $\pi_0(\mrm{sp}_{\Di})$ is surjective. Suppose that $\emptyset \neq \mrm{sp}_{\F}(U)\cap \mrm{sp}_{\F}(V)$, it suffices to show that $U\cap V\neq 0$. Since $\mrm{sp}_{\Di}$ is specializing we can assume there is a closed point $x\in \mrm{sp}_{\Di}(U)\cap \mrm{sp}_{\Di}(V)$. Since $|\F\red|$ is locally Noetherian, closed points are open in the constructible topology. By \Cref{pro:densenbhood}, $\Tup{\Di}{x}$ is dense in $\mrm{sp}_{\Di}^{-1}(x)$, which implies that $\mrm{sp}_{\Di}^{-1}(x)$ is connected. Connectedness gives that $(\mrm{sp}_{\Di}^{-1}(x)\cap U)\cap (\mrm{sp}_{\Di}^{-1}(x)\cap V)\neq \emptyset$ and in particular $U\cap V\neq \emptyset$ which is what we wanted to show. 
\end{proof}

\section{The specialization map for unramified $p$-adic Beilinson--Drinfeld Grassmannians}
So far our discussion has been purely theoretical. In this section, we apply the theory to construct and study the specialization map for some $p$-adic Beilinson--Drinfeld Grassmannians \cite[Definition  20.3.1]{Ber}. 
For the rest of the section $\g$ denotes a reductive over $\Zp$ and we let $T\subseteq B\subseteq \g$ denote integrally defined maximal torus and Borel subgroups respectively. 
Let $\mu\in X_*^+(T_{\overline{\mathbb{Q}}_p})$ be a dominant cocharacter with reflex field $E\subseteq \overline{\mathbb{Q}}_p$. Let $O_E$ denote the ring of integers of $E$ and let $k_E$ denote the residue field. Let $\Grm{\Oe}$ denote the v-sheaf parametrizing $B^+_\mrm{dR}$-lattices with $\g$-structure whose relative position is bounded by $\mu$ as in \cite[Defintion 20.5.3]{Ber} and let $\GrWme{k_E}$ denote the Witt vector affine Grassmannian. 
%
%
Let $F$ be a nonarchimedean field extension of $E$. Let $O_F$ denote the ring of integers of $F$ and the residue field $k_F$, assume that $F$ is complete for the $p$-adic topology and that $k_F$ is perfect. Let $\Grm{O_F}:=\Grm{\Oe}\times_{\Spdf{O_E}}\Spdf{O_F}$ and let $\GrWme{k_F}=\GrWme{k_E}\times_{\mrm{Spec}(k_E)}\mrm{Spec}(k_F)$. 
Here is our result:
\begin{thm}
	\label{thm:GrassmannianisaKimberlite}
	$\Grm{O_F}$ is a topologically normal rich $p$-adic kimberlite with $(\Grm{O_F})\red=\GrWme{k_F}$. 
\end{thm}

\begin{rem}
	This result has partially been generalized in our collaboration \cite{AGLR22}. There, we prove that the local models for parahoric groups are rich $p$-adic kimberlites. Nevertheless, we only improve the ``normality" part of the result if we assume that $\mu$ is minuscule and outside certain cases in small characteristic.  
\end{rem}

\subsection{Twisted loop sheaves}
Fix a perfect field $k$ in characteristic $p$. Let $X=\mrm{Spec}(A)$ be a scheme of finite type over $W(k)$. In \cite{Et}, Scholze associates to $X$ two v-sheaves over $\Wkd$, which we denote $X^\dia$ and $X^\diamond$ following \cite[Definition 2.10]{AGLR22}. Here $X^\dia:\mrm{Perf_k}\to \mrm{Sets}$ assigns to $\Spa{R}$ triples $(R^\sharp,\iota,f)$ with $(R^\sharp,\iota)$ an untilt and $f\in \mathrm{Hom}_{W(k)}(A,R^\sharp)$ is a $W(k)$-algebra homomorphism. On the other hand, $X$ assigns triples $(R^\sharp,\iota,f)$ with $f\in \mathrm{Hom}_{W(k)}(A,R^{\sharp,+})$. Now, $X^\diamond\subseteq X^\dia$ is open. Both of these functors glue to a construction defined for schemes $X$ locally of finite type over $\mrm{Spec}(W(k))$. Visibly, these constructions are related to $\dia:\mrm{PreAd}_{W(k)}\to \topPerf$. Let us elaborate.

Let $X_p$ denote the $p$-adic completion of $X$. Now, $X_p$ is a $p$-adic Noetherian formal scheme that we may regard as an adic space. We have an identification $X_p^\dia=X^\diamond$. If $X=\mrm{Spec}(A)$ and $X_f\to X$ is the open $X_f=\mrm{Spec}(A[\frac{1}{f}])$ with $f\in A$, then $(X_f)_p\to X_p$ is the locus in $X_p$ where $1\leq |f|$.

The construction of $X^\dia$ is more elaborate. Given an adic space $S$ (thought of as a triple $(|S|,\cali{O}_S,\{ v_s: s\in |S|\})$ in Huber's category $\mathscr{V}$ see \cite{ForRig}), we let $S^H$ denote the topologically ringed space $(|S|,\cali{O}_S)$ that is obtained from $S$ by forgetting the last entry of data. Suppose we are given a morphism of schemes $f:X\to Y$ that is locally of finite type and a morphism $g:S^H\to Y$ of locally ringed spaces where $S$ is an adic space for which every point $s\in S$ has an affinoid open neighborhood with Noetherian ring of definition. In \cite[Proposition 3.8]{ForRig}), Huber constructs an adic space $"S\times_Y X"$ together with a map of adic spaces $p_1:"S\times_Y X"\to S$ and a map of locally ringed spaces $p_2:("S\times_Y X")^H\to X$ with the following universal property. If $T$ is an adic space, $\pi_1:T\to S$ is a map of adic spaces and $\pi_2:T^H\to X$ is a map of locally ringed spaces such that $f\circ \pi_1 = g\circ \pi_2^H$, then there is a unique map $\pi:T\to "S\times_Y X"$ such that $p_1\circ \pi=\pi_1$ and $p_2\circ \pi^H=\pi_2$. Letting $Y=\mrm{Spec}(W(k))$ and $S=\Spf{W(k)}$ we define $X^{\mrm{ad}}$ as $("S\times_Y X")$. Then, $X^\dia=(X^{\mrm{ad}})^\dia$. Moreover, if $X=\mrm{Spec}(A)$ and $X_f=\mrm{Spec}(A[\frac{1}{f}])$ for $f\in A$ we can see from the universal property that $X_f^{\mrm{ad}}$ is the open locus of $X^{\mrm{ad}}$ where $f\neq 0$.

%
\begin{pro}
	\label{pro:propernessgivessharpissharpplus}
	If $X\to \mrm{Spec}(W(k))$ is a proper map of schemes, then the natural map $X^\diamond\to X^\dia$ is an isomorphism.
\end{pro}
\begin{proof}
	It follows directly from \cite[Remark 4.6.(iv).d]{ForRig}.
\end{proof}
\begin{pro}
	\label{pro:universallysubtrusgivesvcover}
	If $X$ and $Y$ are qcqs finite type schemes over $\mrm{Spec}(W(k))$ and that $X\to Y$ is universally subtrusive as in \Cref{defi:v-top-sch}, then $X^\dia\to Y^\dia$ and $X^\diamond\to Y^\diamond$ are surjective.
\end{pro}
\begin{proof}
	Replacing $Y$ by an open cover we may assume $Y=\mrm{Spec}(A)$. By \cite[Theorem 3.12]{Ryd} we may assume that $X\to Y$ factors as $X\to Y'\to Y$ with $Y'\to Y$ a proper surjection and $X\to Y'$ a quasicompact open cover. Open covers give surjective maps so we can reduce to the proper case. By Chow's lemma \cite[Tag 0200]{Stacks}, we may assume $Y'\to Y$ is projective. Now, $Y'^\dia\to Y^\dia$ and $Y'^\diamond\to Y^\diamond$ are quasicompact. Indeed, they are the composition of a closed immersion and projection from $(\bb{P}_{W(k)}^n)^\dia\times_{\Wkd} Y^\dia$ (respectively $(\bb{P}_{W(k)}^n)^\dia\times_{\Wkd} Y^\diamond$). By \cite[Lemma 12.11]{Et}, we may check surjectivity at a topological level. Take geometric points $r:\Spa{C}\to Y^\dia$ and $s:\Spa{C}\to Y^\diamond$ given by ring maps $r^*:A\to C$ and $s^*:A\to C^+$. Since $Y'\to Y$ is proper and surjective $\mrm{Spec}(C)\times_Y Y'\to \mrm{Spec}(C)$ admits a section inducing a lift to $Y'^\dia$. Analogously, $\mrm{Spec}(C^+)\times_Y Y'\to \mrm{Spec}(C^+)$ admits a section (by the valuative criterion of properness). 
\end{proof}
\begin{pro}
	Let $X$ be locally of finite type over $W(k)$ with special fiber $X_s$. Then $(X^\dia)\red=X_s^{\mrm{perf}}=(X^\diamond)\red$.
\end{pro}
\begin{proof}

	Both identifications follow from \Cref{pro:globaltworeds}. By the construction of $X_p$ as a $p$-adic completion in the case of $X^\diamond$, and by the universal property of $X^{\mrm{ad}}$ in the case of $X^\dia$.
\end{proof}

For the rest of the section let $C$ be an algebraically closed nonarchimedean field over $k$ with ring of integers $O_C$ and residue field $k_C$. Fix a characteristic $0$ untilt $C^\sharp$ and fix $\xi\in W(O_C)$ a generator for the kernel of $W(O_C)\to O_{C^\sharp}$. The choice of untilt determines a unique map $\Spdf{O_C}\to \Zpd$.
\begin{defi}
	\label{defi:sheavesofrings}
	We denote ring sheaves $W^+(\cali{O}),B^+_\mrm{dR}(\cali{O}^\sharp),W(\cali{O}),B_\mrm{dR}(\cali{O}^\sharp):\mrm{Perf}_{\Spdf{O_C}}\to \mrm{Sets}$
	\begin{enumerate}
		\item Where $W^+(\cali{O})$ assigns to $\Spa{R}\to \Spdf{O_C}$ the ring $W(R^+)$.
		\item Where $B^+_\mrm{dR}(\cali{O}^\sharp)$ assigns to $\Spa{R}\to \Spdf{O_C}$ the ring $B^+_\mrm{dR}(R^\sharp)$.
		\item Where $W(\cali{O})$ assigns to $\Spa{R}\to \Spdf{O_C}$ the ring $W(R^+)[\frac{1}{\xi}]$.
		\item Where $B_\mrm{dR}(\cali{O}^\sharp)$ assigns to $\Spa{R}\to \Spdf{O_C}$ the ring $B_\mrm{dR}(R^\sharp):=B^+_\mrm{dR}(R^\sharp)[\frac{1}{\xi}]$.
	\end{enumerate}
\end{defi}
%
Note that we have reduction maps $\mrm{red}:W^+(\cali{O})\to \cali{O}^{\sharp,+}$ and $\mrm{red}:B^+_\mrm{dR}(\cali{O}^\sharp)\to \cali{O}^\sharp$.
\begin{defi}
	\label{defi:loopgroupsofallsorts}
	Let $H$ be a finite type affine scheme over $\bb{G}_{m,W(k)}$, and $(\Hi,\rho)$ a finite type affine scheme over $\bb{A}^1_{W(k)}$ with an isomorphism $\rho:\Hi\times_{\bb{A}^1_{W(k)}}\bb{G}_m\to H$. We associate v-sheaves over $\Spdf{O_C}$. 
	\begin{enumerate}
		\item $W^+\Hi$ assigns to $\Spa{R}$ the set of sections to $\Hi_{W(R^+)}\to \mrm{Spec}(W(R^+))$. 
		\item $WH$ assigns to $\Spa{R}$ the set of sections to $H_{W(R^+)[\frac{1}{\xi}]}\to \mrm{Spec}(W(R^+)[\frac{1}{\xi}])$. 
		\item $L^+\Hi$ assigns to $\Spa{R}$ the set of sections to $\Hi_{B_\mrm{dR}^+(R^\sharp)}\to \mrm{Spec}(B_\mrm{dR}^+(R^\sharp))$. 
		\item $LH$ assigns to $\Spa{R}$ the set of sections to $\Hi_{B_\mrm{dR}(R^\sharp)}\to \mrm{Spec}(B_\mrm{dR}(R^\sharp))$. 
	\end{enumerate}
Here the base change of $\Hi$ and $H$ are given by maps to $\bb{A}^1_{W(k)}$ and $\bb{G}_{m,W(k)}$ defined by $t\mapsto \xi$.
\end{defi}
$\rho$ induces maps $L^+\Hi\xrightarrow{\rho} LH$ and $W^+\Hi\xrightarrow{\rho} WH$. We get the following diagrams of inclusions:
\begin{center}
	\begin{tikzcd}
	& W(\cali{O}) \arrow[hookrightarrow]{rd}& &
	& WH \arrow[hookrightarrow]{rd}& \\
	W^+(\cali{O}) \arrow[hookrightarrow]{ru} \arrow[hookrightarrow]{rd}& & B_\mrm{dR}(\cali{O}^\sharp) &
	W^+\Hi \arrow[hookrightarrow]{ru}{\rho} \arrow[hookrightarrow]{rd}& & LH \\
	&  B^+_\mrm{dR}(\cali{O}^\sharp)\arrow[hookrightarrow]{ru}& &
	& L^+\Hi \arrow[hookrightarrow]{ru}{\rho}&
	\end{tikzcd}
\end{center}
Moreover, if $\ov{\Hi}$ denotes the basechange $\Hi\times_{\bb{A}^1_{W(k)}} \mrm{Spec}(W(k))$ at $t=0$ we get reduction maps $W^+\Hi\to (\ov{\Hi})^\diamond_{\Spdf{O_C}}$ and $L^+\Hi\to (\ov{\Hi})^\dia_{\Spdf{O_C}}$.
\begin{pro}
	\label{pro:smoothsurjectiveloops}
	If $\Hi$ is smooth over $\mrm{Spec}(W(k)[t])$ the reduction maps are surjective. 
\end{pro}
\begin{proof}
	We claim that the map is surjective even at the level of presheaves. The $\Hub{R}$-valued points of $(\ov{\Hi})^\dia$ and $(\ov{\Hi})^\diamond$ can be seen as maps $\mrm{Spec}(R^\sharp)\to \Hi_{B^+_\mrm{dR}(R^\sharp)}$ and $\mrm{Spec}(R^{\sharp,+})\to \Hi_{W(R^+)}$ whose composition with the projections to $\mrm{Spec}(B^+_\mrm{dR}(R^\sharp))$ and $\mrm{Spec}(W(R^+))$ are the usual closed embeddings. By smoothness of $\Hi$, for any $n\in \bb{N}$ the maps can be lifted to maps $\mrm{Spec}(B^+_\mrm{dR}(R^\sharp)/\xi^n)\to \Hi_{B^+_\mrm{dR}(R^\sharp)}$ and $\mrm{Spec}(W(R^+)/\xi^n)\to \Hi_{W(R^+)}$ respectively. Since $\Hi$ is an affine scheme and since both $B^+_\mrm{dR}(R^\sharp)$ and $W(R^+)$ are $(\xi)$-adically complete we may pass to the inverse limit by choosing compatible lifts.
\end{proof}
\begin{defi}
	We define scheme-theoretic v-sheaves $\Wred(\cali{O}),\Wred\Hi:\mrm{PCAlg}^\mrm{op}_{k_C}\to \mrm{Sets}$.
	\label{defi:reducedloopsheaves}
	\begin{enumerate}
		\item Let $\Wred(\cali{O})$ attach to $\mrm{Spec}(R)$ the ring $W(R)$. 
		\item Let $\Wred\Hi$ attach to $\mrm{Spec}(R)$ the sections to $\Hi\times_{W(k)[t]}{W(R)}\to \mrm{Spec}(W(R))$. 
	\end{enumerate}
Here the base change of $\Hi$ is given by the map to $\bb{A}^1_{W(k)}$ defined by $t\mapsto p$.
\end{defi}
\begin{rem}
	These v-sheaves are Zhu's $p$-adic jet spaces in \cite[\S 1.1.1]{Zhu}. 
\end{rem}
\begin{pro}
	\label{pro:formalloopgroupsarekimberlites}
	With the notation as above, $W^+\Hi$ is a $p$-adic kimberlite and $(W^+\Hi)\red=(\Wred\Hi)$.
\end{pro}
\begin{proof}
	$W^+(\cali{O})$ is represented by $\Spdf{O_C\langle T_n\rangle_{n\in \bb{N}}}$, by \Cref{exa:formalispre-Kimberlite}, \Cref{pro:twoadics} and \Cref{pro:tworeds} $W^+(\cali{O})$ is a $p$-adic kimberlite with $W^+(\cali{O})\red=\mrm{Spec}(k_C[T_n]_{n\in \bb{N}})$ which is $\Wred(\cali{O})$. 
	If $\Hi=\mrm{Spec}(A)$ is presented as $A=W(k)[t][\ov{x}]/I$ with $I=(f_1(t,\ov{x}),\dots,f_m(t,\ov{ x }))$. Then $W^+\Hi$ fits in the commutative diagram with Cartesian square:
	\begin{center}
		\begin{tikzcd}
			W^+\Hi \ar{r}\ar{d}& \Spdf{O_C}\ar{d}{0}\ar{dr}{id} \\
			W^+(\cali{O})^n\ar{r}{F} & W^+(\cali{O})^m\ar{r} & \Spdf{O_C}
		\end{tikzcd}
	\end{center}

	Here $F(\ov{r})=(f_1(\xi,\ov{r}),\dots,f_m(\xi,\ov{r}))$. All of these maps are formally adic, and $\Spdf{O_C}\xrightarrow{0} W^+(\cali{O})^m$ is formally closed. 
	By \Cref{lem:closedsubschemeisevident} $W^+\Hi$ is a $p$-adic kimberlite. Finally, we can basechange by $\mrm{Spec}(k_C)^\diamond\to \Spdf{O_C}$ to compute reductions. 
%
\end{proof}

\subsection{Demazure resolution}
We assume the reader has some familiarity with Bruhat--Tits theory and parahoric group schemes \cite{MR756316}. We use twisted loop sheaves to consider integrally defined Demazure resolutions. Our main observation is that the Demazure resolution can be constructed using either parahoric loop groups or what we call below ``formal parahoric loop" groups. The difference is whether one uses $B_\mrm{dR}$ or $A_{\mrm{inf}}$ coefficients. We keep the notation as above. 
\begin{enumerate}
	\item Let $H$ be a split $W(k)$-reductive group, and $T\subseteq B\subseteq H$ maximal split torus and Borel subgroups. 
	\item Let $(X^*,\Phi,X_*,\Phi^\vee)$ be the root datum associated to $(H,T)$.
	\item We let $\langle\cdot, \cdot \rangle:X^*\times X_*\to \bb{Z}$ denote the perfect pairing between roots and coroots.
	\item Let $\Phi^+$ be the set of positive roots associated to $B$.
	\item Let $N$ be the normalizer of $T$ in $H$.
	\item Let $W=N/T$ be the Weyl group of $H$.
	\item We let $\cali{A}=\cali{A}(H,T)$ denote $X_*(T)\otimes_{\bb{Z}}\bb{R}$. 
	\item We let $\Psi=\{\alpha+n\mid \alpha\in \Phi,\,n\in\bb{Z}\}$ denote the set of affine roots on $\cali{A}$. 
	\item Given a point $q\in \cali{A}$ we let $\Phi_q=\{\alpha\in \Phi\mid \alpha(q)\in \bb{Z}\}$ this is a closed sub-root system. 
	\item Let $M_{q}$ be the Levi subgroup containing $T$ with root datum given by $(X^*,\Phi_q,X_*, \Phi_q^\vee)$. 
	\item If $q\in \cali{A}$ we associate $F_q\subseteq \cali{A}$ containing $q$ and bounded by the hyperplanes defined by $\Psi$. 
	\item 
		We let $o\in \cali{A}$ be the origin and $o\in \cali{C}$ be the alcove contained in the Bruhat chamber of $B$.
	\item Let $\bb{S}$ be the reflections along the walls of $\cali{C}$, we let $W^\mrm{aff}$ the affine Weyl group generated by $\bb{S}$. 
	\item For facets $\cali{F}\subseteq \cali{C}$ let $\bb{S}_\cali{F}$ be elements of $\bb{S}$ fixing $\cali{F}$ and let $W_\cali{F}$ be the subgroup generated by $\bb{S}_\cali{F}$. 
	\item Let $\tilde{W}$ be the Iwahori--Weyl. Recall that if $\Omega_H=\pi_1(H^\mrm{der})$ then $\tilde{W}=W^\mrm{aff}\rtimes \Omega_H$. 
\end{enumerate}

Fix a point $q\in \cali{A}$. Using Bruhat--Tits theory and dilatation techniques (\cite[\S 3]{Pap-Zhu}), one can construct smooth affine algebraic groups $\Hi_{{q}}$ over $\mrm{Spec}(W(k)[t])$ and isomorphism $\rho:\Hi_{{q}}\times_{W(k)[t]} \mrm{Spec}(W(k)[t,t^{-1}])\cong H\times_{W(k)}\mrm{Spec}(W(k)[t,t^{-1}])$ satisfying the following:  
\begin{enumerate}[a)]
	\item For a DVR $V$ with uniformizer $\pi\in V$ the basechange $\Hi_{{q}}\times_{\bb{A}^1_{W(k)}} \mrm{Spec}(V)$ along $t\mapsto \pi$ is the parahoric group scheme of $q\in \cali{A}(H,V[\frac{1}{\pi}])=\cali{A}$. 
	\item For $\alpha\in \Phi$ there are closed subgroups $\cali{U}^q_\alpha\subseteq \Hi_{{q}}$ and $\cali{T}\subseteq \Hi_{{q}}$ with $\bb{G}_a\cong \cali{U}^q_\alpha$ and $\bb{G}_m^n\cong \cali{T}$. These extend the root groups $U_\alpha\subseteq H$ and the torus $T\subseteq H$. 
	\item There is an open cell decomposition: $\cali{V}_{q}:=\prod_{\alpha\in \Phi^{-}} \cali{U}^q_\alpha \times \cali{T} \times \prod_{\alpha\in \Phi^{+}} \cali{U}^q_\alpha\subseteq \Hi_{{q}}$. It is an open immersion surjecting onto a fiberwise Zariski-dense neighborhood of the identity section. 
	\item The group multiplication map $\cali{V}_{q}\times\cali{V}_q\xrightarrow{\mu}  \Hi_{{q}}$ is smooth and surjective.
	\item The basechange $\ov{\Hi}_{{q}}:=\Hi_{{q}}\times_{\bb{A}^1_{W(k)}}\mrm{Spec}(W(k))$ along $t=0$ supports a split reductive quotient $(\ov{\Hi})_{{q}}^\mrm{Red}$ over $W(k)$ with root datum identified with $(X^*,\Phi_{q},X_*, \Phi_{q}^\vee)$ so that $M_{q}=(\ov{\Hi})_{{q}}^\mrm{Red}$.
	\item If $\alpha\in \Phi_{q}$, then $\ov{\cali{U}}^q_{\alpha,t=0}\to (\ov{\Hi})_{{q}}^\mrm{Red}$ identifies with the root group of $(\ov{\Hi})_{{q}}^\mrm{Red}$ corresponding to $\alpha$. 
\item If $\alpha \in \Phi\setminus \Phi_{q}$ then the composition $\ov{\cali{U}}^q_{\alpha,t=0}\to (\ov{\Hi})_{{q}}^\mrm{Red}$ factors through the identity section. 
\item We have a commutative diagram of open cell decompositions: 

	\begin{center}
		\begin{tikzcd}
			\prod_{\alpha\in \Phi\setminus \Phi_q} \ov{\cali{U}}_\alpha  \ar{r} \ar{d}{\cong}	 &	\prod_{\alpha\in \Phi^{-}} \ov{\cali{U}}_\alpha \times \ov{\cali{T}} \times \prod_{\alpha\in \Phi^{+}} \ov{\cali{U}}_\alpha \ar{r}{\pi}	\ar{d}{\mu} & \prod_{\alpha\in \Phi_q^{-}} \ov{\cali{U}}_\alpha \times \ov{\cali{T}} \times \prod_{\alpha\in \Phi_q^{+}} \ov{\cali{U}}_\alpha \ar{d}{\mu} \\
			Ker(m) \ar{r}	& 	\ov{\Hi}_q\ar{r}{m}		& \ov{\Hi}^\mrm{Red}_q
		\end{tikzcd}
	\end{center} 
\end{enumerate}

Also, given $q_1,q_2\in \cali{A}$ with $F_{q_1}\subseteq F_{q_2}$ we get a map groups $f:\Hi_{{q_2}}\to \Hi_{{q_1}}$ satisfying:
\begin{enumerate}
	\item[i)] $\rho_1\circ f=\rho_2$ over $\mrm{Spec}(W(k)[t,t^{-1}])$.
	\item[j)] $\ov{\Hi}_{{q_2}}\to (\ov{\Hi}_{{q_1}})^\mrm{Red}$ surjects onto the parabolic subgroup associated to the sub-root system $\Phi{q_1.q_2}:=\{\alpha\in \Phi_{q_1}\mid\lfloor\alpha(q_2)\rfloor=\alpha(q_1)\}$. 
	\item[k)] The kernel of $\ov{\Hi}_{{q_2}}\to (\ov{\Hi}_{{q_1}})^\mrm{Red}$ is fiberwise a vector group. 
\end{enumerate}

\begin{defi}
	\label{defi:familiesofloopgroups}
	Let $q\in \cali{A}$. Since $\Hi_q$ and $H$ are defined over $\mrm{Spec}(W(k)[t])$ and $\mrm{Spec}(W(k)[t,t^{-1}])$ we can use the construction of \Cref{defi:loopgroupsofallsorts}. We call $LH$ the loop group, we call $L^+\Hi_q$ the parahoric loop group and we call $W^+\Hi_q$ the ``formal parahoric loop group".
\end{defi}
Notice that we have injective maps of v-sheaves $W^+\Hi_q\subseteq L^+\Hi_q\overset{\rho}{\subseteq} LH$.

\begin{pro}
	\label{pro:reductionforparahorloopgrps}
	With the notation as above, for any point $q\in \cali{A}$ we have surjective morphisms of v-sheaves in groups:
	$L^+\Hi_q\to [(\overline{\Hi})_q^\mrm{Red}]^\dia=M_{q}^\dia$ and $W^+\Hi_q\to [(\ov{\Hi}_q)^\mrm{Red}]^\diamond=M_{q}^\diamond$.
\end{pro}
\begin{proof}
	This is a direct consequence of \Cref{pro:smoothsurjectiveloops} and \Cref{pro:universallysubtrusgivesvcover} since $\ov{\Hi}\to \ov{\Hi}^\mrm{Red}$ is smooth. 
\end{proof}
We let $L^u\Hi_q$ and $W^u\Hi_q$ denote respectively the kernels of the morphisms of \Cref{pro:reductionforparahorloopgrps} above. 
\begin{pro}
	\label{pro:reductionforparahorloopgrpstwoparahorics}
If $q_1,q_2\in \cali{A}$ are such that $F_{q_1}\subseteq F_{q_2}$, then we get inclusions of v-sheaves in groups:
	$L^u\Hi_{q_1}\subseteq L^u\Hi_{q_2}\subseteq L^+\Hi_{q_2}\subseteq L^+\Hi_{q_1}$ and
	$W^u\Hi_{q_1}\subseteq W^u\Hi_{q_2}\subseteq W^+\Hi_{q_2}\subseteq W^+\Hi_{q_1}$
	Moreover, the map from $L^+\Hi_{q_2}$ to $M_{q_1}^\dia$ surjects onto $P_{\Phi_{q_1,q_2}}^\dia\subseteq M_{q_1}^\dia$. Analogously, $W^+\Hi_{q_2}$ surjects onto $P_{\Phi_{q_1,q_2}}^\diamond\subseteq M_{q_1}^\diamond$. 
\end{pro}
\begin{proof}

	We deal with the parahoric loop group case, the other is analogous. Recall that $f:\Hi_{q_2}\to \Hi_{q_1}$ satisfies $\rho_1\circ f=\rho_2$. Functoriality of $L^+$, gives $L^+\Hi_{q_2}\to L^+\Hi_{q_1}\to LH$. Since $L^+\Hi_{q_2}\to LH$ is injective, then $L^+\Hi_{q_2}\to L^+\Hi_{q_1}$ also is. 
	The map of affine schemes $\ov{\Hi}_{q_2}\to P_{\Phi_{q_1,q_2}}$ is faithfully flat of finite presentation. By \Cref{pro:smoothsurjectiveloops} and \Cref{pro:universallysubtrusgivesvcover}, the map $L^+\Hi_{q_2}\to (\ov{\Hi}_{q_2})^\dia\to P_{\Phi_{q_1,q_2}}^\dia$ is surjective. 
	Finally, any map $g:\mrm{Spec}(B^+_\mrm{dR}(R^\sharp))\to \Hi_{{q_1},B^+_\mrm{dR}(R^\sharp)}$ whose induced map $\mrm{Spec}(R^\sharp)\to M_q$ factors through the identity lifts to a map $\mrm{Spec}(B^+_\mrm{dR}(R^\sharp))\to \Hi_{{q_2},B^+_\mrm{dR}(R^\sharp)}$. 
	Indeed, $\mrm{Spec}(R^\sharp)\to \ov{\Hi}_{q_1}$ is in the open cell $\ov{\cali{V}}_{q_1}$, so $g$ has the form $g=(\prod_{\alpha\in \Phi^{-}} u_\alpha(g))\cdot t(g)\cdot (\prod_{\alpha\in \Phi^{+}} u_\alpha(g))$ with $t(g)$ and $\{u_\alpha(g)\}_{\alpha\in \Phi_{q_1}}$ reducing to the identity modulo $\xi$.  
	By inspection, each of this elements lifts uniquely to $\cali{V}_{q_2}$. For $\cali{T}$ and $\cali{U}_\alpha$ with $\alpha\in (\Phi\setminus\Phi_{q_1})\cup \Phi_{q_1,q_2}$, $f$ is an isomorphism. 
	For $\alpha \in \Phi_{q_1}\setminus \Phi_{q_1,q_2}$ we may, after making some choices, write $\cali{U}^{q_1}_\alpha$ as $\mrm{Spec}(W(k)[t,u])$ and $\cali{U}^{q_2}_\alpha$ as $\mrm{Spec}(W(k)[t,\frac{u}{t}])$ such that $f$ is the natural map. Now, $u_\alpha(g)^*:W(k)[t,u]\to B^+_\mrm{dR}(R^\sharp)$ with $t\mapsto \xi$ extends (uniquely) to $W(k)[t,\frac{u}{t}]\to B^+_\mrm{dR}(R^\sharp)$ if $\xi$ divides the image of $u$, but this happens whenever $u_\alpha(g)$ reduces to identity.
\end{proof}

\begin{pro}
	\label{pro:flagvarsforparahorloopgrps}
	Let $q_1,q_2\in \cali{A}$ such that $F_{q_1}\subseteq F_{q_2}$, then we have identifications $L^+\Hi_{q_1}/L^+\Hi_{q_2}=W^+\Hi_{q_1}/W^+\Hi_{q_2}=({\mrm{Fl}}_{q_1,q_2,O_C})^{\dia}$, where $\mrm{Fl}_{q_1,q_2}$ denotes the flag variety. 
\end{pro}
\begin{proof}
We compute: $L^+\Hi_{p_1}/L^+\Hi_{p_2} = (L^+\Hi_{p_1}/L^u\Hi_{p_1})/(L^+\Hi_{p_2}/L^u\Hi_{p_1}) = M_{p_1}^\dia/P_{\Phi_{p_1,p_2}}^\dia= \mrm{Fl}_{p_1,p_2}^\dia$ and $W^+\Hi_{p_1}/W^+\Hi_{p_2} = (W^+\Hi_{p_1}/W^u\Hi_{p_1})/(W^+\Hi_{p_2}/W^u\Hi_{p_1})  = M_{p_1}^\diamond/P_{\Phi_{p_1,p_2}}^\diamond = \mrm{Fl}_{p_1,p_2}^\diamond$.
Finally, \Cref{pro:propernessgivessharpissharpplus} gives  $\mrm{Fl}_{p_1,p_2}^\dia=\mrm{Fl}_{p_1,p_2}^\diamond$.
\end{proof}

\begin{lem}
	\label{lem:proetalelocallytrivial}
	Fix $q_1,q_2\in \cali{A}$ with $F_{q_1}\subseteq F_{q_2}$. Let $\F$ be a locally spatial diamond with a map $\F\to \Spdf{O_C}$ and let $\Spa{R}$ be an affinoid perfectoid with a map $\Spa{R}\to \mrm{Fl}_{q_1,q_2}^\dia$. 
	\begin{enumerate}
		\item  The map $L^+\Hi_{q_1}\times_{\Spdf{O_C}} \F\to \mrm{Fl}_{q_1,q_2}^\dia\times_{\Spdf{O_C}}\F$ admits pro-\'etale locally a section. 
		\item The map $L^+\Hi_{q_1}\times_{\mrm{Fl}_{q_1,q_2}^\dia}\Spa{R} \to \Spa{R}$ admits analytic locally a section.
	\end{enumerate}
\end{lem}
\begin{proof}
	We may reduce the first claim to the second one by \cite[Proposition 11.24]{Et}. Indeed, by \Cref{pro:flagvarsforparahorloopgrps} the map is a $L^+\Hi_{q_2}$-torsor. 
	Let us prove the second claim. Let $\mathfrak{obs}\in H_v^1(\Spa{R}, L^+\Hi_{q_2})$ be the obstruction to triviality. We prove $\frak{obs}=e$ after analytic localization. Consider the sequences of maps: $(L^+\Hi_{q_1})\to M_{q_1}^\dia\to \mrm{Fl}_{q_1,q_2}^\dia$ and 
	  $e\to L^u\Hi_{q_1}\to L^+\Hi_{q_2}\to P_{\Phi_{q_1,q_2}}^\dia\to e$.
	  Now, $M_{q_1}^\dia\to \mrm{Fl}_{q_1,q_2}^\dia$ is a $(P_{\Phi_{q_1,q_2}}^\dia)$-torsor with obstruction lying in $H^1_v(\mrm{Fl}_{q_1,q_2}^\dia,P_{\Phi_{q_1,q_2}}^\dia)$. Since $M_{q_1}\to \mrm{Fl}_{q_1,q_2}$ admits Zariski locally a section, replacing $\Spa{R}$ by an analytic cover, we may assume $\mathfrak{obs}=e$ in $H^1_v(\Spa{R}, P_{\Phi_{q_1,q_2}}^\dia)$. 
	  We claim $H_v^1(\Spa{R},L^u\Hi_{q_1})=\{e\}$. Recall the exact sequence $e\to \mrm{\mrm{Ker}}\left( L^+\Hi_{q_1}\to (\ov{\Hi}_{q_1})^\dia\right)\to L^u\Hi_{q_1}\to \mrm{\mrm{Ker}}\left( \ov{\Hi}_{q_1}^\dia \to [(\ov{\Hi}_{q_1})^\mrm{Red}]^\dia \right)\to e$,
	  we prove vanishing of $H^1(\Spa{R},-)$ on the other two groups. 

	  For the left group consider the family of groups $\{L^{u,n}\}_{n=1}^\infty$ filtering $L^{u,1}:=\mrm{Ker}\left( L^+\Hi_{q_1}\to \ov{\Hi}_{q_1}^\dia\right)$. Define them as: 
	  $L^{u,n}\Hub{R}:=\mrm{Ker}\left(\Hi_{q_1,B^+_\mrm{dR}(R^\sharp)}(B^+_\mrm{dR}(R^\sharp))\to \Hi_{q_1,B^+_\mrm{dR}(R^\sharp)}(B^+_\mrm{dR}(R^\sharp)/\xi^n)\right)$.
	  Now, after sheafification $L^{u,n}/L^{u,n+1}=\mrm{Ker}\left(\Hi_{q_1,B^+_\mrm{dR}(R^\sharp)}(B^+_\mrm{dR}(R^\sharp)/\xi^{n+1})\to \Hi_{q_1,B^+_\mrm{dR}(R^\sharp)}(B^+_\mrm{dR}(R^\sharp)/\xi^{n})\right)$. 
	  Since $\mrm{Spec}(B^+_\mrm{dR}(R^\sharp)/\xi^{n})\to \mrm{Spec}(B^+_\mrm{dR}(R^\sharp)/\xi^{n+1})$ is a first order nilpotent thickening, deformation theory gives:
	  $$L^{u,n}/L^{u,n+1}=\mrm{Hom}(e^*\Omega^1_{\Hi_{q_1}}\otimes_{W(k)[t]}B^+_\mrm{dR}(R^\sharp),(\xi^n\cdot B^+_\mrm{dR}(R^\sharp)/\xi^{n+1}))=\mrm{Hom}(e^*\Omega^1_{\Hi_{q_1}}\otimes R^\sharp, R^\sharp).$$
	  Now, $e^*\Omega^1_{\Hi_{q_1}/W(k)[t]}$ is a finite free $W(k)[t]$-module 
	  so $L^{u,n}/L^{u,n+1}\cong(\cali{O}^\sharp)^k$ for some $k$. By \cite[Proposition 8.8]{Et}, $H^1_v(\Spa{R},\cali{O}^\sharp)=0$. This shows that $H_v^1(\Spa{R}, L^{u,1}/L^{u,n})=\{e\}$ for all $n$. Since $L^{u,1}=\varprojlim L^{u,1}/L^{u,n}$ and the transition maps are surjective at the level of presheaves we conclude that $H^1_v\left(\Spa{R},L^{u,1}\right)=\{e\}$.

	  For the right group, we may use \cite[Proposition 8.8]{Et}again since $\mrm{Ker}(\ov{\Hi}_{q_1}\to (\ov{\Hi})^\mrm{Red}_{q_1})$ is a vector group over $W(k)$. 
\end{proof}

We can now consider families of Demazure varieties, see \cite[Definition VI.5.6]{FS21}.

\begin{defi}
	\label{defi:Demazurevsheaves}
	 Let $\sigma_r:=\{r_i\}_{1\leq i\leq n}$ and $\sigma_q:=\{q_i\}_{1\leq i\leq n}$ be a pair of sequences of points in $\cali{A}$ such that $F_{r_{i}},F_{r_{i+1}}\subseteq F_{q_{i}}$, and let $\sigma:=(\sigma_r,\sigma_q)$. To such $\sigma$ we associate a v-sheaf given by the contracted group product:
	 $D(\sigma)=L^+\Hi_{r_1}\overset{L^+\Hi_{q_1}}{\times_{\Spdf{O_C}}}L^+\Hi_{r_2}\overset{L^+\Hi_{q_2}}{\times_{\Spdf{O_C}}}\dots \overset{L^+\Hi_{q_{n-1}}}{\times_{\Spdf{O_C}}}L^+\Hi_{r_n}/L^+\Hi_{q_n}$.
\end{defi}
\begin{pro}
	\label{pro:Demazurevsheavesaresmoothproper}
	The map of v-sheaves $D(\sigma)\to \Spdf{O_C}$ is representable in spatial diamonds, proper and $\ell$-cohomologically smooth for any $\ell\neq p$. 
\end{pro}
\begin{proof}
	Let $\sigma$ be as above and let $\sigma'=(\{r_i\}_{1\leq i\leq n-1},\{q_i\}_{1\leq i\leq n-1})$ be the subsequence of the first $n-1$ points of $\sigma$. We have a projection morphism of v-sheaves $f:D(\sigma)\to D(\sigma')$ given by forgetting the last entry corresponding to $r_n$. One can inductively show that this map satisfies all of the properties in the hypothesis, given that it is a $(\mrm{Fl}_{r_n,q_n})^\dia$-fibration.
\end{proof}
\begin{pro}
	\label{pro:twoloopgroupssamedemazure}
	The map $\pi:W^+\Hi_{r_1}{\times_{\Spdf{O_C}}}\cdots {\times_{\Spdf{O_C}}}W^+\Hi_{r_n}\to D(\sigma)$ coming from $W^+\Hi_{r_i}\subseteq L^+\Hi_{r_i}$ is surjective, it induces an isomorphism 
	$\iota:D(\sigma)\cong W^+\Hi_{r_1}\overset{_{W^+\Hi_{q_1}}}{\times_{\Spdf{O_C}}}\dots \overset{_{W^+\Hi_{q_{n-1}}}}{\times_{\Spdf{O_C}}}W^+\Hi_{r_n}/W^+\Hi_{q_n}$.
	Consequently, $D(\sigma)$ is v-formalizing.
\end{pro}
\begin{proof}
Consider the following basechange diagram: 
\begin{equation}
\footnotesize
	\label{diag:formalizingpoints}
\begin{tikzcd}
	W^+\Hi_{r_1}{\times_{\Spdf{O_C}}}\dots \times_{\Spdf{O_C}}(L^+\Hi_{r_n}/L^+\Hi_{q_n}) \arrow{r} \arrow{d}& L^+\Hi_{r_1}{\times_{\Spdf{O_C}}}\dots \times_{\Spdf{O_C}}(L^+\Hi_{r_n}/L^+\Hi_{q_n}) \arrow{r} \arrow{d} &D(\sigma) \arrow{d} \\
	W^+\Hi_{r_1}{\times_{\Spdf{O_C}}}\dots \times_{\Spdf{O_C}} W^+\Hi_{r_{n-1}} \arrow{r} & L^+\Hi_{r_1}{\times_{\Spdf{O_C}}}\dots \times_{\Spdf{O_C}} L^+\Hi_{r_{n-1}} \arrow{r} & D(\sigma')
\end{tikzcd}
\normalsize
\end{equation}
\Cref{pro:flagvarsforparahorloopgrps} gives $W^+\Hi_{r_n}/W^+\Hi_{q_n}=L^+\Hi_{r_n}/L^+\Hi_{q_n}$ and the surjectivity by induction. Assume that we have an identification: 
$\iota':D(\sigma')\cong W^+\Hi_{r_1}\overset{W^+\Hi_{q_1}}{\times_{\Spdf{O_C}}}\dots \overset{W^+\Hi_{q_{n-2}}}{\times_{\Spdf{O_C}}} W^+\Hi_{r_{n-1}}/W^+\Hi_{q_{n-1}}$.
Since $W^+\Hi_{q_k}\subseteq L^+\Hi_{q_k}$, the map $\iota$ is defined and surjective, we prove that it is also injective. Let $[g_1]$ and $[g_2]$ be two maps 
$[g_1],[g_2]:\Spa{R}\to W^+\Hi_{r_1}\overset{W^+\Hi_{q_1}}{\times_{\Spdf{O_C}}}\dots \overset{W^+\Hi_{q_{n-1}}}{\times_{\Spdf{O_C}}} W^+\Hi_{r_{n}}/W^+\Hi_{q_{n}}$ with $\iota([g_1])=\iota([g_2])$. By inductive hypothesis, we may v-locally find representatives $g_1$ and $g_2$ of $[g_1]$ and $[g_2]$ of the form $g_i=(g^1_i,\dots,g^n_i)$ such that $g^j_1=g^j_2$ for $j\in \{1,\dots,n-1\}$. 
Since $[g_1]$ and $[g_2]$ get identified in $D(\sigma)$, v-locally $g_1$ and $g_2$ are on the same $L^+\Hi_{q_1}\times_{\Spdf{O_C}} \dots\times_{\Spdf{O_C}} L^+\Hi_{q_n}$-orbit. Since $g_1$ and $g_2$ share all of their entries except possibly the last, $g^n_1$ and $g^n_2$ are in the same $L^+\Hi_{q_n}$-orbit. Since $g^n_1,g^n_2\in W^+\Hi_{r_n}$ and $W^+\Hi_{q_n}=W^+\Hi_{r_n}\cap L^+\Hi_{q_n}$ they are in the same $W^+\Hi_{q_n}$-orbit, so $[g_1]=[g_2]$. 

Finally, by \Cref{pro:formalloopgroupsarekimberlites} each $W^+\Hi_{r_i}$ is formalizing, \Cref{pro:v-formalizing+separatedstableprod} implies the same for the product, and since $D(\sigma)$ is the quotient of a v-formalizing sheaf it is also v-formalizing.
\end{proof}

\begin{pro}
	\label{pro:specialfiberschemetheoretic}
	The map $D(\sigma)\to{\Spdf{O_C}}$ is formally adic. Moreover, $D(\sigma)\red$ is represented by a qcqs scheme that is perfectly finitely presented and proper over $\mrm{Spec}(k_C)$ \cite[Proposition 3.11, Definition 3.14]{Witt}. 
\end{pro}
\begin{proof}
	In any Grothendieck topos, pullback commutes with finite limits and colimits. \Cref{pro:formalloopgroupsarekimberlites} gives: 
	$D(\sigma)\times_{\Spdf{O_C}}\mrm{Spec}(k_C)^\diamond =  (\Wred\Hi_{r_1})^\diamond\overset{(\Wred\Hi_{q_1})^\diamond}{\times_{\mrm{Spec}(k_C)^\diamond}}\dots \overset{(\Wred\Hi_{q_{n-1}})^\diamond}{\times_{\mrm{Spec}(k_C)^\diamond}}(\Wred\Hi_{r_n})^\diamond/(\Wred\Hi_{q_n})^\diamond$. 
	Now, this is
	$\left(\Wred\Hi_{p_1}\overset{\Wred\Hi_{q_1}}{\times_{k_C}}\dots \overset{\Wred\Hi_{q_{n-1}}}{\times_{k_C}}\Wred\Hi_{p_n}/\Wred\Hi_{q_n}\right)^\diamond$ 
since $(\cdot)^\diamond$ is a left adjoint and commutes with colimits.
\Cref{lem:basechangerepresentimpliesadic} proves that $D(\sigma)\to \Spdf{O_C}$ is formally adic and that $D(\sigma)\red=D(\sigma)\times_{\Spdf{O_C}}\mrm{Spec}(k_C)^\diamond$. The structural properties of $D(\sigma)\red$ are well known, see \cite{Zhu}, \cite{Witt}. 
\end{proof}

\begin{pro}
	\label{pro:DemazureCJ}
	$D(\sigma)\times_{\Spdf{O_C}}\mrm{Spa}({C,O_{C}})$ has enough facets over $C$. 
\end{pro}
\begin{proof}
	We prove this by induction. Let $\sigma=(\{r_i\}_{1\leq i\leq n},\{q_i\}_{1\leq i\leq n})$ and $\sigma'=(\{r_i\}_{1\leq i\leq n-1},\{q_i\}_{1\leq i\leq n-1})$. 
	Suppose that $D(\sigma')_{C}$ has enough facets, let $S:=\coprod_{i\in I} \mrm{Spd}(B_i,B_i^\circ)$ and let $f:S\to D(\sigma')_{C}$ be a cover as in \Cref{defi:enoughballs}. 
	Let $\F=D(\sigma)_{C}\times_{D(\sigma')_{C}} S$, we prove that $\F$ has a enough facets. By \Cref{lem:proetalelocallytrivial}, $\F\to S$ is analytically locally a trivial $(\mrm{Fl}_{r_n,q_n,C^\sharp})^\dia$-fibration. 
	We may conclude by \Cref{pro:enoughballsprop}. 
\end{proof}
\begin{prop}
	\label{thm:DemazureKimberlite}
	For any $\sigma$ as in \Cref{defi:Demazurevsheaves}, $D(\sigma)$ is a topologically normal rich $p$-adic kimberlite. 
\end{prop}
\begin{proof}
	By \Cref{pro:twoloopgroupssamedemazure}, \Cref{pro:formallypadicgivesformaldiagona} and \Cref{pro:specialfiberschemetheoretic} it is a $p$-adic prekimberlite, and by \Cref{pro:Demazurevsheavesaresmoothproper} together with \Cref{pro:partiallyproperisvaluative} it is a valuative kimberlite whose analytic locus is qcqs. 
	Since $D(\sigma)\red$ is a proper perfectly finitely presented scheme over $k_C$, $|D(\sigma)\red|$ is Noetherian. Also, $D(\sigma)^{\mrm{an}}$ coincides with $D(\sigma)\times_{\Spdf{O_C}}\C$ and by \Cref{pro:DemazureCJ} and  this is a qcqs cJ-diamond. By \Cref{lem:surjectivchecclosedpoints}, we may check surjectivity of $\mrm{sp}_{D(\sigma)^\mrm{an}}$ on closed points. At this point, it suffices to prove that if $x\in |D(\sigma)\red|$ is a closed point, then $\Tup{D(\sigma)}{x}$ is non-empty and connected. This follows inductively from \Cref{lem:tubularofbundles}. Indeed, \Cref{connected} holds by induction over the maps $D(\sigma)\to D(\sigma')$ and \Cref{specializingcondition} follows from the diagram \ref{diag:formalizingpoints} since each of the $W^+\Hi_{r_i}$ is formalizing and basechanges along maps that factor through $W^+\Hi_{r_1}{\times_{\Spdf{O_C}}}\dots \times_{\Spdf{O_C}} W^+\Hi_{r_{n-1}}$ will give a trivial bundle.
\end{proof}

	\begin{lem}
		\label{lem:surjectivchecclosedpoints}
		Let $C$ be a characteristic zero nonarchimedean algebraically closed field, and let $k=O_C/\frak{m}_C$. Let $(\F,\F_C)$ be a smelted kimberlite over $\Spdf{O_C}$ with $\F_C=\F\times_{\Spdf{O_C}} \mrm{Spa}(C,O_C)$. Consider $\F_{O_{C'}}:=\F\times_{\Spdf{O_C}} \Spdf{O_{C'}}$ ranging over algebraically closed nonarchimedean field extension $C'/C$. Suppose that for every $C'$ and every closed point $x\in |\F_{O_{C'}}\red|$ the tubular neighborhood 
		$\Tup{(\F_{C'})}{x}$ is non-empty. Then $\mrm{sp}_{\F_\eta}$ is a surjection.
	\end{lem}
	\begin{proof}
		Given a point in $x\in |\F\red|$ we can find a field extension of perfect fields $K/k$ for which $\F\red\times_k \mrm{Spec}(K)$ has a section $y:\mrm{Spec}(K)\to \F\red\times_k \mrm{Spec}(K)$ mapping to $x$ under $\F\red\times_k \mrm{Spec}(K)\to \F\red$. Since $\F$ is formally separated, $\F\red\times_k \mrm{Spec}(K)$ is also separated and sections to the structure map define closed points. We can construct a nonarchimedean field $C'$ with $C\subseteq C'$ and $W(k)[\frac{1}{p}]\subseteq W(K)[\frac{1}{p}]\subseteq C'$. We get a map $\F_{O_{C'}}\to \F$, and in $|\F_{O_{C'}}\red|$ there is a closed point $y$ mapping to $x$. Any point $r\in |\F_{{C'}}|$ with $\mrm{sp}_{{\F_{C'}}}(r)=y$ maps to a point whose image under the specialization map is $x$. 
	\end{proof}

\begin{lem}
	\label{lem:tubularofbundles}
	Let $f:\F\to \G$ be a proper, $\ell$-cohomologically smooth map of $p$-adic kimberlites over $\Spdf{O_C}$. 
	Suppose $\G\red$ and $\F\red$ are perfectly finitely presented over $\mrm{Spec}({k_C})$ \cite[Definition 3.10]{Witt}.
	Let $X\to \mrm{Spec}(O_{C^\sharp})$ be a smooth projective scheme and suppose that $f$ is $X^\diamond$-bundle. 
	Suppose that:
\begin{enumerate}
	\item For any nonarchimedean field $C'$ extension $C$ and a map $t:{\mrm{Spd}(C',O_{C'})}\to \G$ there is an extension $C''$ of $C'$ such that $\F\times_\G {\mrm{Spd}(O_{C''})}$ is isomorphic to $X^\diamond\times_{\Spdf{O_C}} \Spdf{O_{C''}}$. \label{specializingcondition}
	\item For any closed point $x\in |\G\red|$ the tubular neighborhood $\Tup{\G}{x}$ is non-empty and connected. \label{connected}
\end{enumerate}
	Then, for any closed point $y\in |\F\red|$ the tubular neighborhood $\Tup{\F}{y}$ is also non-empty and connected. 
\end{lem}
\begin{proof}
	Take a closed point $y\in |\F\red|$ with $x=f(y)$ and consider the map $f:\Tup{\F}{y}\to \Tup{\G}{x}$. Assume for now that for all maps $\Cp\to \Tup{\G}{x}$ the base change $\Tup{\F}{y}\times_{\Tup{\G}{x}} \Cp$ is non-empty and connected, we finish the proof under this assumption. 
	The map $|\Tup{\F}{y}|\to |\Tup{\G}{x}|$ is specializing, and by assumption surjective on rank $1$ points. 
	Let $U$ and $V$ non-empty open and closed subsets with $U\cup V=|\Tup{\F}{y}|$. Since $f$ is open and closed \cite[Proposition 23.11]{Et}, $f(U)\cup f(V)=|\Tup{\G}{x}|$ and $f(U)$ and $f(V)$ meet at a rank $1$ point. 

	Let us prove our assumption holds, take a map $t:\Cp\to \Tup{\G}{x}$. After, replacing $\Cp$ by a v-cover we can assume $\G$ formalizes $t$ and that $t$ has the base change property of \Cref{specializingcondition}. We get a pair of Cartesian diagrams, the right being obtained by taking the reduction of the left: 
	\footnotesize
\begin{center}
\begin{tikzcd}
	\Tf{\F}{y}\times_\G \Spdf{O_{C'}}\ar{r} \ar{d}	&X^\diamond\times_{\Spdf{O_C}}\Spdf{O_{C'}} \arrow{r} \arrow{d} & \Spdf{O_{C'}} \arrow{d} & 
	Z\ar{r}\ar{d}&	X\times \mrm{Spec}(k') \arrow{r} \arrow{d} & \mrm{Spec}(k') \arrow{d} \\
	\Tf{\F}{y}\ar{r} & \F \arrow{r} &\G & 
	\mrm{Spec}(k(y))\ar{r}{y} &\F\red \arrow{r} &\G\red
\end{tikzcd}
\end{center}
\normalsize
Since $\F\red\to \G\red$ is perfectly finitely presented and $k$ is algebraically closed, $k=k(y)=k(x)$ and the composition $y:\mrm{Spec}(k)\to \G\red$ is a closed immersion. Consequently $Z\to \mrm{Spec}(k')$ is an isomorphism. This gives $\Tf{\F}{y}\times_\G \Spdf{O_{C'}}=X^\diamond_{O_{C'}}\times_\F \Tf{\F}{y}= \Tf{({X^\diamond_{O_{C'}}})}{Z}$ by \Cref{pro:pullbackformal}. But $Z\to X\times \mrm{Spec}(k')$ is a closed point, so $\Tup{{(X^\diamond_{O_{C'}}})}{Z}$ is isomorphic to an open unit ball $\bb{B}_{<1}^n$ over $C'^\sharp$, proving the assumption. 
\end{proof}

We keep the notation from the beginning of the previous subsection and we restrict our attention to parahoric loop groups associated to points contained in our chosen alcove $\cali{C}$. Given $s_j\in \bb{S}$ we denote by $L^+\Hi_{s_j}$ the parahoric loop group associated to the wall $F_{s_j}$ in $\cali{C}$ corresponding to the reflection $s_j$. For a point $r\in \cali{C}$ we let $J_r\subseteq \bb{S}$ denote the set $\{s_j\mid r\in F_{s_j}\}$. We will denote by $L^+B$ the parahoric loop group associated to $\cali{C}$.

By functoriality of $L(-)$ we have loop group versions of the Weyl and Iwahori--Weyl groups by the formula $LW:=LN/LT$ and $L\tilde{W}:=LN/L^+\cali{T}$. They fit in an exact sequence:
$e\to LT/L^+\cali{T}\to L\tilde{W}\to LW\to e$.
A direct computation shows that $LW=L(N/T)=\underline{W}\times \Spdf{O_C}$, that $LT/L^+\cali{T}=\underline{X_*(T)}\times \Spdf{O_C}$ and that $L\tilde{W}=\underline{\tilde{W}}\times \Spdf{O_C}$.
Since $H$ is a split reductive group over $W(k)[t,t^{-1}]$, for any element $w\in W$ we can find a section $n_w:\mrm{Spec}(W(k)[t,t^{-1}])\to N$ whose projection to $W$ is $w$ \cite[Corollary 5.1.11]{Conrad}. This allow us to define a similar section $n_w:\Spdf{O_C}\to LN\subseteq LH$. Also for any $\mu\in X_*(T)$ and any $\Spa{R}\to \Spdf{O_C}$ we can consider the element $\xi^\mu\in T(B_\mrm{dR}(R^\sharp))$. This is functorial and defines a section $\Spdf{O_C}\to LT$ mapping to $\mu\in \underline{X_*(T)}\times \Spdf{O_C}$. In particular, for any element $\tilde{w}\in \tilde{W}$ there is a section $n_{\tilde{w}}:\Spdf{O_C}\to LN$ projecting to $\tilde{w}$ in $L\tilde{W}$. We can use $n_{\tilde{w}}$ to construct an automorphism $n_{\tilde{w}}:\GruH{{O_C}}\to \GruH{{O_C}}$ with $n_{\tilde{w}}(x\cdot L^+H):= n_{\tilde{w}}\cdot x\cdot L^+H$. We will use this discussion in the proof of \Cref{thm:GrassmannianisaKimberlite}.

\begin{lem}
	\label{pro:productmorphismconnectedfibers}
	Let $\sigma=(\sigma_r,\sigma_q)$ with $\sigma$ as in the previous subsection except that we require $\sigma_r,\sigma_q\subseteq \cali{C}$. Suppose that $L^+\Hi_{q_n}=L^+\Hi_{r_n}=L^+H$ then the multiplication map $\mu:D(\sigma)\to \GruH{O_C}=LH/L^+H$ has geometrically connected fibers.
\end{lem}
\begin{proof}	
	The proof is combinatorial and follows the classical case. The key geometric inputs are as follows, the basechange of $D(\sigma)\to \Spdf{O_C}$ by geometric points are proper spatial diamonds, rank $1$ points are always dense for any spatial diamond and the group of rank $1$ geometric points of a parahoric loop group coincide with the “parabolic subgroups” of a Tits-systems (or $BN$-pair). These two observations together with \cite[Lemma 12.11]{Et} reduces the proof to the classical combinatorial case. 
	We omit the details.	
\end{proof}	

We can now prove \Cref{thm:GrassmannianisaKimberlite}.

\begin{proof}[Proof (of \Cref{thm:GrassmannianisaKimberlite})]
	Observe that $\Grm{O_F}\to \Spdf{O_F}$ is formally adic. Indeed, $\Grm{O_F}\times_ {\Spdf{O_F}} \mrm{Spec}(k_F)^\diamond$ is $(\GrWme{k_F})^\dia$ and since $\GrWme{k_F}$ is proper this is $(\GrWme{k_F})^\diamond$ we may conclude by \Cref{lem:basechangerepresentimpliesadic}. 
	This gives that $(\Grm{O_F})^\mrm{red}=\GrWme{k_F}$, which is represented by a scheme \cite[Theorem 8.3]{Witt}, that the adjunction map $((\Grm{O_F})^\mrm{red})^\diamond\to \Grm{O_F}$ is a closed immersion and that $(\Grm{O_F})^{\mrm{an}}=\Grm{O_F}\times_{\Spdf{O_F}}\Spdf{F}$, which is represented by a spatial diamond by \cite[Proposition 20.4.5]{Ber}.
Also, $\Grm{O_F}$ is separated by \cite[Proposition 20.5.4]{Ber}, and by \Cref{pro:formallypadicgivesformaldiagona} it is formally separated. 
To prove that $\Grm{O_F}$ is a $p$-adic kimberlite we need to prove it is v-formalizing. To prove that it is rich it suffices to prove $\Grm{O_F}\times_{\Spdf{O_F}}\Spdf{F}$ has enough facets and that the tubular neighborhoods at closed points are non-empty. These can be checked after basechange to $\Spdf{O_C}$ for $C/F$ a completed algebraic closure. 
	
	 Suppose for now that $F=C$. 
	 In this case $\g\times_{\Zp} W(k_F)$ is a split reductive group, and since $\Grm{{O_C}}$ only depends on $\g_{W(k_F)}$, we may assume $\g=H$ with $H$ split reductive. Furthermore, using the discussion of the beginning of the section one can reduce to the case in which $H$ is semisimple and simply connected. In this case $\tilde{W}=W^\mrm{aff}$.
	Recall that we have inclusions $X_*^+(T)\subseteq X_*(T)\subseteq \tilde{W}$ so we may think of $\mu$ as an element of the Iwahori--Weyl group. By definition, $\GrumH{{O_C}}\Hub{R}$ consists of those elements in $\GruH{{O_C}}\Hub{R}$ satisfying that for any geometric point $q:\mrm{Spa}(C',C'^{,+})\to \Spa{R}$ the type of $q$, $\mu_q$, is in the double coset  
	$H(B^+_\mrm{dR}(C'^\sharp))\backslash H(B_\mrm{dR}(C'^\sharp))/H(B^+_\mrm{dR}(C'^\sharp))=X_*^+(T)=W_o\backslash W^\mrm{aff}/W_o$
	satisfies $\mu_q\leq \mu$ in the Bruhat order. Now, given $w\in \tilde{W}$ we can consider the subsheaf $\Gruw{{O_C}}\subseteq \Gru{{O_C}}$ with the similar property on a geometric point using instead the double coset 
	$B(B^+_\mrm{dR}(C'^\sharp))\backslash H(B_\mrm{dR}(C'^\sharp))/H(B^+_\mrm{dR}(C'^\sharp))=W^\mrm{aff}/W_o$. The projection map $\pi:W^\mrm{aff}/W_o\to W_o\backslash W^\mrm{aff}/W_o$ is order preserving and $\pi^{-1}(\mu)$ has a unique element $[w_\mu]$ of largest length, it has the property that $v\leq w_\mu$ if and only if $\pi(v)\leq \mu$. In particular, we have an equality of sheaves $\GrueH{{O_C}}{w_\mu}=\GrumH{{O_C}}$. We prove that for $w\in W^\mrm{aff}$ the v-sheaf $\GruwH{{O_C}}$ satisfies the conclusions of the theorem. 

	Find a reduced expression for $w=s_{j_1}\dots s_{j_n}$ and consider $D(w):=L^+H_{s_{j_1}}\overset{L^+B}{\times_{\Spdf{O_C}}}\dots \overset{L^+B}{\times_{\Spdf{O_C}}}L^+H_{s_{j_n}}/L^+H$.
	The multiplication map $m:D(w)\to \GruH{{O_C}}$ factors through $\GruwH{{O_C}}$ and surjects onto it. This implies $\GruH{{O_C}}$ is a kimberlite. \Cref{pro:partiallyproperisvaluative}, \Cref{pro:basicpropCJ} and \Cref{pro:spectralmapSch-Spatial} imply it is rich.
	\Cref{thm:DemazureKimberlite} and \Cref{pro:productmorphismconnectedfibers} combined with \Cref{lem:quotientofKimberliteisKimberlite} allow us to conclude in this case.

	Finally, let us deal with the case $F\neq C$. Let ${F'}$ the completed maximal unramified subextension of $F$ in $C$. We have surjective maps of v-sheaves:
	$\Grm{{O_C}}\to \Grm{O_{{F'}}}\to \Grm{O_F}$.
	We may argue as above to prove $\Grm{O_F}$ and $\Grm{O_{{F'}}}$ are rich kimberlites. By \Cref{pro:pullbackformal}, $\Grm{O_{{F'}}}$ has connected tubular neighborhoods since $(\Grm{{O_C}})\red=(\Grm{O_{{F'}}})\red$. On the other hand, $\Grum{O_{{F'}}}\to \Grm{O_F}$ is a $\underline{\pi_1^{\mrm{f\acute{e}t}}}(\mrm{Spec}(O_F))$-torsor and for any closed point $x\in |(\Grm{O_F})\red|$ the action of $\pi^\mrm{f\acute{e}t}_1(\mrm{Spec}(O_F))$ permutes transitively the closed points $y\in |(\Grm{O_{{F'}}})\red|$ over $x$. The action permutes transitively the tubular neighborhoods associated to $y$ which proves that the tubular neighborhood over $x$ is also connected.
\end{proof}
\begin{lem}
	\label{lem:quotientofKimberliteisKimberlite}  
	Let $f:\F\to\G$ be a surjective map of rich $p$-adic kimberlites over $\Spdf{O_F}$, such that $\G\red$ and $\F\red$ are perfectly of finite type over $\mrm{Spec}(k_F)$. Suppose $f\red$ has geometrically connected fibers and that $\F$ is topologically normal. Then $\G$ is topologically normal.  
\end{lem}
\begin{proof}

	Pick a closed point $x\in |\G\red|$, by \Cref{pro:pullbackformal} $\Tup{\G}{x}\times_{\G}\F=\Tup{\F}{S}$ with $S=(f\red)^{-1}(x)$. By \Cref{pro:mainproKimberlites} and \Cref{pro:neighborhoodspreserverich}, $\Tup{\F}{S}$ is connected. Since $f$ is surjective $\Tup{\G}{x}$ is also connected.
\end{proof}

\bibliography{biblio.bib}
\bibliographystyle{amsplain}


\end{document}